\documentclass[12pt]{amsart}
\usepackage{graphicx}
\usepackage{amsmath}
\usepackage{amsfonts}
\usepackage{amssymb}

\usepackage{setspace}
\usepackage{datetime}
\usepackage[colorlinks,
            linkcolor=black,
            anchorcolor=blue,
            citecolor=black]{hyperref}
\usepackage{geometry}
\allowdisplaybreaks[4]

\newtheorem{thm}{Theorem}[section]
\newtheorem{cor}[thm]{Corollary}

\newtheorem{lem}[thm]{Lemma}

\theoremstyle{definition}
\newtheorem{defn}[thm]{Definition}
\theoremstyle{remark}
\newtheorem{rem}[thm]{Remark}
\theoremstyle{conclusion}

\theoremstyle{conjecture}
\newtheorem{conj}[thm]{Conjecture}
\numberwithin{equation}{section}

\newcommand{\R}{\mathbb{R}}

\newcommand{\N}{\mathbb{N}}

\newcommand{\be}{\begin{equation}}
\newcommand{\ee}{\end{equation}}

\begin{document}
\title[$D^{1,p}$-critical quasi-linear Schr\"{o}dinger-Hartree equation with $-\Delta_{p}$]{Radial symmetry and sharp asymptotic behaviors of nonnegative solutions to $D^{1,p}$-critical quasi-linear static Schr\"{o}dinger-Hartree equation involving $p$-Laplacian $-\Delta_{p}$}

\author{Wei Dai, Yafei Li, Zhao Liu}

\address{School of Mathematical Sciences, Beihang University (BUAA), Beijing 100191, P. R. China, and Key Laboratory of Mathematics, Informatics and Behavioral Semantics, Ministry of Education, Beijing 100191, P. R. China}
\email{weidai@buaa.edu.cn}

\address{School of Mathematical Sciences, Beihang University (BUAA), Beijing 100191, P. R. China}
\email{yafeili@buaa.edu.cn}

\address{School of Mathematics and Computer Science, Jiangxi Science and Technology Normal University, Nanchang 330038, P. R. China}
\email{liuzhao@mail.bnu.edu.cn}

\thanks{Wei Dai is supported by the NNSF of China (No. 12222102), the National Science and Technology Major Project (2022ZD0116401), the NNSF of China (No. 11971049) and the Fundamental Research Funds for the Central Universities. Yafei Li is supported by the National Science and Technology Major Project (2022ZD0116401) and the Fundamental Research Funds for the Central Universities. Zhao Liu is supported by the NNSF of China (No. 12261041) and the Natural Science Foundation of Jiangxi Province (No. 20232ACB211002).}

\begin{abstract}
In this paper, we mainly consider nonnegative weak solution to the $D^{1,p}(\R^{N})$-critical quasi-linear static Schr\"{o}dinger-Hartree equation with $p$-Laplacian $-\Delta_{p}$ and nonlocal nonlinearity:
\begin{align*}
-\Delta_p u =\left(|x|^{-2p}\ast |u|^{p}\right)|u|^{p-2}u \qquad &\mbox{in} \,\, \mathbb{R}^N,
\end{align*}
where $1<p<\frac{N}{2}$, $N\geq3$ and $u\in D^{1,p}(\R^N)$. Being different to the $D^{1,p}(\R^{N})$-critical local nonlinear term $u^{p^{\star}-1}$ with $p^{\star}:=\frac{Np}{N-p}$ investigated in \cite{CFR,LDSMLMSB,GV,Ou,BS16,VJ16} etc., since the nonlocal convolution $|x|^{-2p}*u^p$ appears in the Hartree type nonlinearity, it is impossible for us to use the scaling arguments and the Doubling Lemma as in \cite{VJ16} to get preliminary estimates on upper bounds of asymptotic behaviors for any positive solutions $u$. Moreover, it is also quite difficult to obtain the boundedness of the quasi-norm $\|u \|_{L^{s,\infty}(\R^N)}$ and hence derive the sharp estimates on upper bounds of asymptotic behaviors from the preliminary estimates as in \cite{VJ16}. Fortunately, by showing a better preliminary estimates on upper bounds of asymptotic behaviors through the De Giorgi-Moser-Nash iteration method and combining the result from \cite{XCL}, we are able to overcome these difficulties and establish regularity and the sharp estimates on both upper and lower bounds of asymptotic behaviors for any positive solution $u$ to more general equation $-\Delta_p u=V(x)u^{p-1}$ with $V\in L^{\frac{N}{p}}(\mathbb{R}^{N})$. Then, by using the arguments from \cite{BS16,VJ16}, we can deduce the sharp estimates on both upper and lower bounds for the decay rate of $|\nabla u|$. Finally, as a consequence, we can apply the method of moving planes to prove that all the nontrivial nonnegative solutions are radially symmetric and strictly decreasing about some point $x_0\in\R^N$. Our results will lead to a complete classification of nonnegative solution provided that the uniqueness of radially symmetric solution holds. The radial symmetry and sharp asymptotic estimates for more general nonlocal quasi-linear equations were also included.
\end{abstract}

\maketitle {\small {\bf Keywords:} Quasilinear elliptic equations with $p$-Laplacian;  Nonlocal Hartree type nonlinearity;  Radial symmetry;  Sharp estimates on asymptotic behaviors;  The method of moving planes.\\

}

\section{Introduction}
\subsection{Background and setting of the problem}
In this paper, we are mainly concerned with nonnegative $D^{1,p}(\R^{N})$-weak solution to the following $D^{1,p}(\R^{N})$-critical quasi-linear static Schr\"{o}dinger-Hartree equation with $p$-Laplacian $-\Delta_{p}$ and nonlocal nonlinearity:
\begin{align}\label{eq1.1}
\left\{ \begin{array}{ll} \displaystyle
-\Delta_p u =\left(|x|^{-2p}\ast u^{p}\right)u^{p-1}  \quad\,\,\,\,&\mbox{in}\,\, \R^N, \\ \\
u \in D^{1,p}(\R^N),\quad\,\,\,\,\,\,u\geq0 \qquad\,\,\,\,&  \mbox{in}\,\, \R^N,
\end{array}
\right.\hspace{1cm}
\end{align}
where the full range of $p$ is $1<p<\frac{N}{2}$, $N\geq3$, $V(x):=\left(|x|^{-2p}\ast u^{p}\right)(x):=\int_{\R^N} \frac{u^p(y)}{|x-y|^{2p}}\mathrm{d}y$ and $\Delta_p u:=div (|\nabla u|^{p-2} \nabla u)$ is the usual $p$-Laplace operator. The function space
$$ D^{1,p}(\R^N):=\left\{ u\in L^{p^{\star}}(\R^N) \, \bigg| \, \int_{\R^N} |\nabla u|^p \mathrm{d}x <+\infty \right\}$$
is the completion of $C_0^\infty$ with respect to the norm $\| u \|:=\| \nabla u \|_{L^p(\R^N)}$. For $1<p<2$ or $p>2$, the $p$-Laplace operator $-\Delta_p$ is singular elliptic or degenerate elliptic, respectively. Moreover, the Hartree type nonlinearity in equation \eqref{eq1.1} is nonlocal. Thus it would be an quite interesting and challenging problem to study the sharp asymptotic estimates, radial symmetry, strictly radial monotonicity and classification of solutions to \eqref{eq1.1}.

\smallskip

The nonlocal quasi-linear equation \eqref{eq1.1} is $D^{1,p}(\R^{N})$-critical in the sense that both \eqref{eq1.1} and the $D^{1,p}$-norm $\| \nabla u \|_{L^p(\R^N)}$ are invariant under the scaling $u\mapsto u_{\lambda}(\cdot):=\lambda^{\frac{N-p}{p}}u(\lambda\cdot)$. Equation \eqref{eq1.1} is also invariant under arbitrary rotations and translations. For \eqref{eq1.1}, the finite $D^{1,p}(\R^N)$-energy assumption $u\in D^{1,p}(\R^N)$ is necessary to guarantee $V(x):=|x|^{-2p}\ast u^{p}\in L^{\frac{N}{p}}(\mathbb{R}^{N})$ and hence $V(x)$ is finite and well-defined almost everywhere. In fact, by Hardy-Littlewood-Sobolev inequality (see \cite{FL1,FL2,Lieb,Lions2}), one has $V(x):=\left(\frac{1}{|x|^{2p}}\ast |u|^{p}\right)(x)\in L^{\frac{N}{p}}(\mathbb{R}^{N})$, more precisely,
\begin{equation}\label{conv}
  \|V\|_{L^{\frac{N}{p}}(\mathbb{R}^{N})}\leq C\|u\|^{p}_{L^{p^{\star}}(\mathbb{R}^{N})}, \qquad \forall \,\, u\in D^{1,p}(\mathbb{R}^{N}).
\end{equation}
Thus it follows from H\"{o}lder's inequality and Sobolev embedding inequality that, for any $u\in D^{1,p}(\mathbb{R}^{N})$,
$$\int_{\mathbb{R}^{N}}\left(\frac{1}{|x|^{2p}}\ast |u|^{p}\right)(x) \cdot |u|^{p}(x)\mathrm{d}x\leq \|V\|_{L^{\frac{N}{p}}(\R^N)}\|u\|^{p}_{L^{p^{\star}}(\R^N)}\leq C\|\nabla u\|^{2p}_{L^{p}(\R^N)},$$
that is,
\begin{equation}\label{HLS}
  \|u\|_{V}:=\left(\int_{\mathbb{R}^{N}}\int_{\mathbb{R}^{N}}\frac{|u|^{p}(x)|u|^{p}(y)}{|x-y|^{2p}}\mathrm{d}x\mathrm{d}y\right)^{\frac{1}{2p}}\leq C_{N,p}\| \nabla u \|_{L^p(\R^N)},
\end{equation}
where $C_{N,p}$ is the best constant for the Hardy-Littlewood-Sobolev inequality \eqref{HLS}, which can be characterized by the following minimization problem
\begin{equation}\label{mini}
  \frac{1}{C_{N,p}}:=\inf\left\{\| \nabla u \|_{L^p(\R^N)}\,\Big|\,u\in D^{1,p}(\mathbb{R}^{N}), \,\, \|u\|_{V}=1\right\}.
\end{equation}
The $D^{1,p}(\mathbb{R}^{N})$-critical nonlocal quasi-linear equation \eqref{eq1.1} is the Euler-Lagrange equation for the minimization problem \eqref{mini}.

\smallskip

The general Choquard-Pekar equations use the limiting embedding inequality of the type \eqref{HLS} (c.f. e.g. \cite{LE76,LPL80,Lions1}). By using the concentration-compactness Lemma I.1 in \cite{Lions1} with $|u_{n}|^{p^{\star}}$ replaced by $\left(|x|^{-2p}\ast |u_{n}|^p\right)|u_{n}|^p$, Lions \cite{Lions1} proved that the minimization problem \eqref{mini} has a minimum (c.f. iv) in page 169 of \cite{Lions1}), and hence the Hardy-Littlewood-Sobolev inequality \eqref{HLS} is sharp in the sense that the best constant $C_{N,p}$ can be attained by some (positive) extremal function $u\in D^{1,p}(\mathbb{R}^{N})$. As a consequence, we have the following existence result on positive weak solution to the Euler-Lagrange equation \eqref{eq1.1}.
\begin{thm}[Existence of positive solutions, \cite{Lions1}]\label{thm0}
Every minimizing sequence $\{u_{n}\}_{n\geq 1}\subset D^{1,p}(\mathbb{R}^{N})$ of the minimization problem \eqref{mini} is relatively compact in $D^{1,p}(\mathbb{R}^{N})$ up to a translation and delation, i.e., there exist sequences $\{y_{n}\}_{n\geq 1}\subset\mathbb{R}^{N}$ and $\{\sigma_{n}\}_{n\geq 1}\subset(0,+\infty)$ such that the new minimizing sequence $\widetilde{u_{n}}:=\sigma_{n}^{-\frac{N-p}{p}}u\left(\frac{\cdot-y_{n}}{\sigma_{n}}\right)$ ($\forall \, n\geq1$) is relatively compact in $D^{1,p}(\mathbb{R}^{N})$ for $1<p<\frac{N}{2}$. In particular, there exists a minimum in $D^{1,p}(\mathbb{R}^{N})$ of the minimization problem \eqref{mini}, i.e., the best constant $C_{N,p}$ in inequality \eqref{HLS} can be attained by some (positive) extremal function $u\in D^{1,p}(\mathbb{R}^{N})$. Consequently, the nonlocal quasi-linear Hartree equation \eqref{eq1.1} possesses (at least) a positive weak solution $u\in D^{1,p}(\mathbb{R}^{N})$.
\end{thm}

\smallskip



In the semi-linear elliptic case $p=2$, \eqref{eq1.1} becomes the following $H^{1}$-energy-critical nonlinear Hartree equation
\begin{align}\label{ciHeq}
-\Delta u =\left(|x|^{-4}\ast u^{2}\right)u  \quad\,\,\,\,&\mbox{in}\,\, \R^N,
\end{align}
where $N \geq 5$. Equation \eqref{eq1.1} is the quasi-linear counterpart of \eqref{ciHeq}. PDEs of the type \eqref{ciHeq} arise in the Hartree-Fock theory of the nonlinear Schr\"{o}dinger equations (see \cite{LS}). The solution $u$ to problem \eqref{ciHeq} is also a ground state or a stationary solution to the following focusing energy-critical dynamic Schr\"{o}dinger-Hartree equation (see e.g. \cite{LMZ,MXZ3})
\begin{equation}\label{Hartree}
i\partial_{t}u+\Delta u=-\left(|x|^{-4}\ast u^{2}\right)u, \qquad (t,x)\in\mathbb{R}\times\mathbb{R}^{N}.
\end{equation}
The Schr\"{o}dinger-Hartree equations have many interesting applications in the quantum theory of large systems of non-relativistic bosonic atoms and molecules (see, e.g. \cite{FL}) and have been quite intensively studied, c.f. e.g. \cite{LMZ,MXZ3} and the references therein, in which the ground state solution can be regarded as a crucial criterion or threshold for global well-posedness and scattering in the focusing case. Therefore, the classification of solutions to \eqref{ciHeq} plays an important and fundamental role in the study of the focusing dynamic Schr\"{o}dinger-Hartree equations \eqref{Hartree}. In \cite{Liu}, Liu classified positive solutions to the energy-critical Hartree equation \eqref{ciHeq}. For classification results and other related quantitative and qualitative properties of solutions to energy-critical Hartree type equations like \eqref{ciHeq} involving fractional, second, higher and arbitrary order Laplacians $(-\Delta)^{s}$ ($0<s<+\infty$) and Choquard equations, please refer to \cite{AY,CD,CDZ,DFHQW,DFQ,DL,DLQ,DQ,GHPS,Lei,LE76,LPL80,MZ,MS} and the references therein.

\smallskip

In the quasi-linear case (i.e. $1<p<\frac{N}{2}$, $p \neq 2$), we can not deduce the comparison principles from the maximum principles directly and apply the the Kelvin transforms any more. Moreover, since the $p$-Laplace operator is singular or degenerate elliptic in the case $1<p<2$ or $p>2$ respectively, there is no uniform form of the strong comparison principle for $p$-Laplace operator, so we need to discuss the cases $1<p<2$ and $p>2$ separately. Consequently, due to the lack of general comparison principles for quasi-linear operators and the fact that a Kelvin type transformation is not available, the problem \eqref{eq1.1} becomes much more difficult and challenging.

\subsection{Main results}
In this paper, we will overcome all the difficulties caused by the quasi-linear $p$-Laplace operator $-\Delta_{p}$ and the nonlocal Hartree type nonlinear interaction, and prove sharp asymptotic estimates, the radial symmetry and strictly radial monotonicity of positive solution to \eqref{eq1.1}. These results will leads to a complete classification of nonnegative weak solutions to \eqref{eq1.1}. For this purpose, by \eqref{conv}, we will first generalize the equation \eqref{eq1.1} and consider the following generalized equation
\begin{equation}\label{geq}
  -\Delta_p u=V(x)|u|^{p-2}u \qquad \text{in} \,\, \mathbb{R}^{N}
\end{equation}
with $0\leq V\in L^{\frac{N}{p}}(\mathbb{R}^{N})$, where $1<p<N$. Equation \eqref{eq1.1} is a typical special case of \eqref{geq} with $V:=|x|^{-2p}\ast |u|^{p}$.

\smallskip

\begin{defn}[$D^{1,p}(\R^{N})$-weak solution]\label{weak}
A $D^{1,p}(\R^{N})$-weak solution of the generalized equation \eqref{geq} is a function $u\in D^{1,p}(\R^N)$ such that
$$ \int_{\R^N} |\nabla u|^{p-2} \langle \nabla u, \nabla\psi\rangle \mathrm{d}x = \int_{\R^N}V(x)|u|^{p-2}(x)u(x)\psi(x)\mathrm{d}x, \,\,\,\,\quad \forall \,\,\psi \in C_0^\infty(\R^N).$$
\end{defn}

\smallskip

For any nonnegative weak solution $u$ of the generalized equation \eqref{geq}, by exploiting the De Giorgi-Moser-Nash iteration technique (c.f. \cite{TMOSER}) as in \cite{JSLB}, \cite[Lemma B.3]{CDPSYS} and \cite{Trud}, we prove in Lemmas \ref{lm.8} and \ref{lm.5} that $u\in L^\infty(\R^N)$. Next, using the standard $C^{1,\alpha}$ estimates (see \cite{ED83,KM,PT}), we deduce that $u\in C_{loc}^{1,\alpha}(\R^N)$ for some $0<\alpha<\min\{1,\frac{1}{p-1}\}$. Moreover, it follows from the strong maximum principle (see Lemma \ref{hopf}) that any nonnegative nontrivial solution of \eqref{eq1.1} is actually strictly positive provided that $V\in L^{\infty}_{loc}(\mathbb{R}^{N})$.

\smallskip

For $1<p<N$, if $V=u^{p^{\star}-p}$ with $p^{\star}=\frac{Np}{N-p}$ and nonnegative $u\in D^{1,p}(\R^N)$, the generalized equation \eqref{geq} becomes the following $D^{1,p}(\mathbb{R}^{N})$-critical quasi-linear equations with local nonlinearity:
\begin{align}\label{criticeqution}
-\Delta_p u = u^{p^{\star}-1},\,\,\,\,\,\,u\geq0 \qquad \text{in} \,\, \mathbb{R}^{N}.
\end{align}
Equation \eqref{eq1.1} is the nonlocal counterpart of \eqref{criticeqution}. Serrin \cite{S} and Serrin and Zou \cite{SJZH02}, among other things, proved regularity results for general quasi-linear equations and sharp lower bound on the asymptotic estimate of super-$p$-harmonic functions (see Lemma \ref{estimate2}). Guedda and Veron \cite{GV} proved the uniqueness of radially symmetrical positive solution $U(x)=U(r)$ with $r=|x|$ to \eqref{criticeqution} satisfying $U(0)=b>0$ and $U'(0)=0$, and hence reduced the classification of solutions to the radial symmetry of solutions. For radial symmetry of positive $D^{1,p}(\R^N)$-weak solutions to $D^{1,p}(\mathbb{R}^{N})$-critical quasi-linear equations with local nonlinearity of type \eqref{criticeqution} via the method of moving planes, c.f. \cite{LD,DP,LDSMLMSB,DLPFRM,LDMR,LDBS04,LPLDHT,OSV,BS16,VJ16} and the references therein.

\smallskip

In the singular elliptic case $1<p<2$, under the $C^{1}\cap W^{1,p}$-regularity assumptions on the positive solutions and several assumptions on the (local) nonlinear term, the radial symmetry and strictly radial monotonicity of the positive solution was obtained in \cite{DLPFRM,LDMR} by using the moving plane technique. Later, in \cite{LDSMLMSB}, Damascelli, Merch\'{a}n, Montoro and Sciunzi refined the moving plane method used in \cite{DLPFRM,LDMR} and derived the radial symmetry of positive $D^{1,p}(\R^N)$-weak solutions to \eqref{criticeqution} without regularity assumptions on solutions. However, they still required in \cite{LDSMLMSB} that the nonlinear term is locally Lipschitz continuous, i.e., $p^{\star} \geq 2$. Subsequently, by applying scaling arguments and the doubling Lemma (see \cite{PPPQPS}), V\'{e}tois \cite{VJ16} first derived a preliminary estimate on the decay rete of positive $D^{1,p}(\R^N)$-weak solution $u$ to \eqref{criticeqution} (i.e., $u(x)\leq \frac{C}{|x|^{\frac{N-p}{p}}}$ for $|x|$ large), combining this preliminary estimate with scaling arguments, the Harnack-type inequalities (c.f. \cite{LDBS,JSLB}), regularity estimates and the boundedness estimate of the quasi-norm $L^{s,\infty}(\mathbb{R}^{N})$ with $s=\frac{N(p-1)}{N-p}$ (Lemma 2.2 in \cite{VJ16}) implied sharp asymptotic estimates of the positive $D^{1,p}(\R^N)$-weak solution $u$ (i.e., $u\sim\left(1+|x|^{\frac{N-p}{p-1}}\right)^{-1}$) and the sharp upper bound on decay rate of $|\nabla u|$ (i.e., $|\nabla u|\leq \frac{C}{|x|^{\frac{N-1}{p-1}}}$ for $|x|$ large) for $1<p<N$. These sharp asymptotic estimates combined with the results in \cite{DLPFRM,LDMR} successfully extended the radial symmetry results for \eqref{criticeqution} in \cite{LDSMLMSB} to the full singular elliptic range $1<p<2$.

\smallskip

For equation \eqref{criticeqution} in the degenerate elliptic case $p>2$, Sciunzi \cite{BS16} developed a new technique based on scaling arguments, the study of limiting profile at infinity, the uniqueness up to multipliers of $p$-harmonic maps in $\mathbb{R}^{N}\setminus\{0\}$ and the sharp asymptotic estimates on $u$ proved in \cite{VJ16}, and derived the sharp lower bound on decay rate of $|\nabla u|$ (i.e., $|\nabla u|\geq \frac{c}{|x|^{\frac{N-1}{p-1}}}$ for $|x|$ large). Thanks to the sharp asymptotic estimates on $u$ and $|\nabla u|$, Sciunzi \cite{BS16} finally proved the classification of positive $D^{1,p}(\R^N)$-weak solution $u$ to \eqref{criticeqution} via the moving planes technique in conjunction with the weighted Poincar\'{e} type inequality and Hardy's inequality and so on.

\smallskip

Ciraolo, Figalli and Roncoroni \cite{CFR} proved classification of positive $D^{1,p}(\R^N)$-weak solutions to $D^{1,p}(\R^N)$-critical anisotropic $p$-Laplacian equations of type \eqref{criticeqution} in convex cones. Under some energy growth conditions or suitable control of the solutions at $\infty$ in some cases, Catino, Monticelli and Roncoroni \cite{CMR} proved classification results on general positive solution $u$ to \eqref{criticeqution} (possibly having infinite $D^{1,p}(\R^N)$-energy) for $1<p<N$. For $\frac{N+1}{3}\leq p<N$, Ou \cite{Ou} classified general positive solution $u$ to \eqref{criticeqution} (possibly having infinite energy). One should note that, for our nonlocal quasi-linear equation \eqref{eq1.1}, the finite $D^{1,p}(\R^N)$-energy assumption $u\in D^{1,p}(\R^N)$ is natural and necessary, since we need it to guarantee that $V(x):=|x|^{-2p}\ast u^{p}\in L^{\frac{N}{p}}(\mathbb{R}^{N})$ and hence $V(x)$ is finite and well-defined almost everywhere. For more literatures on quantitative and qualitative properties of solutions for quasi-linear equations involving $p$-Laplacians with local nonlinearities, please c.f. \cite{A,AY,B,CDPSYS,CV,LD,DFSV,DP,LDBS04,LDBS,DYZ,ED83,Di,DPZZ,FSV,GL,LPLDHT,Lieberman,OSV,BSR14,JSLB,SJZH02,Tei,PT,JLV,ZL} and the references therein.

\medskip

In order to prove the radial symmetry and strictly radial monotonicity of solutions, one of the key ingredients is to prove the sharp asymptotic estimates for positive $D^{1,p}(\R^N)$-weak solution $u$ and $|\nabla u|$. However, being different to the $D^{1,p}(\R^{N})$-critical local nonlinear term $u^{p^{\star}-1}$ with $p^{\star}:=\frac{Np}{N-p}$ investigated in \cite{CFR,LDSMLMSB,GV,Ou,BS16,VJ16} etc., due to the nonlocal feature of the convolution $|x|^{-2p}*u^p$ in the Hartree type nonlinearity in \eqref{eq1.1}, it is impossible for us to use the scaling arguments and the Doubling Lemma as in \cite{VJ16} to get preliminary estimates on upper bounds of asymptotic behaviors for any positive solutions $u$. Moreover, it is also quite difficult to obtain the boundedness estimate of the quasi-norm $\|u \|_{L^{s,\infty}(\R^N)}$ with $s=\frac{N(p-1)}{N-p}$ as in Lemma 2.2 of \cite{VJ16} and hence derive the sharp estimates on upper bounds of asymptotic behaviors from the preliminary estimates as in \cite{VJ16}. Fortunately, by showing a better preliminary estimates on upper bounds of asymptotic behaviors through the De Giorgi-Moser-Nash iteration method (i.e., $u(x) \leq\frac{C}{1+|x|^{\frac{N-p}{p}+\tau}}$ for some $\tau>0$, see Lemmas \ref{lm.8} and \ref{lm.5}) and combining Theorem 1.5 in \cite{XCL} and Theorem 2.3 in \cite{SJZH02}, we are able to overcome these difficulties and establish regularity and the sharp estimates on both upper and lower bounds of asymptotic behaviors for any positive solution $u$ to the generalized equation \eqref{geq}. Then, using the arguments from \cite{BS16,VJ16}, by scaling arguments, the regularity results in DiBenedetto \cite{ED83} and Tolksdorf \cite{PT}, and the uniqueness up to multipliers of $p$-harmonic maps in $\R^N\setminus \{0\}$ under suitable conditions at zero and at infinity (see Theorem 2.1 in \cite{BS16}), we can deduce the sharp estimates on both upper and lower bounds for the decay rate of $|\nabla u|$.

\medskip

Our approach can deal with the generalized equation \eqref{geq} and establish the following sharp asymptotic estimates.
\begin{thm}[Sharp asymptotic estimates for generalized equation \eqref{geq}]\label{th2.1}
Assume $1<p<N$, and let $u$ be a nonnegative $D^{1,p}(\R^{N})$-weak solution to the generalized equation \eqref{geq}. Then $u\in L^{r}_{loc}(\mathbb{R}^{N})$ for any $0<r<+\infty$, and there exist $\tau>0$ and $\hat{p}\geq p^{\star}$ such that, for any $\bar{p}\geq \hat{p}$, there exists $R_{0}>1$ sufficiently large depending on $\bar{p}$ such that the following decay property holds:
\begin{equation}\label{a1}
  \|u\|_{L^{\bar{p}}(\mathbb{R}^{N}\setminus B_{2R}(0))}\leq\frac{C}{R^{\frac{N-p}{p}+\tau}}, \qquad \forall \,\, R>R_{0},
\end{equation}
where the constant $C>0$ is independent of $\bar{p}$ and $R$. If $V\in L^{q}_{loc}(\mathbb{R}^{N})$ for some $\frac{N}{p}<q<+\infty$ provided that $u\in L^{r}_{loc}(\mathbb{R}^{N})$ for any $0<r<+\infty$, then $u\in C(\R^N)\cap L^{\infty}_{loc}(\mathbb{R}^{N})$. If $V\in L^{q}_{loc}(\mathbb{R}^{N})$ for any $\frac{N}{p}<q<+\infty$ provided that $u\in C(\R^N)\cap L^{\infty}_{loc}(\mathbb{R}^{N})$, then $u \in C^{1,\alpha}(\R^N)$ for some $0<\alpha<\min\{1,\frac{1}{p-1}\}$. Assume that there exists $q>\frac{N}{p}$ and $\widetilde{R}>1$ large enough, such that the following uniform boundedness holds:
\begin{equation}\label{a2}
  \|V\|_{L^{q}(B_{1}(x))}\leq C, \qquad \forall \,\, |x|>2\widetilde{R},
\end{equation}
provided that the decay property in \eqref{a1} holds for sufficiently large $\bar{p}\geq \hat{p}$ and $R>R_{0}$, where the constant $C>0$ is independent of $x$ and $\widetilde{R}$, then, there exist $\tau>0$ (the same as in \eqref{a1}) and a constant $C>0$ independent of $x$ such that
$$ u(x) \leq\frac{C}{1+|x|^{\frac{N-p}{p}+\tau}} \quad \,\,\, \mbox{in}\,\, \R^N,$$
and hence $u\in L^{\infty}(\mathbb{R}^{N})$. Moreover, suppose that $V(x) \leq C_V |x|^{-\beta}$ for some $C_V>0$, $\beta>p$ and $|x|$ large provided that $u(x)\leq C|x|^{-\gamma}$ for some $C>0$, $\gamma>\frac{N-p}{p}$ and $|x|$ large, then, for any positive $D^{1,p}(\R^{N})$-weak solution $u$ to \eqref{geq}, there exist $c_0, C_0, R_0 > 0$ such that
\begin{equation}\label{eq0806}
  \frac{c_0}{1+|x|^\frac{N-p}{p-1}} \leq u(x) \leq \frac{C_0}{1+|x|^\frac{N-p}{p-1}} \qquad \mbox{in}\,\,\,\R^N,
\end{equation}
\begin{align}\label{eq0806+}
\frac{c_0}{|x|^\frac{N-1}{p-1}} \leq |\nabla u(x)| \leq \frac{C_0}{|x|^\frac{N-1}{p-1}} \qquad \mbox{in}\,\,\,\R^N \setminus B_{R_0}(0).
\end{align}
\end{thm}

\begin{rem}
When $V=u^{p^{\star}-p}$, assume that the decay property in \eqref{a1} holds for sufficiently large $\bar{p}\geq \hat{p}$ and $R>R_{0}$, then for $q>\frac{N}{p}$ large enough such that $\bar{p}:=q(p^{\star}-p)\geq\hat{p}$ and $\widetilde{R}:=4R_{0}$, one has, for any $|x|>2\widetilde{R}$,
$$\|V\|_{L^{q}(B_{1}(x))}\leq\|u\|^{p^{\star}-p}_{L^{\bar{p}}(B_{1}(x))}\leq\|u\|^{p^{\star}-p}_{L^{\bar{p}}(\mathbb{R}^{N}\setminus B_{4R_{0}}(0))}\leq\left(\frac{C}{(2R_{0})^{\frac{N-p}{p}+\tau}}\right)^{p^{\star}-p}\leq C,$$
where the constant $C>0$ is independent of $x$ and $\widetilde{R}$. That is, the uniform boundedness in \eqref{a2} holds. Suppose that $u(x)\leq C|x|^{-\gamma}$ for some $C>0$, $\gamma>\frac{N-p}{p}$ and $|x|$ large (this assumption can actually be guaranteed by the better preliminary estimate $u(x)\leq \frac{C}{1+|x|^{\frac{N-p}{p}+\tau}}$ with $\tau>0$ in Lemma \ref{lm.5}), then $V\leq C|x|^{-\beta}$ with $\beta:=(p^{\star}-p)\gamma>p$ for $|x|$ large. Therefore, it follows from Theorem \ref{th2.1} that the sharp asymptotic estimates \eqref{eq0806} and \eqref{eq0806+} also hold for any positive $D^{1,p}(\R^{N})$-weak solution $u$ to the $D^{1,p}(\R^{N})$-critical quasi-linear local equation \eqref{criticeqution}.
\end{rem}

\smallskip

When $V(x):=|x|^{-2p}\ast u^{p}$, one can also easily verify that all the assumptions on $V$ in Theorem \ref{th2.1} are fulfilled (for details, see the proof of Theorem \ref{th2.1-} in Section 3). For instance, assume that $u(x)\leq C|x|^{-\gamma}$ for some $C>0$, $\gamma>\frac{N-p}{p}$ and $|x|$ large (this assumption for $\gamma=\frac{N-p}{p}+\tau$ with $\tau>0$ can actually be guaranteed by the better preliminary estimate in Lemma \ref{lm.5}), then the basic estimates on the convolution $V(x)=|x|^{-2p}\ast u^{p}$ in (ii) of Lemma \ref{lm.6} (with $\sigma=2p$ and $q=p$) implies that, $V(x)\leq C_{\star}|x|^{-\beta}$ with $\beta:=p\gamma+2p-N=p+p\tau>p$ for some $C_{\star}>0$ and $|x|$ large. Therefore, as a consequence of Theorem \ref{th2.1}, we can deduce the following regularity result and sharp asymptotic estimates for the nonlocal quasi-linear Schr\"{o}dinger-Hartree equation \eqref{eq1.1}.
\begin{thm}[Sharp asymptotic estimates]\label{th2.1-}
Assume $1 < p < \frac{N}{2}$, and let $u$ be a nonnegative $D^{1,p}(\mathbb{R}^{N})$-weak solution to the $D^{1,p}(\R^N)$-critical quasi-linear nonlocal equation \eqref{eq1.1}. Then $u \in C^{1,\alpha} (\R^N) \cap L^\infty(\R^N)$ for some $0<\alpha<\min\{1,\frac{1}{p-1}\}$. Moreover, the sharp asymptotic estimates \eqref{eq0806} and \eqref{eq0806+} hold for any positive $D^{1,p}(\mathbb{R}^{N})$-weak solution $u$ to \eqref{eq1.1}.
\end{thm}

\begin{rem}
From $1< p<\frac{N}{2}$, the regularity $u \in C^{1,\alpha} (\R^N) \cap L^\infty(\R^N)$ in Theorem \ref{th2.1-}, one can easily verify that $V(x):=|x|^{-2p}\ast u^{p}\in L^{\infty}_{loc}(\mathbb{R}^{N})$ for any nonnegative weak solution $u\in D^{1,p}(\mathbb{R}^{N})$ to \eqref{eq1.1} (see \eqref{eq10.26.30.03}). Thus it follows from the strong maximum principle in Lemma \ref{hopf} that, for any nonnegative $D^{1,p}(\mathbb{R}^{N})$-weak solution $u$ to \eqref{eq1.1}, we have either $u\equiv0$ or $u>0$ in $\mathbb{R}^{N}$.
\end{rem}

\medskip

Thanks to the sharp asymptotic estimates on $u$ and $|\nabla u|$ in Theorem \ref{th2.1-}, inspired by \cite{LDSMLMSB,BS16,VJ16} etc., we are ready to prove the following theorem on radial symmetry and strictly radial monotonicity of positive $D^{1,p}(\mathbb{R}^{N})$-weak solutions to \eqref{eq1.1} via the method of moving planes.
\begin{thm}[radial symmetry and strictly radial monotonicity]\label{th2}
Assume $1<p<\frac{N}{2}$ and let $u$ be a nonnegative $D^{1,p}(\mathbb{R}^{N})$-weak solution of the $D^{1,p}(\mathbb{R}^{N})$-critical nonlocal quasi-linear equation \eqref{eq1.1}. Then, either $u\equiv0$; or $u>0$ is radially symmetric and strictly decreasing about some point $x_0\in\R^N$, i.e., positive solution $u$ must assume the form
\[u(x)=\lambda^{\frac{N-p}{p}}U(\lambda(x-x_{0}))\]
for $\lambda:=u(0)^{\frac{p}{N-p}}>0$ and some $x_{0}\in\mathbb{R}^{N}$, where $U(x)=U(r)$ with $r=|x|$ is a positive radial solution to \eqref{eq1.1} with $U(0)=1$, $U'(0)=0$ and $U'(r)<0$ for any $r>0$ ($U$ may not be unique).
\end{thm}

\begin{rem}\label{rm.2.1}
One should observe that the restriction $p\in (1,\frac{N}{2})$ for $p$ is derived from the $p$-Laplace operator, the convolution $V(x)=|x|^{-2p}\ast u^p$ in \eqref{eq1.1} and the Hardy-Littlewood-Sobolev inequality. Thus, in Theorems \ref{th2.1-} and \ref{th2}, we have derived a complete result on the sharp asymptotic estimates, radial symmetry and strictly radial monotonicity of the positive weak solutions to equation \eqref{eq1.1}.
\end{rem}

\smallskip

Since the $p$-Laplace operator is singular elliptic or degenerate elliptic if $1<p<2$ or $p>2$ respectively, there is no uniform form of the strong comparison principle for $p$-Laplace operator. Hence, we will carry out the proof of Theorem \ref{th2} by discussing the different two cases $1<p<2$ and $p>2$ separately. For the singular elliptic case $1<p<2$ (degenerate elliptic case $p>2$, resp.), we can apply the moving planes technique (see Section \ref{sc4}) by using the sharp asymptotic estimates in Theorem \ref{th2.1-}, weighted Hardy-Sobolev inequality and the (weighted, resp.) Poincar\'{e} type inequality (c.f. \cite{FV}), etc..

\smallskip

It is worth noting that the sharp lower bound of $|\nabla u|$ in Theorem \ref{th2.1-} indicates that the critical set $Z_u=\{x \in \R^N \mid |\nabla u(x)| = 0 \}$ of the solution $u$ is a closed set belonging to $B_{R_0}(0)$. Meanwhile, it follows from Corollary \ref{re2333} that $|Z_{u}|=0$ (see Remark \ref{re2334}), and from Lemma \ref{lm.7} that $\Omega\setminus Z_u$ is connected for any smooth bounded domain $\Omega\subset\mathbb{R}^{N}$ with connected boundary such that $B_{R_0}(0)\subseteq\Omega$. As a consequence, one has $\mathbb{R}^{N}\setminus Z_u$ is connected. Due to the closeness of $Z_{u}$ and the fact that $|Z_{u}|=0$, we can deduce that $\mathbb{R}^{N}\setminus Z^\nu_{\lambda}$ is connected, where $Z^\nu_{\lambda}\subset Z_u \subset B_{R_0}$ is defined by \eqref{Z} (see Section \ref{sc4}). Consequently, in the singular elliptic case $1<p<2$, we can simplify the moving planes techniques in \cite{DP,LDSMLMSB,DLPFRM,LDMR} (see sub-subsection 4.2.2), since the comparison principles therein work on each connected component $\mathcal{C}^{\nu}$ of $\Sigma^{\nu}_{\lambda}\setminus Z^\nu_{\lambda}$ and hence subtle analysis on different connected components is necessary in \cite{DP,LDSMLMSB,DLPFRM,LDMR} (for definition of the half space $\Sigma^{\nu}_{\lambda}$, see Section \ref{sc4}). For more literatures on the classification results and other quantitative and qualitative properties of solutions for various PDE and IE problems via the methods of moving planes, please refer to \cite{CGS,CD,CDZ,CL,CL1,CLL,CLO,CDQ,DFHQW,DFQ,DL,DQ,LD,DP,LDSMLMSB,DLPFRM,LDMR,LDBS04,DMPS,GNN,LPLDHT,Lin,Liu,MZ,BS16,S,VJ16,WX} and the references therein.

\smallskip

Since \eqref{eq1.1} is the Euler-Lagrange equation for the minimization problem \eqref{mini}, we can deduce the following sharp asymptotic estimates and radial symmetry and monotonicity results on minimizers of the minimization problem \eqref{mini} and on extremal functions of the Hardy-Littlewood-Sobolev inequality \eqref{HLS} as a corollary from Theorems \ref{th2.1-} and \ref{th2}.
\begin{cor}\label{cor1}
Assume $1<p<\frac{N}{2}$ and let positive function $u\in D^{1,p}(\mathbb{R}^{N})$ be a minimizer of the minimization problem \eqref{mini} \emph{or} a extremal function of the Hardy-Littlewood-Sobolev inequality \eqref{HLS}. Then $u \in C^{1,\alpha} (\R^N) \cap L^\infty(\R^N)$ for some $0<\alpha<\min\{1,\frac{1}{p-1}\}$ and $u$ satisfies the sharp asymptotic estimates \eqref{eq0806} and \eqref{eq0806+}. Moreover, $u$ is radially symmetric and strictly decreasing about some point $x_0\in\R^N$, i.e., $u$ must assume the form
\[u(x)=\lambda^{\frac{N-p}{p}}U(\lambda(x-x_{0}))\]
for $\lambda:=u(0)^{\frac{p}{N-p}}>0$ and some $x_{0}\in\mathbb{R}^{N}$, where $U(x)=U(r)$ with $r=|x|$ is a positive radial solution to \eqref{eq1.1} with $U(0)=1$, $U'(0)=0$ and $U'(r)<0$ for any $r>0$ ($U$ may not be unique).
\end{cor}

\medskip

Based on the radial symmetry and strictly radial monotonicity result in Theorem \ref{th2}, we may further guess that there exists an unique positive radial solution $\widehat{U}(x)=\widehat{U}(r)$ with $r=|x|$ to \eqref{eq1.1} with $\widehat{U}(0)=1$, $\widehat{U}'(0)=0$ and $\widehat{U}'(r)<0$ for any $r>0$, i.e., $\widehat{U}(r)$ is the unique solution to the following ODE:
\begin{footnotesize}
\begin{align}\label{eq1.1r}
\left\{ \begin{array}{ll} \displaystyle
|U'(r)|^{p-2}\left[\frac{1-N}{r}U'(r)-(p-1)U''(r)\right]=\left[\int_{\R^{N}}\frac{U^{p}(|y|)}{|(r,0,\cdots,0)-y|^{2p}}\mathrm{d}y\right]U^{p-1}(r),   & \forall \,\, r>0, \\ \\
U(0)=1, \,\,\, U'(0)=0, \,\,\, U(r)>0, \,\,\, U'(r)<0,  &  \forall \,\, r>0.
\end{array}
\right.\hspace{1cm}
\end{align}\end{footnotesize}

\noindent As a consequence, any positive weak solution $u$ to \eqref{eq1.1} must be uniquely determined by $\widehat{U}$ up to scalings and translations. We have the following conjecture on complete classification result on all nonnegative $D^{1,p}(\mathbb{R}^{N})$-weak solutions for \eqref{eq1.1} (including any minimizers of the minimization problem \eqref{mini} \emph{and} any extremal functions of the Hardy-Littlewood-Sobolev inequality \eqref{HLS}).
\begin{conj}\label{conjecture}
Assume $1<p<\frac{N}{2}$ and let $u$ be a nonnegative $D^{1,p}(\mathbb{R}^{N})$-weak solution of the $D^{1,p}(\mathbb{R}^{N})$-critical nonlocal quasi-linear equation \eqref{eq1.1}. Then, either $u\equiv0$; or $u>0$ must be uniquely determined by the following form
\[u(x)=\lambda^{\frac{N-p}{p}}\widehat{U}(\lambda(x-x_{0}))\]
for $\lambda:=u(0)^{\frac{p}{N-p}}>0$ and some $x_{0}\in\mathbb{R}^{N}$, where $\widehat{U}(x)=\widehat{U}(r)$ with $r=|x|$ is the unique positive solution to the ODE \eqref{eq1.1r}.
\end{conj}

\subsection{Extensions to more general $D^{1,p}(\R^{N})$-critical nonlocal quasi-linear equations}
More generally, we also consider the following general $D^{1,p}(\R^{N})$-critical nonlocal quasi-linear equation with generalized nonlocal nonlinearity:
\begin{align}\label{gnqe}
\left\{ \begin{array}{ll} \displaystyle
-\Delta_p u =\left(|x|^{-\sigma}\ast u^{p_{\sigma}^{\star}}\right)u^{p_{\sigma}^{\star}-1}  \quad\,\,\,\,&\mbox{in}\,\, \R^N, \\ \\
u \in D^{1,p}(\R^N),\quad\,\,\,\,\,\,u\geq0 \qquad\,\,\,\,&  \mbox{in}\,\, \R^N,
\end{array}
\right.\hspace{1cm}
\end{align}
where $1<p<N$, $N\geq2$, $p\leq p_{\sigma}^{\star}:=\frac{p(2N-\sigma)}{2(N-p)}<p^{\star}$, $p_{\sigma}^{\star}>\frac{p^{\star}}{2}$, $0<\sigma<N$ and $\sigma\leq 2p$. Define $V(x):=\left(|x|^{-\sigma}\ast u^{p_{\sigma}^{\star}}\right)(x)u^{p_{\sigma}^{\star}-p}(x)$, then the general nonlocal quasi-linear equation \eqref{gnqe} is a special case of the generalized equation \eqref{geq}. In the most important endpoint case $\sigma=2p$ and $p_{\sigma}^{\star}=p$ with $1<p<\frac{N}{2}$, the general nonlocal quasi-linear equation \eqref{gnqe} degenerates into the $D^{1,p}(\R^{N})$-critical nonlocal quasi-linear Schr\"{o}dinger-Hartree equation \eqref{eq1.1}.

\medskip

For \eqref{gnqe}, the finite $D^{1,p}(\R^N)$-energy assumption $u\in D^{1,p}(\R^N)$ is necessary to guarantee $V(x):=\left(|x|^{-\sigma}\ast u^{p_{\sigma}^{\star}}\right)(x)u^{p_{\sigma}^{\star}-p}(x)\in L^{\frac{N}{p}}(\mathbb{R}^{N})$ and hence $V(x)$ is finite and well-defined almost everywhere. Indeed, by Hardy-Littlewood-Sobolev inequality and H\"{o}lder's inequality, one has
\begin{equation*}
  \|V\|_{L^{\frac{N}{p}}(\mathbb{R}^{N})}\leq \||x|^{-\sigma}\ast u^{p_{\sigma}^{\star}}\|_{L^{\frac{2N}{\sigma}}(\mathbb{R}^{N})}\|u\|^{p_{\sigma}^{\star}-p}_{L^{p^{\star}}(\mathbb{R}^{N})}\leq C\|u\|^{2p_{\sigma}^{\star}-p}_{L^{p^{\star}}(\mathbb{R}^{N})}, \qquad \forall \,\, u\in D^{1,p}(\mathbb{R}^{N}).
\end{equation*}
Thus it follows from H\"{o}lder's inequality and Sobolev embedding inequality that
\begin{equation}\label{gHLS}
  \|u\|_{V}:=\left(\int_{\mathbb{R}^{N}}\int_{\mathbb{R}^{N}}\frac{|u|^{p_{\sigma}^{\star}}(x)|u|^{p_{\sigma}^{\star}}(y)}{|x-y|^{\sigma}}\mathrm{d}x\mathrm{d}y\right)^{\frac{1}{2p_{\sigma}^{\star}}}
  \leq C_{N,p,\sigma}\| \nabla u \|_{L^p(\R^N)},  \quad\, \forall \,\, u\in D^{1,p}(\mathbb{R}^{N}),
\end{equation}
where $C_{N,p,\sigma}$ is the best constant for the general Hardy-Littlewood-Sobolev inequality \eqref{gHLS}, which can be characterized by the following minimization problem
\begin{equation}\label{gmini}
  \frac{1}{C_{N,p,\sigma}}:=\inf\left\{\| \nabla u \|_{L^p(\R^N)}\,\Big|\,u\in D^{1,p}(\mathbb{R}^{N}), \,\, \|u\|_{V}=1\right\}.
\end{equation}
The general $D^{1,p}(\mathbb{R}^{N})$-critical nonlocal quasi-linear equation \eqref{gnqe} is the Euler-Lagrange equation for the minimization problem \eqref{gmini}.

\medskip

In \cite{Lions1}, by using the concentration-compactness Lemma I.1 in \cite{Lions1} with $|u_{n}|^{p^{\star}}$ replaced by $\left(|x|^{-\sigma}\ast |u_{n}|^{p_{\sigma}^{\star}}\right)|u_{n}|^{p_{\sigma}^{\star}}$, Lions proved in iv) in page 169 of \cite{Lions1} that the minimization problem \eqref{gmini} has a minimum, and hence the general nonlocal quasi-linear equation \eqref{gnqe} possesses (at least) a positive weak solution $u\in D^{1,p}(\mathbb{R}^{N})$. As a consequence, the limiting embedding inequality \eqref{gHLS} is sharp in the sense that the best constant $C_{N,p,\sigma}$ can be attained by some (positive) extremal function $u\in D^{1,p}(\mathbb{R}^{N})$. We aim to investigate the sharp asymptotic estimates, radial symmetry and strictly radial monotonicity of positive solutions to \eqref{gnqe}.

\medskip

By showing that $V(x):=\left(|x|^{-\sigma}\ast u^{p_{\sigma}^{\star}}\right)(x)u^{p_{\sigma}^{\star}-p}(x)$ satisfies all the assumptions in Theorem \ref{th2.1}, we can derive the sharp asymptotic estimates on $u$ and $|\nabla u|$, and hence the radial symmetry and strictly radial monotonicity in the following theorem via the method of moving planes.
\begin{thm}[Sharp asymptotic estimates and radial symmetry]\label{gth2}
Let $u$ be a nonnegative $D^{1,p}(\mathbb{R}^{N})$-weak solution to the general $D^{1,p}(\R^N)$-critical quasi-linear nonlocal equation \eqref{gnqe}. Then $u \in C^{1,\alpha} (\R^N) \cap L^\infty(\R^N)$ for some $0<\alpha<\min\{1,\frac{1}{p-1}\}$. Moreover, the sharp asymptotic estimates \eqref{eq0806} and \eqref{eq0806+} on $u$ and $|\nabla u|$ hold for any positive $D^{1,p}(\mathbb{R}^{N})$-weak solution $u$ to \eqref{gnqe}. Then, either $u\equiv0$; or $u>0$ is radially symmetric and strictly decreasing about some point $x_0\in\R^N$, i.e., positive solution $u$ must assume the form
\[u(x)=\lambda^{\frac{N-p}{p}}U(\lambda(x-x_{0}))\]
for $\lambda:=u(0)^{\frac{p}{N-p}}>0$ and some $x_{0}\in\mathbb{R}^{N}$, where $U(x)=U(r)$ with $r=|x|$ is a positive radial solution to \eqref{gnqe} with $U(0)=1$, $U'(0)=0$ and $U'(r)<0$ for any $r>0$ ($U$ may not be unique).
\end{thm}

Once the uniqueness of positive radial solutions $U$ with $U(0)=1$, $U'(0)=0$ and $U'(r)<0$ ($\forall \,\, r>0$) to \eqref{gnqe} was derived, then any positive weak solution $u$ to \eqref{gnqe} must be uniquely determined by the unique positive radial solution $\widehat{U}$ up to scalings and translations, and hence the complete classification result follows.

\smallskip

The rest of our paper is organized as follows. In Section \ref{sc2}, we will give some preliminaries on important inequalities, (strong) comparison principles, strong maximum principle and H\"{o}pf's Lemma, related regularity results and properties of the critical set etc.. Section \ref{sc3} is devoted to the proof of the key ingredients: Theorems \ref{th2.1} and \ref{th2.1-}, i.e., the sharp asymptotic estimates for positive weak solution $u$ and its gradient $|\nabla u|$ to the generalized equation \eqref{geq} and hence the nonlocal equation \eqref{eq1.1}. The proof of Theorem \ref{th2} will be carried out in Section \ref{sc4}. Finally, in Section \ref{sc5}, we prove Theorem \ref{gth2}.

\smallskip

In the following, we will use the notation $B_{R}$ to denote the ball $B_{R}(0)$ centered at $0$ with radius $R$, and use $C$ to denote a general positive constant that may depend on $N$, $p$, $\sigma$ and $u$, and whose value may differ from line to line.

\section{Preliminaries}\label{sc2}
In this section, we will give some useful tools in our proofs of Theorems \ref{th2.1}, \ref{th2.1-} and \ref{th2}, including some important inequalities, (strong) comparison principles, strong maximum principle and H\"{o}pf's Lemma, related regularity results and properties of the critical set etc.. For clarity of presentation, we will split this Section into two sub-sections.

\subsection{Some key inequalities}\label{sc2.1}

We will use the following Hardy-Littlewood-Sobolev inequality.
\begin{thm}[Hardy-Littlewood-Sobolev inequality, c.f. e.g. \cite{FL1,FL2,Lieb,Lions2}]\label{HLSI}
Assume $t,r>1$ and $0<\sigma<N$ with $\frac{1}{t}+\frac{\sigma}{N}+\frac{1}{r}=2$. Let $f\in L^t(\R^N)$ and $h\in L^r(\R^N)$. Then there exists a sharp constant $C(N,\sigma,r,t)$, independent of $f$ and $h$, such that
\begin{align*}
\left| \int_{\R^N} \int_{\R^N} \frac{f(x) h(y)}{|x-y|^{\sigma}}\mathrm{d}x\mathrm{d}y \right|\leq C(N,\sigma,r,t)  \| f \|_{L^t(\R^N)} \| h \|_{L^r(\R^N)}.
\end{align*}
Moreover, if $\frac{N}{q}=\frac{N}{m}-N+\sigma$ and $1<m<q<\infty$, we have
\begin{align*}
\left\| \int_{\R^N} \frac{g(x)}{|x-y|^{\sigma}}\mathrm{d}x \right\|_{L^q(\R^N)}\leq \overline{C}(N,\sigma,q,m)  \| g \|_{L^m(\R^N)},
\end{align*}
for all $g\in L^m(\R^N)$.

\end{thm}

Now, we prove the following basic point-wise estimates for the convolution term of the type $|x|^{-\sigma}*u^{q}$.
\begin{lem}\label{lm.6}
Let $0<\sigma<N$, $q>0$ and $R>0$.\\
{\bf $(i)$} If $u\geq0$ in $\mathbb{R}^{N}$ and $u \geq C_1 |x|^{-\beta}$ in $\R^N\setminus B_R$ for some $C_1,\beta>0$, then for some constant $C>0$, we have
\begin{align}\label{eq10.26.18}
\left\{ \begin{array}{ll}
|x|^{-\sigma}*u^q=\infty,  \,\,\,\,&\mbox{if}\,\, q\beta\leq N-\sigma \\
|x|^{-\sigma}*u^q\geq C |x|^{N-\sigma-q\beta},  \,\,\,\,&\mbox{if}\,\, q\beta >N-\sigma
\end{array}
\right. \,\,\,\,\,\,\, &\mbox{in}\,\, \R^N\setminus B_R(0).
\end{align}
{\bf $(ii)$} If $u \in L^\infty(B_R)$ and $ 0\leq u \leq C_1 |x|^{-\beta}$ in $\R^N\setminus B_{R}$ for some $C_1>0$ and $q\beta>N-\sigma$, then
\begin{align}\label{eq10.26.19}
|x|^{-\sigma}*u^q \leq \left\{ \begin{array}{ll}
                            C |x|^{N-\sigma-q\beta}   \,\,\,\,&\mbox{if}\,\, N-\sigma <q\beta <N \\
                            C |x|^{-\sigma}\log |x|  \,\,\,\,&\mbox{if}\,\, q\beta =N \\
                            C |x|^{-\sigma}  \,\,\,\,&\mbox{if}\,\, q\beta>N \\
                    \end{array}
                    \right.
                    \,\,\,\,\,\,\,&\mbox{in}\,\, \R^N\setminus B_{2R}(0)
\end{align}
for some constant $C>0$.\\
{\bf $(iii)$} 
If $ \frac{c }{1+|x|^\beta} \leq f(x) \leq \frac{C }{1+|x|^\beta}$ in $\R^N$ for some $\beta>N-\sigma$ and constants $c,\,C>0$,  then
\begin{align}\label{eq10.26.30}
C_1 \leq |x|^{-\sigma}*f \leq C_2\,\,\,\,\,\,\,\,&\mbox{in}\,\, B_R(0)
\end{align}
for some constants $C_1, \, C_2>0$ depending on $R$.
\end{lem}
\begin{proof}
{\bf $(i)$} Since $|x-y|\leq |x|+|y| \leq 2|y|$ if $|x|\leq |y|$, we get
$$ |x|^{-\sigma}*u^q \geq C_1 \int_{|x|\leq |y|} \frac{|y|^{-q\beta}}{|x-y|^{\sigma}} \mathrm{d}y \geq C \int_{|x|\leq |y|} |y|^{-q\beta-\sigma}\mathrm{d}y \geq C \int_{|x|}^\infty \tau^{N-1-q\beta-\sigma} \mathrm{d}\tau,$$
which leads to the conclusion in $(i)$.

\smallskip

{\bf $(ii)$} For $|x|\geq 2R$, we have
\begin{align}\label{eq10.28.30}
&\int_{\R^N} \frac{u^q(y)}{|x-y|^{\sigma}}\mathrm{d}y =\left\{\int_{|y|\leq\frac{|x|}{2}} + \int_{\frac{|x|}{2} \leq |y| \leq 2|x|} +\int_{|y|\geq 2|x|}  \right\} \frac{u^q(y)}{|x-y|^{\sigma}}\mathrm{d}y \\
&\leq C |x|^{-\sigma}\int_{|y|\leq\frac{|x|}{2}} u^q(y) \mathrm{d}y + C_1 |x|^{-q\beta} \int_{\frac{|x|}{2} \leq |y| \leq 2|x|} \frac{\mathrm{d}y}{|x-y|^{\sigma}} + C \int_{|y|\geq 2|x|} \frac{\mathrm{d}y}{|y|^{\sigma+q\beta}}  \nonumber\\
&\leq  C |x|^{-\sigma} \left\{\int_{|y|\leq 1 } + \int_{1< |y|\leq\frac{|x|}{2}} \right\} u^q(y) \mathrm{d}y
+ C_1 |x|^{-q\beta} \int_{|y-x|\leq 3|x|} \frac{\mathrm{d}y}{|x-y|^{\sigma}}
+ C |x|^{N-\sigma-q\beta} \nonumber\\
&\leq  C |x|^{-\sigma} \left\{C + \int_{1< |y|\leq\frac{|x|}{2}} |y|^{-q\beta} \mathrm{d}y \right\}
+ C |x|^{N-\sigma-q\beta},\nonumber
\end{align}
in which we have used the following facts:
\begin{align*}
|x-y|\geq |x|-|y| \geq \frac{|x|}{2}, \,\,\,\,\,\,\,&\mbox{if}\,\,|y|\leq \frac{|x|}{2},\\
|x-y|\leq |x|+|y| \leq 3|x|, \,\,\,\,\,\,\,&\mbox{if}\,\, \frac{|x|}{2} \leq |y|\leq 2|x|,\\
|x-y|\geq |y|-|x| \geq \frac{|y|}{2}, \,\,\,\,\,\,\,&\mbox{if}\,\,|y|\geq 2|x|.
\end{align*}
On the other hand, for $|x|\geq 2R$,
$$\int_{1< |y|\leq\frac{|x|}{2}} |y|^{-q\beta} \mathrm{d}y
                    \leq \left\{ \begin{array}{ll}
                            C|x|^{N-q\beta},   \,\,\,\,\,\,&\mbox{if}\,\, q\beta <N, \\
                            C\log |x|,  \,\,\,\,\,\,&\mbox{if}\,\, q\beta =N, \\
                            C,  \,\,\,\,\,\,&\mbox{if}\,\, q\beta >N, \\
                    \end{array}
                    \right.$$
which together with \eqref{eq10.28.30} yields $(ii)$.

\smallskip

{\bf $(iii)$} For any $x\in B_R(0)$,
\begin{align*}
|x|^{-\sigma}*f &= \int_{\R^N} \frac{f(y)}{|x-y|^{\sigma}} \mathrm{d}y \\
&= \int_{B_{2R}} \frac{f(y)}{|x-y|^{\sigma}} \mathrm{d}y +  \int_{\R^N \setminus B_{2R}} \frac{f(y)}{|x-y|^{\sigma}} \mathrm{d}y =:I_1 + I_2.
\end{align*}
Since $\frac{1}{2}|y| \leq |y|-R \leq |x-y| \leq 2|y|$ for $x\in B_R$ and $y\in B^c_{2R}$, we have
\begin{align*}
I_2 &= \int_{\R^N \setminus B_{2R}} \frac{f(y)}{|x-y|^{\sigma}} \mathrm{d}y\\
&\leq C \int_{\R^N \setminus B_{2R}} \frac{1}{|y|^{\sigma}} |y|^{-\beta} \mathrm{d}y \leq C \int_{2R}^\infty r^{-\beta-\sigma+N-1} \mathrm{d}r\leq C R^{-\beta-\sigma+N},\,\,\,\,\,\,\,\,\,x\in B_R(0),
\end{align*}
and
\begin{align*}
I_2 &= \int_{\R^N \setminus B_{2R}} \frac{f(y)}{|x-y|^{\sigma}} \mathrm{d}y \\
&\geq c \int_{\R^N \setminus B_{2R}} \frac{1}{|y|^{\sigma}} |y|^{-\beta} \mathrm{d}y  \geq c \int_{2R}^\infty r^{-\beta-\sigma+N-1} \mathrm{d}r = C R^{-\beta-\sigma+N},\,\,\,\,\,\,\,\,\,x\in B_R(0).
\end{align*}
On the other hand, for any $x\in B_R(0)$,
\begin{align*}
I_1 &= \int_{B_{2R}(0)} \frac{f(y)}{|x-y|^{\sigma}} \mathrm{d}y\leq C \int_{B_{3R}(0)} \frac{1}{|y|^{\sigma}} \mathrm{d}y \\
&\leq C \int_0^{3R} r^{N-1-\sigma} dr = \frac{C}{N-\sigma}(3R)^{N-\sigma}.
\end{align*}
Therefore, we have arrived at $(iii)$. The proof of Lemma \ref{lm.6} is completed.
\end{proof}

\begin{rem}\label{rm.2.7}
Assume $u$ is a nonnegative weak solution to \eqref{eq1.1} and let the nonlocal nonlinear term $f(x):=V(x)u^{p-1}(x):=\left(|x|^{-2p}*u^{p}\right)u^{p-1}(x)$ with $1<p<\frac{N}{2}$. It is known that $u\in C_{loc}^{1,\alpha}(\R^N)$ for some $0<\alpha<\min\{1,\frac{1}{p-1}\}$ (see Theorems \ref{th2.1} and \ref{th2.1-}). Moreover, if $u$ is positive, by Theorems \ref{th2.1} and \ref{th2.1-}, we obtain that, there exist $c_0, C_0 > 0$ such that
$$ \frac{c_0}{1+|x|^\frac{N-p}{p-1}} \leq u(x) \leq \frac{C_0}{1+|x|^\frac{N-p}{p-1}} \,\,\,\,\,\,\,\mbox{in}\,\,\R^N,$$
\[\frac{c_0}{|x|^\frac{N-1}{p-1}} \leq |\nabla u(x)| \leq \frac{C_0}{|x|^\frac{N-1}{p-1}} \qquad \mbox{in}\,\,\,\R^N \setminus B_{R_0}(0).\]
Therefore, by $(iii)$ in Lemma \ref{lm.6} with $\sigma=2p$, for any $R>0$, we have
\begin{align}\label{eq10.26.30.03}
C_1 \leq V(x) \leq C_2\,\,\,\,\,\,\,\mbox{in}\,\, B_R,
\end{align}
and
\begin{align}\label{eq10.26.30.02}
|V_{x_i}(x)| \leq \int_{\R^N} \frac{|u(x-y)|^{p-1} |\nabla u(x-y)|}{|y|^{2p}} \mathrm{d}y \leq C_3 \,\,\,\,\,\,\mbox{in}\,\, B_R,\,\,\,i=1,2,\cdots,N,
\end{align}
where the constants $C_i>0$ $(i=1,2,3)$ depend on $R$. Furthermore, \eqref{eq10.26.30.03} and \eqref{eq10.26.30.02} show that, for any bounded domain $\Omega \subset \R^N$, one has $f(x):=V(x)u^{p-1}(x)$ satisfies
\begin{align}\label{eq26.30.001}
f(x)\in L^s (\Omega) \,\,\,\,\,\,\,\,\mbox{with}\,s>N,
\end{align}
and the so-called ``condition ($I_{\alpha}$)" in Lemma \ref{lm.4}, i.e., for some $0<\mu<\alpha(p-1)$, there holds:
\begin{align}\label{eq26.30.002}
f(x)\in W^{1,m}(\Omega)
\end{align}
for some $m>\frac{N}{2(1-\mu)}$, and, given any $x_0\in \Omega$, there exist $C(x_0,\mu)>0$ and $\rho(x_0,\mu)>0$ such that, $B_{\rho(x_0,\mu)}(x_{0})\subset \Omega$ and
\begin{align}\label{eq26.30.003}
|f(x)|\geq C(x_0,\mu) |x-x_0|^\mu \qquad \mbox{in}\,\,B_{\rho(x_0,\mu)}(x_{0}).
\end{align}
\end{rem}

\medskip

In order to apply the moving planes method to $p$-Laplace equations, we will frequently use the following basic point-wise estimate (Lemma \ref{BSICI}), the weighted Hardy-Sobolev inequality (Lemma \ref{HDSI}) and the strong comparison principles (Lemma \ref{th3.1}-\ref{th3.2}).
\begin{lem}[Lemma 2.1 in \cite{LD}] \label{BSICI}
For any $ p>1 $, the following elementary point-wise inequality
\begin{align}\label{eq0802}
[|\eta|^{p-2}\eta-|\eta^\prime|^{p-2}\eta^\prime][\eta-\eta^\prime] &\geq C' (|\eta|+|\eta^\prime|)^{p-2}|\eta-\eta^\prime|^2; \\
\left| |\eta|^{p-2}\eta-|\eta^\prime|^{p-2}\eta^\prime \right| &\leq C'' (|\eta|-|\eta^\prime|)^{p-1},\,\,\,\,\,\,\,\mbox{if}\,\, 1<p\leq 2 \nonumber
\end{align}
hold for all $\eta,\eta^\prime \in \R^N$ with $|\eta|+|\eta^\prime|>0$, where $p>1$ and $C', C''>0$ depends only on $N$ and $p$.
\end{lem}

\begin{lem}[Weighted Hardy-Sobolev inequality, c.f. \cite{LDMR}]\label{HDSI}
Let $ u \in C_0^1(\R^N)$ and $q\geq 1$. Then for any $ s > q-N $ and $ T\geq 0 $ we have
$$ \int_{\R^N\setminus B_T(0)} |u|^q |x|^{s-q} \mathrm{d}x \leq \left( \frac{q}{N-q+s} \right)^q \int_{\R^N\setminus B_T(0)} |\nabla u|^q |x|^s \mathrm{d}x.$$
\end{lem}

\begin{lem}[Theorem 1.4 in \cite{LDBS}, c.f. also Theorem 2.1 in \cite{OSV}]\label{th3.1}
Let $\frac{2N+2}{N+2}<p<2$ or $p>2$ and $u,v \in C^1(\overline{\Omega})$, where $\Omega$ is a bounded smooth domain of $\R^N$. Suppose that either $u$ or $v$ is a weak solution of
\begin{align*}
\left\{ \begin{array}{ll}
-\Delta_p u =f(x,u) \,\,\,\,\,\, & \mbox{in}\,\,\Omega, \\
u>0 \,\,\,\,\,\,& \mbox{in}\,\,\Omega, \\
u=0 \,\,\,\,\,\,& \mbox{on}\,\,\partial\Omega \\
\end{array}
\right.\hspace{1cm}
\end{align*}
with $f:\overline{\Omega}\times [0,\infty)\to \R$ is a continuous function which is positive and of class $C^1$ in $\Omega\times(0,\infty)$. Assume that
$$ -\Delta_p u + \Lambda u \leq  -\Delta_p v + \Lambda v \qquad \,\,\,\mbox{and} \quad \,u\leq v\,\quad\mbox{in}\,\,\Omega,$$
where $\Lambda\in\R$. Then $u\equiv v$ in $\Omega$ unless $u<v$ in $\Omega$.
\end{lem}

\begin{lem}[Theorem 1.4 in \cite{LD}]\label{th3.2}
Suppose $\Omega$ is a bounded domain of $\R^N$, $1<p<\infty$ and let $u,v \in C^1(\Omega)$ weakly satisfy
$$ -\Delta_p u + \Lambda u \leq  -\Delta_p v + \Lambda v \quad\,\,\,\mbox{and}\quad\,u\leq v\,\,\,\,\,\mbox{in}\,\,\Omega,$$
and denote by $Z_v^u:=\{ x\in\Omega\,\mid\,|\nabla u|=|\nabla v| =0 \}$. Then if there exists $x_0\in\Omega\setminus Z_v^u$ with $u(x_0)=v(x_0)$, then $u\equiv v$ in the connected component of $\Omega\setminus Z_v^u$ containing $x_0$. The same result still holds if, more generally,
$$ -\Delta_p u - f(u) \leq  -\Delta_p v - f(u) \quad\,\,\,\mbox{and}\,\quad u\leq v \quad\,\mbox{in}\,\,\Omega$$
with $f:\R \to \R$ locally Lipschitz continuous.
\end{lem}

We recall a version of the strong maximum principle and the Hopf Lemma for the $p$-Laplacian, c.f. \cite{JLV}, see also Theorem 2.1 in \cite{LDBS04}, Theorem 2.4 in \cite{LDSMLMSB} or \cite{LD,PS}.
\begin{lem}[Strong Maximum Principle and H\"{o}pf's Lemma, e.g. \cite{JLV}, Theorem 2.1 in \cite{LDBS04}]\label{hopf}
Let $\Omega$ be a domain in $\mathbb{R}^{N}$ and suppose that $u\in C^{1}(\Omega)$, $u\geq0$ in $\Omega$ weakly solves
\[-\Delta_{p}u+cu^{q}=g\geq0 \qquad \text{in} \,\,\Omega\]
with $1<p<+\infty$, $q\geq p-1$, $c\geq0$ and $g\in L^{\infty}_{loc}(\Omega)$. If $u\not\equiv0$, then $u>0$ in $\Omega$. Moreover, for any point $x_{0}\in\partial\Omega$ where the interior sphere condition is satisfied and $u\in C^{1}(\Omega\cup\{x_{0}\})$ and $u(x_{0})=0$, we have that $\frac{\partial u}{\partial s}>0$ for any inward directional derivative (this means that if $y$ approaches $x_{0}$ in a ball $B\subseteq\Omega$ that has $x_{0}$ on its boundary, then $\lim\limits_{y\rightarrow x_{0}}\frac{u(y)-u(x_{0})}{|y-x_{0}|}>0$).
\end{lem}

\subsection{Regularity results and Properties of the critical Set}\label{sc2.2}

In this subsection, we will give the integrability properties of $\frac{1}{|\nabla u|}$ and the connectivity properties of the critical set $Z_{u}$ for the nonnegative weak solution $u$ of \eqref{eq1.1}, and the weighted Poincar\'{e} type inequality.

\smallskip

From Theorems \ref{th2.1} and \ref{th2.1-}, the nonnegative weak solution $u$ of \eqref{eq1.1} satisfies $u\in L^{\infty}(\mathbb{R}^{N})\cap C^{1,\alpha}_{loc} (\R^N)$ for some $0<\alpha<\min\{1,\frac{1}{p-1}\}$. Moreover, using standard elliptic regularity, the solution $u\in C^{1,\alpha} (\Omega)$ belongs to the class $C^2(\Omega\setminus Z_u)$, where $\Omega\subset\subset\R^N$ and $Z_u=\{ x\in\Omega : |\nabla u|=0 \}$ (see \cite{ED83,DGNT,PT}). We need the following Hessian and reversed gradient estimates from \cite{BSR14}.
\begin{lem}[Hessian and reversed gradient estimates, \cite{BSR14}] \label{lm.4}
Let $\Omega\subset\subset\R^N$ and $1<p<\infty$. Assume $u\in W^{1,p}(\Omega)$ be a weak solution to
\begin{align}\label{eq10.252252}
-\Delta_p u =f(x) \quad\,\,\,\mbox{in}\,\,\Omega,
\end{align}
where $f(x)$ satisfies \eqref{eq26.30.001} and condition ($I_\alpha$) in Remark \ref{rm.2.7}, i.e., \eqref{eq26.30.002}-\eqref{eq26.30.003}. Then, for any $x_0\in Z_u$ and $B_{2\rho}(x_0)\subset \Omega$, we have
\begin{align}\label{eq10.25.52}
\int_{B_\rho(x_0)} \frac{|\nabla u|^{p-2-\beta}(x)}{|x-y|^\gamma} |\nabla  u_{x_i}(x)|^2 \mathrm{d}x < C,\quad\,\,i=1,\cdots,N
\end{align}
uniformly for any $y\in \Omega$, where $0\leq\beta<1$, $\gamma=0$ if $N=2$, $\gamma<N-2$ if $N \geq 3$ and
$$C=C(N,p,x_0,\gamma,\beta,u,f,\rho)>0.$$
Moreover, we have, uniformly for any $y\in \Omega$,
\begin{align}\label{eq10.25.2.1}
\int_{B_\rho(x_0)} \frac{1}{|\nabla u|^t} \frac{1}{|x-y|^\gamma}\mathrm{d}x <\widetilde{C},
\end{align}
where $\max\{p-2,0\}\leq t < p-1$ and $\gamma=0$ if $N=2$, $\gamma<N-2$ if $N \geq 3$ and $$\widetilde{C}=\widetilde{C}(N,p,x_0,\gamma,t,u,f,\rho)>0.$$
\end{lem}

From Lemma \ref{lm.4}, we can infer the following Corollary.
\begin{cor}\label{re2333}
Let $1<p<\frac{N}{2}$ and $u$ be a positive weak solution of \eqref{eq1.1}. Setting $\rho:=|\nabla u|^{p-2}$, for any $\Omega \subset\subset\R^N$, we have
\begin{align}\label{eq10.25.2}
\int_{\Omega} \frac{1}{\rho^t} \frac{1}{|x-y|^\gamma}\mathrm{d}x < C
\end{align}
uniformly for any $y\in\Omega$, where $\max\{p-2,0\}\leq (p-2)t<p-1$ and $\gamma=0$ if $N=2$, $\gamma<N-2$ if $N \geq 3$.
\end{cor}
\begin{proof}
In Remark \ref{rm.2.7}, we know that $f(x):=\left(|x|^{-2p}*u \right)u^{p-1}$ satisfies \eqref{eq26.30.001} and the condition ($I_\alpha$) (i.e., \eqref{eq26.30.002}-\eqref{eq26.30.003}) in Lemma \ref{lm.4}. Noting that $u\in C^2(\Omega\setminus Z_u)$, it follows from the finite covering theorem and \eqref{eq10.25.2.1} in Lemma \ref{lm.4} that \eqref{eq10.25.2} holds. This finishes the proof of Corollary \ref{re2333}.
\end{proof}

\begin{rem}\label{re2334}
Let $u$ be a positive weak solution of \eqref{eq1.1} and let $Z_u = \{x \in \R^N: |\nabla u(x)| = 0 \}$. Clearly, it follows from Theorem \ref{th2.1} that the critical set $Z_u\subset B_{R_0}$ and $Z_u$ is a closed set. Moreover, \eqref{eq10.25.2} in Corollary \ref{re2333} implies that the Lebesgue measure of $Z_{u}$ is zero, i.e., $|Z_u|=0$.
\end{rem}

From Lemma \ref{th3.2}, we can clearly see that it is very useful to know whether $\Omega\setminus Z_u$ is connected or not. We have the following Lemma on the connectivity properties of the critical set $Z_{u}$ from \cite[Theorem 4.1 and Remark 4.1]{LDBS04}.
\begin{lem}[Properties of the Critical Set, Theorem 4.1 in \cite{LDBS04}]\label{lm.7}
Assume $1<p<\frac{N}{2}$ and let $u$ be a positive weak solution of \eqref{eq1.1}. Let us define the critical set
$$Z_u = \{ x\in\Omega : |\nabla u(x)|=0 \},$$
where $\Omega$ is a general bounded domain. 
Then $\Omega\setminus Z_u$ does not contain any connected component $C$ such that $\overline{C} \subset \Omega$.
\end{lem}

Finally, we give the weighted Poincar\'{e} type inequality inequality, which plays a key role in the proof of Theorem \ref{th2} via the moving planes techniques in the degenerate elliptic case $p>2$. To this end, we need the following definition.
\begin{defn}[Weighted Sobolev spaces]\label{lm.3}
Assume $\Omega\subset\R^N$ is a bounded domain and let $\rho\in L^1(\Omega)$ be a positive function. We define $H^{1,p}(\Omega,\rho)$ as the completion of $C^1(\overline{\Omega})$ (or $C^\infty(\overline{\Omega})$) with the norm
$$ \| v \|_{ H^{1,p}(\Omega,\rho) } := \| v \|_{L^p(\Omega)} + \| \nabla u\|_{L^p(\Omega,\rho)}$$
and $\| \nabla u\|^p_{L^p(\Omega,\rho)} = \int_\Omega  \rho  |\nabla u |^p \mathrm{d}x$. In this way $H^{1,2}(\Omega,\rho)$ is a Hilbert space, and we also define $H_0^{1,p}(\Omega,\rho)$ as the closure of $C_0^\infty(\overline{\Omega})$ in $H^{1,p}(\Omega,\rho)$.
\end{defn}

\begin{lem}[Weighted Poincar\'{e} type inequality, \cite{LDBS04,FV,LPLDHT}]\label{lm.3}
Assume $\Omega\subset\R^N$ is a bounded domain and let $\rho\in L^1(\Omega)$ be a positive function such that
\begin{align}\label{eq10.25.3}
\int_\Omega \frac{1}{\rho |x-y|^\gamma}\mathrm{d}y < C, \,\,\,\quad \forall \,\, x\in\Omega,
\end{align}
where $\gamma<N-2$ if $N \geq 3$ and $\gamma=0$ if $N = 2$. Let $v\in H^{1,2}(\Omega,\rho)$ be such that
\begin{align}\label{eq10.25.4}
|v(x)|\leq C \int_\Omega \frac{|\nabla v(y)|}{ |x-y|^{N-1}}\mathrm{d}y,\quad\,\,\,\forall \,\, x\in\Omega.
\end{align}
Then we have
\begin{equation}\label{wp}
  \int_\Omega v^2(x) \mathrm{d}x \leq C(\Omega) \int_\Omega \rho |\nabla v(x)|^2 \mathrm{d}x,
\end{equation}
where $C(\Omega)\to 0$ if $|\Omega|\to 0$. The same inequality holds for $v\in H^{1,2}_0(\Omega,\rho)$.
\end{lem}

\begin{rem}\label{rem0}
Assume $1<p<\frac{N}{2}$ and let $u$ be a positive weak solution of \eqref{eq1.1}. By Theorems \ref{th2.1} and \ref{th2.1-}, we know that, for any $p\geq2$, $\rho:=|\nabla u|^{p-2}\in L^1(\Omega)$ for any $\Omega\subset\subset\R^N$. In addition, from \eqref{eq10.25.2} in Corollary \ref{re2333} and Remark \ref{re2334}, we deduce that \eqref{eq10.25.3} holds for $\rho=|\nabla u|^{p-2}$ and $\rho>0$ a.e.. Hence the weighted Poincar\'{e} type inequality \eqref{wp} in Lemma \ref{lm.3} holds with $\rho=|\nabla u|^{p-2}$ for $p\geq2$.
\end{rem}

\section{Proof of Theorems \ref{th2.1} and \ref{th2.1-}: the sharp estimates on asymptotic behaviors of positive solutions}\label{sc3}

In subsections \ref{sc3.1}--\ref{sc3.3}, we will first carry out our proof of Theorem \ref{th2.1} and establish the sharp estimates on asymptotic behaviors of positive weak solutions to the generalized equation \eqref{geq}. In subsection \ref{sc3.4}, we will show that $V(x):=|x|^{-2p}\ast u^{p}$ satisfies all the assumptions on $V$ in Theorem \ref{th2.1} and hence derive Theorem \ref{th2.1-} as a consequence from Theorem \ref{th2.1} immediately.

\smallskip

Since the nonnegative weak solution $u$ of \eqref{eq1.1} belongs to $D^{1,p}(\R^N)$, it follows from the Hardy-Littlewood-Sobolev inequality that $V(x):=|x|^{-2p}*u^{p} \in L^\frac{N}{p}(\R^N)$. Therefore, in order to investigate the sharp asymptotic behavior of positive solution $u$ to \eqref{eq1.1} and $|\nabla u|$, we first consider the following more general equation:
\begin{align}\label{eq10.25.1}
\left\{ \begin{array}{ll}
-\Delta_p u =V(x) u^{p-1}  \,\,\,\,&\mbox{in}\, \R^N, \\ \\
u \in D^{1,p}(\R^N), \,\,\,\,u\geq0\,\,\,\,&  \mbox{in}\, \R^N,
\end{array}
\right.\hspace{1cm}
\end{align}
where $0\leq V(x)\in L^\frac{N}{p}(\R^N)$ is a nonnegative function.

\subsection{Local integrability, boundedness and regularity}\label{sc3.1}
Inspired by \cite{CDPSYS,ED83,PT}, we can prove the local integrability, boundedness, regularity and a better preliminary estimate on decay rate in the following Lemmas \ref{lm.8} and \ref{lm.5} via the De Giorgi-Moser-Nash iteration techniques, which goes back to e.g. \cite{TMOSER}.
\begin{lem}[Local integrability, boundedness and regularity]\label{lm.8}
Let $u$ be a nonnegative solution to the generalized equation \eqref{eq10.25.1} with $p>1$ and $0\leq V(x)\in L^\frac{N}{p}(\R^N)$. Then $u\in L_{loc}^{r}(\R^N)$ for any $0<r<+\infty$. If $V\in L^{q}_{loc}(\mathbb{R}^{N})$ for some $\frac{N}{p}<q<+\infty$ provided that $u\in L^{r}_{loc}(\mathbb{R}^{N})$ for any $0<r<+\infty$, then $u\in C(\R^N)\cap L^{\infty}_{loc}(\mathbb{R}^{N})$. If $V\in L^{q}_{loc}(\mathbb{R}^{N})$ for any $\frac{N}{p}<q<+\infty$ provided that $u\in C(\R^N)\cap L^{\infty}_{loc}(\mathbb{R}^{N})$, then $u \in C^{1,\alpha}(\R^N)$ for some $0<\alpha<\min\{1,\frac{1}{p-1}\}$.
\end{lem}
\begin{proof}
Lemma \ref{lm.8} can be proved via the De Giorgi-Morse-Nash iteration method. Let $\Omega \subset\subset \R^N$ be any bounded domain. For any $q\geq1$ and $k>0$, let
$$F(t)=\left\{ \begin{array}{ll}
                t^q,\,\,\,\,&  \mbox{if}\,\, 0\leq t\leq k, \\
                qk^{q-1}t-(q-1)k^q,\,\,\,\,&  \mbox{if}\,\, t\geq k, \\
                  \end{array}
        \right.$$
and
$$G(t)=F(t)[F^\prime(t)]^{p-1}=\left\{ \begin{array}{ll}
                q^{p-1} t^{(q-1)p+1},\,\,\,\,&  \mbox{if}\,\, 0\leq t\leq k, \\
                q^{p-1} k^{(q-1)p} [qt-(q-1)k],\,\,\,\,&  \mbox{if}\,\, t\geq k. \\
                  \end{array}
        \right.$$
It is easy to see that $G$ is a piecewise $C^1$--function with only one corner at $t=k$ and
\begin{align}\label{eq10.28.1}
 G^\prime(t)=\left\{ \begin{array}{ll}
                [(q-1)p+1]q^{-1}[F^\prime (t)]^p,\,\,\,\,&  \mbox{if}\,\, 0\leq t\leq k, \\
                {[F^\prime(t)]}^p,\,\,\,\,&  \mbox{if}\,\, t>k, \\
                  \end{array}
        \right.
\end{align}
and
\begin{align}\label{eq10.28.2}
 t^{p-1} G(t)\leq q^{p-1}{[F(t)]}^p,
\end{align}
where $[(q-1)p+1]q^{-1}>1$ and $C$ is independent of $k$.

Now we choose $q>1$, 
and let
$$\xi=\eta^p G(u),$$
where $\eta\in C^\infty_0 (\Omega)$ to be determined later. Then the equation \eqref{eq10.25.1} yields
$$ \int_{\Omega} \eta^p G^\prime(u) |\nabla u|^p \mathrm{d}x + p\int_{\Omega} \eta^{p-1}| \nabla u|^{p-2} G(u) \nabla u \cdot \nabla \eta \mathrm{d}x = \int_{\Omega} V(x) u^{p-1} \eta^p G(u) \mathrm{d}x.$$
For any $\varepsilon\in(0,1)$, by direct computations and Young's inequality with $\varepsilon$, the previous identity becomes
\begin{align*}
&\quad \int_{\Omega} \eta^p [F^\prime(u)]^p |\nabla u|^p \mathrm{d}x \leq \int_{\Omega} \eta^p G^\prime(u) |\nabla u|^p \mathrm{d}x \\
&\leq p \int_{\Omega} \eta^{p-1}| \nabla u|^{p-1} [F^\prime(u)]^{p-1} F(u) |\nabla \eta| \mathrm{d}x + \int_{\Omega} V(x) u^{p-1} \eta^p G(u) \mathrm{d}x \\
&\leq \varepsilon \int_{\Omega} \eta^p [F^\prime(u)]^p |\nabla u|^p \mathrm{d}x + C_\varepsilon \int_{\Omega}  [F(u) |\nabla \eta|]^p \mathrm{d}x
+ q^{p-1}\int_{\Omega} V(x) [\eta F(u)]^p \mathrm{d}x \\
&\leq \varepsilon \int_{\Omega} \eta^p [F^\prime(u)]^p |\nabla u|^p \mathrm{d}x +  C_\varepsilon \int_{\Omega}  [F(u) |\nabla \eta|]^p \mathrm{d}x+ q^{p-1}\|V\|_{L^\frac{N}{p}(\Omega\cap supp(\eta))} \left(\int_{\Omega} [\eta F(u)]^{p^{\star}} \mathrm{d}x\right)^\frac{p}{p^{\star}},
\end{align*}
where we also have used \eqref{eq10.28.1} and \eqref{eq10.28.2}, as well as the definition of $G(u)$. Then, taking $\varepsilon=\frac{1}{2}$, we get
\begin{align*}
\int_{\Omega} \eta^p  [F^\prime(u)]^p |\nabla u|^p \mathrm{d}x \leq
C \int_{\Omega}  [F(u) |\nabla \eta|]^p \mathrm{d}x + q^{p-1}\|V\|_{L^\frac{N}{p}(\Omega\cap supp(\eta))} \left(\int_{\Omega} [\eta F(u)]^{p^{\star}} \mathrm{d}x\right)^\frac{p}{p^{\star}}.
\end{align*}
Using the Sobolev inequality, we can get
\begin{align}\label{eq10.28.3}
\left( \int_{\Omega} |\eta F(u)|^{p^{\star}}\mathrm{d}x  \right)^\frac{p}{p^{\star}} &\leq \int_{\Omega} |\eta \nabla F(u) + F(u) \nabla\eta|^p \mathrm{d}x \\
&\leq C_1 \int_{\Omega} \eta^p  [F^\prime(u)]^p |\nabla u|^p \mathrm{d}x + C_1 \int_{\Omega}  [F(u) |\nabla \eta|]^p \mathrm{d}x \nonumber\\
& \leq C \int_{\Omega}  [F(u) |\nabla \eta|]^p \mathrm{d}x + Cq^{p-1}\|V\|_{L^\frac{N}{p}(\Omega\cap supp(\eta))}\left(\int_{\Omega} [\eta F(u)]^{p^{\star}} \mathrm{d}x\right)^\frac{p}{p^{\star}}. \nonumber
\end{align}
Since $V(x)\in L^\frac{N}{p}(\R^N)$, for any $x_0\in\Omega$, we can choose $B(x_0,2R)\subsetneq\Omega$ with $0<R<1$ such that
$$C\|V\|_{L^\frac{N}{p}(B(x_0,2R))}<\frac{1}{2q^{p-1}}$$
and $supp(\eta) \subset B(x_0, 2R)$. From \eqref{eq10.28.3}, we get
\begin{align}\label{eq10.28.4}
\left( \int_{\Omega}  [\eta F(u)]^{p^{\star}} \mathrm{d}x  \right)^\frac{p}{p^{\star}}
& \leq C  \int_{\Omega}  [F(u) |\nabla \eta|]^p \mathrm{d}x.
\end{align}
Letting $k\to\infty$ and taking the limit, we conclude
\begin{align}\label{eq10.28.4}
\left( \int_{\Omega} \eta^{p^{\star}} u^{qp^{\star}} \mathrm{d}x  \right)^\frac{p}{p^{\star}}
& \leq C \int_{\Omega} |\nabla \eta|^p u^{qp} \mathrm{d}x.
\end{align}
for any $q>1$.  

Therefore, for any $R \leq h < h^\prime < 2R$, we can choose $\eta=1$ in $B(x_0,h)$, and $\eta=0$ in $\Omega\setminus B(x_0,h')$, $0\leq \eta \leq 1$ and $|\nabla\eta|<\frac{2}{h^\prime-h}$, and hence deduce from \eqref{eq10.28.4} that
\begin{align*}
\left( \int_{B(x_0,h)} u^{qp^{\star}} \mathrm{d}x  \right)^\frac{p}{p^{\star}}
\leq  \frac{C}{(h^\prime-h)^p} \int_{B(x_0,h^\prime)} u^{qp} \mathrm{d}x,\,\,\,\,\,\,\forall \,\,q>1,
\end{align*}
which implies
\begin{align}\label{eq10.28.5}
\left( \int_{B(x_0,h)} u^{qp^{\star}} \mathrm{d}x  \right)^\frac{1}{qp^{\star}}
\leq  \left( \frac{C}{h^\prime-h} \right)^\frac{1}{q} \left( \int_{B(x_0,h^\prime)} u^{qp} \mathrm{d}x  \right)^\frac{1}{qp},\,\,\,\,\,\,\forall \,\,q>1.
\end{align}

Let $\chi:=\frac{p^{\star}}{p}=\frac{N}{N-p}$, and define $r_{i}:=R+\frac{2R}{2^i}$ for $i=1,2,\cdots$. Then $r_{i}-r_{i+1}=\frac{R}{2^{i}}$. By taking $h^\prime=r_{i}$, $h=r_{i+1}$ and $q=\chi^i$ in \eqref{eq10.28.5}, we get
\begin{align}\label{eq10.28.6}
\left( \int_{B(x_0,r_{i+1})} u^{p\chi^{i+1}} \mathrm{d}x  \right)^\frac{1}{p\chi^{i+1}}
\leq  \left( \frac{C 2^{i}}{R} \right)^\frac{1}{\chi^i} \left( \int_{B(x_0,r_{i})} u^{p\chi^i} \mathrm{d}x  \right)^\frac{1}{p\chi^i}.
\end{align}
By iteration, we infer that
\begin{align}\label{eq10.28.7}
\left( \int_{B(x_0,r_{i+1})} u^{p\chi^{i+1}} \mathrm{d}x  \right)^\frac{1}{p\chi^{i+1}}
&\leq  2^\frac{i}{\chi^i} \left( \frac{C}{R} \right)^\frac{1}{\chi^i} \left( \int_{B(x_0,r_{i})} u^{p\chi^i} \mathrm{d}x  \right)^\frac{1}{p\chi^i} \\
&\leq  2^{\sum_{k=1}^i \frac{k}{\chi^k}} \left( \frac{C}{R} \right)^{\sum_{k=1}^i\frac{1}{\chi^k}} \left( \int_{B(x_0,2R)} u^{p^{\star}} \mathrm{d}x  \right)^\frac{1}{p^{\star}}. \nonumber
\end{align}
Noting that
$$ \sum_{k=1}^\infty \frac{k}{\chi^k} <\infty, $$
and
$$ \sum_{k=1}^\infty \frac{1}{\chi^k} =\frac{1}{\chi} \frac{1}{1-\frac{1}{\chi}} = \frac{N-p}{p}, $$
we can let $i$ sufficiently large (depending on $r$) in \eqref{eq10.28.7}, and obtain that, for any $p^{\ast}\leq r<+\infty$,
\begin{align}\label{eq10.28.9}
\|u\|_{L^r(B(x_0,R))}
\leq \frac{C}{R^\frac{N-p}{p}} \|u\|_{L^{p^{\star}}(B(x_0,2R))}
\leq \frac{C}{R^\frac{N-p}{p}} \|u\|_{L^{p^{\star}}(\R^N)}, \,\,\,\,\,\,\,\,\,\forall \,\, x_0\in\Omega,
\end{align}
where $R\rightarrow0$ as $r\rightarrow+\infty$, and the constant $C>0$ is independent of $r$ and $R$. So we obtain from finite coving theorem that
\begin{align}\label{eq10.28.10}
\|u\|_{L^r(\Omega)}
\leq C (|\Omega|) \, \|u\|_{L^{p^{\star}}(\R^N)}, \,\,\,\,\,\,\, \forall\,\, \Omega \subset\subset \R^N.
\end{align}
Thus $u\in L^{r}_{loc}(\mathbb{R}^{N})$ for any $0<r<+\infty$.

Due to $V\in L^{q}_{loc}(\mathbb{R}^{N})$ for some $\frac{N}{p}<q<+\infty$ provided that $u\in L^{r}_{loc}(\mathbb{R}^{N})$ for any $0<r<+\infty$, by using the local boundedness and regularity estimates (c.f. e.g. Chapter 4 in \cite{HL}, see also \cite{ED83,KM,PT,XCL}), we can deduce that $u\in C(\R^N)\cap L_{loc}^\infty(\R^N)$. Moreover, since $V\in L^{q}_{loc}(\mathbb{R}^{N})$ for any $\frac{N}{p}<q<+\infty$ provided that $u\in C(\R^N)\cap L^{\infty}_{loc}(\mathbb{R}^{N})$, by using the standard $C_{loc}^{1,\alpha}$-regularity estimates (c.f. e.g. Chapter 4 in \cite{HL}, see also \cite{ED83,KM,PT,XCL}), we can deduce that $u\in C_{loc}^{1,\alpha}(\R^N)$ for some $0<\alpha<\min\{1,\frac{1}{p-1}\}$. This completes our proof of Lemma \ref{lm.8}.
\end{proof}

\subsection{A better preliminary estimate on upper bound of asymptotic behavior}\label{sc3.2}

The following result provides a better preliminary decay estimate for solution of the generalized equation \eqref{eq10.25.1}, which is not sharp, but this estimate plays a key role in the proof of Theorem \ref{th2.1}.

\begin{lem}[A better preliminary estimate]\label{lm.5}
Let $u$ be a nonnegative solution to the generalized equation \eqref{eq10.25.1} with $p>1$ and $0\leq V(x)\in L^\frac{N}{p}(\R^N)$. Then, there exist $\tau>0$ and $\hat{p}\geq p^{\star}$ such that, for any $\bar{p}\geq \hat{p}$, there exists $R_{0}>0$ sufficiently large depending on $\bar{p}$ such that the following decay property holds:
\begin{equation}\label{a1+}
  \|u\|_{L^{\bar{p}}(\mathbb{R}^{N}\setminus B_{2R}(0))}\leq\frac{C}{R^{\frac{N-p}{p}+\tau}}, \qquad \forall \,\, R>R_{0},
\end{equation}
where the constant $C>0$ is independent of $\bar{p}$ and $R$. Assume that there exists $q>\frac{N}{p}$ and $\widetilde{R}>0$ large enough, such that the following uniform boundedness holds:
\begin{equation}\label{a2+}
  \|V\|_{L^{q}(B_{1}(x))}\leq \overline{C}, \qquad \forall \,\, |x|>2\widetilde{R},
\end{equation}
provided that the decay property \eqref{a1+} holds for sufficiently large $\bar{p}\geq \hat{p}$ and $R>R_{0}$, where the constant $\overline{C}>0$ is independent of $x$ and $\widetilde{R}$, then, there is a constant $C>0$ such that
$$ u(x) \leq\frac{C}{1+|x|^{\frac{N-p}{p}+\tau}} \quad \,\,\, \mbox{in}\,\, \R^N.$$
\end{lem}
\begin{proof}
Let $R_0>0$ large (to be determined later) and $q>1$. For any $R>r>R_0$, take $\xi\in C^2(\R^N)$, with $\xi=0$ in $B_r(0)$, $\xi=1$ in $\mathbb{R}^{N}\setminus B_R(0)$, $0 \leq \xi\leq 1$, and $|\nabla \xi| \leq \frac{2}{(R-r)}$. Let $\eta=\xi^p u^{1+p(q-1)}$. The generalized equation \eqref{eq10.25.1} implies that
$$\int_{\R^N} |\nabla u|^{p-2} \nabla u \nabla\eta \mathrm{d}x=\int_{\R^N} V(x) u^{p-1} \eta \mathrm{d}x.$$
First, by H\"{o}lder's inequality, one has
\begin{align}\label{eq10.26.1}
\int_{\R^N} V(x) u^{p-1} \eta \mathrm{d}x \leq \left( \int_{\R^N\setminus B_{R_0}} |V|^\frac{N}{p} \mathrm{d}x  \right)^\frac{p}{N}  \left( \int_{\R^N \setminus B_{R_0}} |\xi u^q|^{p^{\star}} \mathrm{d}x  \right)^\frac{p}{p^{\star}}.
\end{align}
On the other hand,
\begin{align}\label{eq10.26.2}
&\quad \int_{\R^N}|\nabla u|^{p-2} \nabla u \nabla\eta \mathrm{d}x\\
           &= (1+p(q-1))\int_{\R^N \setminus B_{R_0}} |\nabla u|^p \xi^p u^{p(q-1)}  \mathrm{d}x  + p \int_{\R^N \setminus B_{R_0}} \xi^{p-1} u^{1+p(q-1)} |\nabla u|^{p-2} \nabla u \nabla\xi \mathrm{d}x \nonumber\\
           &\geq \frac{C_1}{q^{p-1}} \int_{\R^N} |\nabla(\xi u^q)|^p \mathrm{d}x - C\int_{\R^N} |\nabla \xi|^p u^{pq} \mathrm{d}x \nonumber\\
           &\geq \frac{C_2}{q^{p-1}} \left(\int_{\R^N \setminus B_{R_0}} |\xi u^q|^{p^{\star}} \mathrm{d}x \right)^\frac{p}{p^{\star}} - C\int_{\R^N \setminus B_{R_0}} |\nabla \xi|^{p}u^{pq} \mathrm{d}x. \nonumber
\end{align}
Combining \eqref{eq10.26.1} with \eqref{eq10.26.2}, we have
\begin{align}\label{eq10.26.3}
&\frac{C_2}{q^{p-1}} \left(\int_{\R^N \setminus B_{R_0}} |\xi u^q|^{p^{\star}} \mathrm{d}x \right)^\frac{p}{p^{\star}} - C\int_{\R^N \setminus B_{R_0}} |\nabla \xi|^p u^{pq} \mathrm{d}x \\
&\leq \left( \int_{\R^N\setminus B_{R_0}} |V|^\frac{N}{p} \mathrm{d}x  \right)^\frac{p}{N}  \left( \int_{\R^N \setminus B_{R_0}} |\xi u^q|^{p^{\star}} \mathrm{d}x  \right)^\frac{p}{p^{\star}}. \nonumber
\end{align}
Take $R_0$ sufficiently large such that
$$\left( \int_{\R^N\setminus B_{R_0}} |V|^\frac{N}{p} \mathrm{d}x  \right)^\frac{p}{N} \leq \frac{C_2}{2q^{p-1}},$$
we obtain that
\begin{align}\label{eq10.26.4}
\left(\int_{\R^N \setminus B_{R_0}} |\xi u^q|^{p^{\star}} \mathrm{d}x \right)^\frac{p}{p^{\star}} \leq C q^{p-1} \int_{\R^N \setminus B_{R_0}} |\nabla \xi|^p u^{pq} \mathrm{d}x,
\end{align}
which implies
\begin{align}\label{eq10.26.8}
\left(\int_{\R^N \setminus B_{R}} |u|^{qp^{\star}} \mathrm{d}x \right)^\frac{1}{q p^{\star}} \leq C q^\frac{p-1}{pq} \left(\frac{2}{R-r}\right)^\frac{1}{q} \left(\int_{\R^N \setminus B_{r}} u^{pq} \mathrm{d}x\right)^\frac{1}{pq},
\end{align}
for any $R>r>R_0$.

Let $\chi=\frac{p^{\star}}{p}=\frac{N}{N-p}>1$. For any $R>R_0$, define $r_i = 2R - \frac{2R}{2^i}$ for $i=1,2,\cdots$. Then $r_{i+1}-r_i= \frac{2R}{2^i} -\frac{2R}{2^{i+1}}=\frac{R}{2^i}$. Taking $R=r_{i+1}$, $r=r_i$ and $q=\chi^i$ in \eqref{eq10.26.8}, we derive
\begin{align}\label{eq10.26.5}
\left(\int_{\R^N \setminus B_{r_{i+1}}} |u|^{\chi^{i+1} p} \mathrm{d}x \right)^\frac{1}{\chi^{i+1} p}
& \leq C (\chi^i)^\frac{p-1}{p\chi^i} \left(\frac{2}{r_{i+1}-r_i}\right)^\frac{1}{\chi^i} \left(\int_{\R^N \setminus B_{r_i}} u^{p \chi^i} \mathrm{d}x\right)^\frac{1}{p\chi^i} \\
& \leq C \chi^{ \frac{p-1}{p} \frac{i}{\chi^i} } \left(\frac{2^i}{R}\right)^\frac{1}{\chi^i} \left(\int_{\R^N \setminus B_{r_i}} u^{p \chi^i} \mathrm{d}x\right)^\frac{1}{p\chi^i} \nonumber
\end{align}
for any $R>R_0$. By iteration, we can obtain from \eqref{eq10.26.5} that
\begin{align}\label{eq10.26.6}
\left(\int_{\R^N \setminus B_{r_{i+1}}} |u|^{\chi^{i+1} p} \mathrm{d}x \right)^\frac{1}{\chi^{i+1} p}
& \leq C \prod^i_{k=1} \left[ \chi^{ \frac{p-1}{p} \frac{k}{\chi^k} } \left(\frac{2^k}{R}\right)^\frac{1}{\chi^k} \right] \left(\int_{\R^N \setminus B_{R}} u^{p \chi} \mathrm{d}x\right)^\frac{1}{p\chi} \\
& \leq C   \left( 2 \chi^\frac{p-1}{p}\right)^{\sum\limits^i_{k=1}\frac{k}{\chi^k} } \left(\frac{1}{R}\right)^{\sum\limits^i_{k=1}\frac{1}{\chi^k}}
\left(\int_{\R^N \setminus B_{R}} u^{p^{\star}} \mathrm{d}x\right)^\frac{1}{p^{\star}} \nonumber
\end{align}
for any $R>R_0$. Since
$$\sum_{i=1}^\infty \frac{k}{\chi^k} < + \infty \qquad \text{and} \qquad \sum_{k=1}^i \frac{1}{\chi^k} < \sum_{k=1}^\infty \frac{1}{\chi^k}=\frac{1}{\chi}\frac{1}{1-\chi}=\frac{N-p}{p},$$
we have proved that for any $\overline{p}>p^{\star}$, there is a $R_0>0$ large, depending on $\overline{p}$, such that
\begin{align}\label{eq10.26.7}
\| u \|_{L^{\overline{p}}(\R^N \setminus B_{2R})} \leq \frac{C}{R^{ \frac{N-p}{p}-o_{\overline{p}}(1)} }
\| u \|_{L^{p^{\star}}(\R^N \setminus B_{R})},\,\quad\,\,\, \,\forall \,\, R>R_0,
\end{align}
where $o_{\overline{p}}(1)\to 0$ as $\overline{p} \to \infty$.

\smallskip

Next we estimate $\| u \|_{L^{p^{\star}}(\R^N \setminus B_{R})}$.

Take $\eta=\xi^p u$, where $\xi$ satisfies $\xi=0$ in $B_R$, $\xi=1$ in $\R^N \setminus B_{2R}$, $0\leq \xi\leq 1$, and $|\nabla \xi| \leq 2R^{-1}$. Then, similar to \eqref{eq10.26.4}, we can get
\begin{align}\label{eq10.26.9}
\left(\int_{\R^N } |\xi u|^{p^{\star}} \mathrm{d}x \right)^\frac{1}{p^{\star}}
&\leq C  \left( \int_{\R^N} |\nabla \xi|^p u^{p} \mathrm{d}x \right)^\frac{1}{p} \\
&= C \left( \int_{B_{2R}\setminus B_R} |\nabla \xi|^p u^{p} \mathrm{d}x \right)^\frac{1}{p} \leq C \left( \frac{1}{R^p} \int_{B_{2R}\setminus B_R} u^{p} \mathrm{d}x \right)^\frac{1}{p} \nonumber\\
& \leq C \left[\frac{1}{R^p}\left(\int_{B_{2R}\setminus B_R} 1 \mathrm{d}x\right)^\frac{p}{N}\left(\int_{B_{2R}\setminus B_R} u^{p\frac{N}{N-p}} \mathrm{d}x\right)^\frac{N-p}{N} \right]^\frac{1}{p} \nonumber\\
&=C \left(\int_{B_{2R}\setminus B_R} u^{p^{\star}} \mathrm{d}x \right)^\frac{1}{p^{\star}}. \nonumber
\end{align}
Thus, we have
\begin{align}\label{eq10.26.10}
\int_{\R^N \setminus B_{2R}} |u|^{p^{\star}} \mathrm{d}x
&\leq \int_{\R^N } |\xi u|^{p^{\star}} \mathrm{d}x \\
&\leq C \int_{B_{2R}\setminus B_R} u^{p^{\star}} \mathrm{d}x = C \int_{\R^N \setminus B_R} u^{p^{\star}} \mathrm{d}x - C \int_{\R^N \setminus B_{2R}} u^{p^{\star}} \mathrm{d}x, \nonumber
\end{align}
which implies
\begin{align}\label{eq10.26.11}
\int_{\R^N \setminus B_{2R}} u^{p^{\star}} \mathrm{d}x
&\leq  \frac{C}{C+1} \int_{\R^N \setminus B_R} u^{p^{\star}} \mathrm{d}x.
\end{align}

Set $0<\rho:=\frac{C}{C+1}<1$, and let $\Psi(R)=\| u \|_{L^{p^{\star}}(\R^N \setminus B_R)}$. So by \eqref{eq10.26.11},
\begin{align}\label{eq10.26.12}
\Psi(2R)\leq \rho\Psi(R),\,\,\quad \forall \,\, R\geq R_0,
\end{align}
from which we get
\begin{align}\label{eq10.26.13}
\Psi(2^k R_0)\leq \rho^k \Psi(R_0).
\end{align}
For any $|x|\geq R_0$, there is a $k\in \N^+$, such that
$$ 2^k R_0 \leq |x| \leq 2^{k+1} R_0,$$
which combined with the definition of $\Psi(R)$ and \eqref{eq10.26.13} yield that
$$ \Psi(|x|) \leq \Psi(2^k R_0)\leq \rho^k \Psi(R_0) \leq \rho^{\log_2 |x|-\log_2 (2R_{0})} \Psi(R_0).$$
Since
$$ \rho^{\log_2 |x|}= 2^{\log_2\rho \cdot \log_2 |x|} = |x|^{\log_2\rho},$$
we obtain
$$ \Psi(|x|) \leq \frac{C}{|x|^{\log_2 \frac{1}{\rho}}},$$
where $\log_2 \frac{1}{\rho}>0$. Thus, we have shown that there exists a $\tau:=\log_2 \frac{1}{\rho}>0$ independent of $\overline{p}>p^{\star}$ such that
$$ \Psi(|x|) \leq \frac{C}{|x|^\tau}.$$

By choosing $\hat{p}>p^{\star}$ sufficiently large such that $o_{\overline{p}}(1)<\frac{\tau}{2}$ for any $\bar{p}\geq\hat{p}$, we obtain from \eqref{eq10.26.7} that
\begin{align}\label{eq10.26.17}
\| u \|_{L^{\overline{p}}(\R^N \setminus B_{2R}(0))} \leq \frac{C}{R^{ \frac{N-p}{p}+\frac{\tau}{2}}},\,\quad\,\forall \,\,R>R_0,
\end{align}
where the constant $C>0$ is independent of $\bar{p}$ and $R$. That is, the decay property in \eqref{a1+} holds. Consequently, the uniform boundedness in \eqref{a2+} holds, i.e., there exists $q>\frac{N}{p}$ and $\widetilde{R}>0$ large enough, such that
\begin{equation*}
  \|V\|_{L^{q}(B_{1}(x))}\leq \overline{C}, \qquad \forall \,\, |x|>2\widetilde{R},
\end{equation*}
where the constant $\overline{C}>0$ is independent of $x$ and $\widetilde{R}$. Therefore, for any $x$ such that $|x|=4R$ with $R>\max\{R_0,\widetilde{R}\}>1$, by the local boundedness estimates (c.f. e.g. Chapter 4 in \cite{HL}, see also \cite{ED83,KM,PT,XCL}) and the decay property in \eqref{eq10.26.17}, we have, for any fixed $\bar{p}\geq\hat{p}$ large enough,
$$ |u(x)|\leq \max_{y\in B_{\frac{1}{2}}(x)} |u(y)| \leq \widetilde{C}\|u\|_{L^{\overline{p}}(B_1(x))}\leq \widetilde{C}\|u\|_{L^{\overline{p}}(\R^N \setminus B_{2R}(0))} \leq \frac{\widetilde{C}C}{R^{ \frac{N-p}{p}+\frac{\tau}{2}}},$$
where the constants $C$ and $\widetilde{C}=\widetilde{C}(N,p,\overline{C})$ are independent of $x$ and $R$. Combining this with Lemma \ref{lm.8} finish our proof of Lemma \ref{lm.5}.
\end{proof}

\subsection{Completion of the proof of Theorem \ref{th2.1}}\label{sc3.3}

For any nonnegative weak solution $u$ of the generalized equation \eqref{eq10.25.1}, we have proved in Lemmas \ref{lm.8} and \ref{lm.5} that $u\in C^{1,\alpha}(\mathbb{R}^{N})\cap L^\infty(\R^N)$ and $u$ satisfies a better preliminary estimate under appropriate assumptions on $V$. Moreover, it follows from the strong maximum principle (see Lemma \ref{hopf}) that any nonnegative nontrivial solution of \eqref{eq1.1} is actually strictly positive provided that $V\in L^{\infty}_{loc}(\mathbb{R}^{N})$.

\smallskip

Now, we are ready to prove the sharp asymptotic estimates on any positive weak solution $u$ to the generalized equation \eqref{eq10.25.1} and $|\nabla u|$ in Theorem \ref{th2.1}. The proof of the sharp asymptotic estimates in Theorem \ref{th2.1} will be divided into the following three steps.

\medskip

\noindent{\bf Step 1.} We will show that there exists some constants $c_0,C_0>0$ such that
\begin{align}\label{eq10.26.57}
\frac{c_0}{1+|x|^\frac{N-p}{p-1}} \leq u(x) \leq \frac{C_0}{1+|x|^\frac{N-p}{p-1}} \,\,\,\,\,\,\,\,\mbox{in}\,\,\R^N.
\end{align}

We will show \eqref{eq10.26.57} by applying the following two key Lemmas \ref{estimate1}-\ref{estimate2}.
 \begin{lem}[Theorem 1.5 in \cite{XCL}]\label{estimate1}
Let $\Omega$ be an exterior domain in $\R^N$ such that $\Omega^c=\R^N\setminus \Omega$ is bounded and
function $V(x) \in L^\frac{N}{p}(\Omega)$
satisfying $V(x) \leq C_{V} |x|^{-\beta}$ in $\Omega$ for some positive constants $C_{V}$ and $\beta>p$. If $u\in D^{1,p}(\Omega)$ is a weak subsolution to
$$-\Delta_p u =V(x) u^{p-1}  \,\,\,\,\,\,\, \mbox{in}\,\, \Omega$$
with $1<p<N$, then there exists a positive constant $C=C(N,p,\beta,C_{V})>0$ such that
$$ u(x)\leq C M |x|^\frac{N-p}{p-1}\,\,\,\,\,\,\,\,\mbox{in}\,\,\, \R^N\setminus B_{R_1}(0),$$
where $M = \sup\limits_{\partial B_{R_1}(0)} u^+ $ and $R_1 > 1$ is a constant depending on $N, p, C_{V}, \beta$.
\end{lem}

\begin{lem}[Theorem 2.3 in \cite{SJZH02}]\label{estimate2}
Let $u$ be a weak solution of
$$-\Delta_p u \geq 0 \qquad \mbox{in}\,\, \R^N\setminus \Omega$$
with $\Omega$ compact and $1<p<N$. Then there exist positive constants $m = \inf\limits_{\partial (\R^N\setminus\Omega)} u$ and $C = \inf\limits_{\partial (\R^N\setminus\Omega)} |x|^\frac{N-p}{p-1}$ such that
$$ u(x)\geq C m |x|^\frac{N-p}{p-1} \qquad \mbox{in}\,\,\,\, \R^N\setminus\Omega.$$
%
\end{lem}

By the better preliminary estimate in Lemma \ref{lm.5} and the assumption ``$V(x) \leq C_V |x|^{-\beta}$ for some $C_V>0$, $\beta>p$ and $|x|$ large provided that $u(x)\leq C|x|^{-\gamma}$ for some $C>0$, $\gamma>\frac{N-p}{p}$ and $|x|$ large" on $V(x)$ in Theorem \ref{th2.1}, one has $V(x) \leq C_{V} |x|^{-\beta}$ for some $C_{V}>0$, $\beta>p$ and $|x|$ large. Thus we deduce the desired sharp upper bound in \eqref{eq10.26.57} from Lemma \ref{estimate1} and the sharp lower bound in \eqref{eq10.26.57} from Lemma \ref{estimate2}, respectively. This proves the sharp asymptotic estimates \eqref{eq0806} for $u$ in Theorem \ref{th2.1}.

\medskip

Next, in Steps 2 and 3, we will prove the sharp asymptotic estimates \eqref{eq0806+} for $|\nabla u|$ in Theorem \ref{th2.1} by using the arguments in \cite{BS16,VJ16}.

\smallskip

\noindent{\bf Step 2.}
We are to show that, for some constant $C_0>0$,
$$ |\nabla u(x)| \leq \frac{C_0}{|x|^\frac{N-1}{p-1}}\,\,\,\,\,\,\,\,\mbox{for}\,\,|x|>R_0\geq 1.$$

For any $R>0$ and $y\in \R^N$, we define
$$ u_R (y): =R^\frac{N-p}{p-1}u(Ry).$$
By the generalized equation \eqref{geq}, we deduce that
\begin{align}\label{eq2901}
-\Delta_p u_R(y)
=-R^N \Delta_p u(x)=R^p V(Ry)  u^{p-1}_R(y) \,\,\,\,\,\,\,\mbox{in}\,\,\R^N.
\end{align}
It follows from $V(x) \leq C_{V} |x|^{-\beta}$ with $\beta>p$ that
\begin{align}\label{eq2902}
-\Delta_p u_R(y)
&=R^p V(Ry) u^{p-1}_R(y) \leq \frac{C_{V}}{R^{\beta-p} |y|^{\beta}} u^{p-1}_R(y) \, \,\,\,\,\,\,\,\mbox{for}\,\,R>R_0\,\,\,\mbox{and}\,\,|y|>1,
\end{align}
which implies that
\begin{align}\label{eq2903}
-\Delta_p u_R(y)
\leq C_{V} u^{p-1}_R(y) \,\,\,\,\,\,\,\,\,\,\mbox{in}\,\,\R^N\setminus B_1.
\end{align}

On the other hand, according to \eqref{eq10.26.57}, we have, for any $R\geq R_0$,
\begin{align}\label{eq2905}
u_R (y) =R^\frac{N-p}{p-1} u(Ry)\leq C_0 R^\frac{N-p}{p-1} |Ry|^{-\frac{N-p}{p-1}}= C_0 |y|^{-\frac{N-p}{p-1}} \leq C_0 \,\,\,\,\,\,\mbox{in}\,\,\R^N\setminus B_1.
\end{align}
By \eqref{eq2903}--\eqref{eq2905} and the regularity estimates in DiBenedetto \cite{ED83} and Tolksdorf \cite{PT}, we have
\begin{align}\label{eq2906}
\|\nabla u_R\|_{L^\infty (B(0,4) \setminus B(0,2))}\leq C_1
\end{align}
for some constant $C_1>0$ depending on $C_{0}$. Finally, for any $x\in\R^N\setminus B(0,3R)$, by applying \eqref{eq2906} with $R=\frac{|x|}{3}$ and $y\in B(0,4) \setminus B(0,2)$, we obtain
\begin{align}\label{eq2907}
|\nabla_x u(x)|= R^\frac{1-N}{p-1} |\nabla_y u_R(y) | \leq C_1 |x|^\frac{1-N}{p-1}\,\,\,\,\,\,\,\mbox{for}\,\,|x|\geq 3 R_0.
\end{align}

\smallskip

\noindent{\bf Step 3.}
We will show that, for some constant $c_0$,
$$ |\nabla u(x)| \geq \frac{c_0}{|x|^\frac{N-1}{p-1}} \quad \,\,\mbox{for}\,\,|x|>R_0.$$

The proof will be carried out by contradiction arguments. We assume on the contrary that there exist sequences of radii $R_n$ and points $x_n\in\R^N$ with $R_n\to\infty$ as $n\to\infty$ and $|x_n|=R_n$, such that
\begin{align}\label{eq2908}
|\nabla u(x_n)|\leq \frac{\theta_n}{R_n^\frac{N-1}{p-1}}
\end{align}
with $\theta_n\to 0$ as $n\to\infty$. For $0<a<A$ fixed, we set
$$ u_{R_n} (x): =R_n^\frac{N-p}{p-1}   u(R_n  x), \qquad \forall \,\, x\in\overline{B_A \setminus B_a}.$$
It follows from \eqref{eq10.26.57} that for any $n$ large enough such that $|R_n x|\geq R_{n}a>R_0$, we have
\begin{align}\label{eq2910}
\frac{c_0}{|x|^\frac{N-p}{p-1}}   \leq    u_{R_n} (x)=R_n^\frac{N-p}{p-1} u(R_n  x) \leq \frac{C_0}{|x|^\frac{N-p}{p-1}},
\end{align}
which implies that
\begin{align}\label{eq2911}
\frac{c_0}{A^\frac{N-p}{p-1}}   \leq    u_{R_n} (x) \leq \frac{C_0}{a^\frac{N-p}{p-1}}\,\,\,\,\,\,\,\,\mbox{in}\,\,\overline{B_A \setminus B_a},
\end{align}
and in particular,
\begin{align}\label{eq2912}
u_{R_n} (x) \leq \frac{C_0}{A^\frac{N-p}{p-1}}\,\,\,\,\,&\mbox{on}\,\,\partial B_A, \\
u_{R_n} (x) \geq \frac{c_0}{a^\frac{N-p}{p-1}}\,\,\,\,\,&\mbox{on}\,\,\partial B_a. \nonumber
\end{align}

On the other hand, similar to \eqref{eq2902}, we obtain
\begin{align}\label{eq2913}
0\leq-\Delta_p u_{R_n} (x) \leq \frac{C_{V}}{a^{\beta} {R_n}^{\beta-p}} u^{p-1}_{R_{n}}(x) \leq C_{V} u^{p-1}_{R_n}(x) \,\,\,\,\,\,\,\mbox{in}\,\,\overline{B_A \setminus B_a}
\end{align}
for any $n$ large enough such that $|R_n x|\geq R_{n}a>R_0$ and $a^{\beta} {R_n}^{\beta-p}\geq1$. Therefore, from \eqref{eq2913} and \eqref{eq2911}, by the regularity results in \cite{PT}, we get that
$$ \|u_{R_n}\|_{C^{1,\alpha} (K)}<+\infty \,\,\,\,\quad \mbox{for}\,\,0<\alpha<1$$
for any compact set $K\subset B_A \setminus B_a$. Without loss of generality, we may redefine $a$ and $A$ and assume that the $C^{1,\alpha}$ estimates hold in $\overline{B_A \setminus B_a}$. Hence, up to subsequences, we have
\begin{align}\label{eq2914}
u_{R_n} (x) \to  u_{a,A}(x),\,\quad \mbox{as}\,\,n\to\infty \,\,\,\,\,\quad\,\, \mbox{in}\,\,C^{1,\alpha^\prime}(B_A \setminus B_a)
\end{align}
for $0<\alpha^\prime<\alpha$. From \eqref{eq2913}, we deduce that
\begin{align}\label{eq2915}
-\Delta_p u_{a,A}(x)=0\,\,\,\,\,\,\,\mbox{in}\,\,B_A \setminus \overline{B_a}.
\end{align}
Now, for every $j\in\N^+$, let $a_j=\frac{1}{j}$ and $A_j=j$ and we constructs $u_{a_j,A_j}$ by reasoning as above. Then, as $j\to\infty$, a diagonal argument implies that there exist a limiting profile $u_\infty$ such that
$$ u_\infty \equiv u_{a_j,A_j} \,\,\,\,\,\,\,\mbox{in}\,\, B_{A_j} \setminus B_{a_j}, \quad \forall \,\,j\in\N^+$$
In particular,
\begin{align}\label{eq2916}
-\Delta_p u_\infty (x)=0\,\,\,\,\,\,\,\,\mbox{in}\,\,\R^N \setminus \{ 0 \}.
\end{align}
From \eqref{eq2912}, which holds with $a=a_j$ and $A=A_j$, it follows that
$$ \lim_{|x|\to\infty} u_\infty (x)=0\,\,\,\,\quad\mbox{and}\,\,\,\,\quad\lim_{|x|\to 0} u_\infty (x)= \infty.$$
Thus it follows from the uniqueness up to multipliers of $p$-harmonic maps in $\R^N\setminus \{0\}$ under suitable conditions at zero and at infinity in \cite[Theorem 2.1]{BS16} that
$$ u_\infty (x)= \frac{C}{|x|^\frac{N-p}{p-1}}$$
for some constants $C$.
Let the sequence $x_n$ be the same as in \eqref{eq2908} and set $y_n=\frac{x_n}{R_n}$. Then, by \eqref{eq2908}, it follows that
$$|\nabla u_{R_n}(y_n)| ={R_n}^\frac{N-1}{p-1} |\nabla u(x_n)|  \leq \theta_n \to 0,\quad\,\,\,\,\mbox{as}\,\,n\to\infty.$$
Since $|y_n|=1$, up to subsequences, we have that $y_n\to\overline{y} \in \partial B_1$.
Consequently, by the uniform convergence of the gradient indicated by \eqref{eq2914}, one has
$$ \nabla u_\infty(\overline{y})=0,$$
which is a contradiction since the fundamental solution $u_{\infty}$ has no critical points. This establishes the sharp asymptotic estimates for $|\nabla u|$ and hence concludes our proof of Theorem \ref{th2.1}.

\subsection{Proof of Theorem \ref{th2.1-}}\label{sc3.4}
Theorem \ref{th2.1-} is an immediate consequence of Theorem \ref{th2.1}. We only need to verify that $V(x):=|x|^{-2p}\ast u^{p}$ satisfies all the assumptions on $V$ in Theorem \ref{th2.1}.

\medskip

Indeed, first, if $u\in L^{r}_{loc}(\mathbb{R}^{N})$ for any $0<r<+\infty$, then we can deduce from the Hardy-Littlewood-Sobolev inequality and the decay property in \eqref{a1} that, for any $q>\frac{N}{p}$ large enough, there exists $\bar{p}\geq\hat{p}\geq p^{\star}$ such that, for any $R>R_{0}$,
\begin{eqnarray}\label{1}
  && \|V\|_{L^{q}(B_{R}(0))}\leq C\|u\|^{p}_{L^{\bar{p}}(\mathbb{R^{N}})}\leq C\left\{\|u\|^{p}_{L^{\bar{p}}(B_{2R}(0))}+\|u\|^{p}_{L^{\bar{p}}(\mathbb{R}^{N}\setminus B_{2R}(0))}\right\} \\
 \nonumber  && \qquad\qquad\quad\,\,\, \leq C\left\{\|u\|^{p}_{L^{\bar{p}}(B_{2R}(0))}+\left(\frac{C}{R^{\frac{N-p}{p}+\tau}}\right)^{p}\right\}<+\infty,
\end{eqnarray}
that is, $V\in L^{q}_{loc}(\mathbb{R}^{N})$ for any $\frac{N}{p}<q<+\infty$.

\medskip

Second, assume that the decay property \eqref{a1} holds for sufficiently large $\bar{p}\geq \hat{p}$ and $R>R_{0}>1$, then by the Hardy-Littlewood-Sobolev inequality and the H\"{o}lder's inequality, we have, for any given $q>\frac{N}{p}$ large enough and $\widetilde{R}=4R_{0}$,
\begin{eqnarray}\label{2}
  &&\quad \|V\|_{L^{q}(B_{1}(x))}\leq \||x|^{-2p}\ast (u^{p}\chi_{\mathbb{R}^{N}\setminus B_{4R_{0}}(0)})\|_{L^{q}(B_{1}(x))}+\||x|^{-2p}\ast (u^{p}\chi_{B_{4R_{0}}(0)})\|_{L^{q}(B_{1}(x))} \\
 \nonumber && \leq C\|u\|^{p}_{L^{\bar{p}}(\mathbb{R}^{N}\setminus B_{4R_{0}}(0))}+\left(|B_{1}(0)|\right)^{\frac{1}{q}}R_{0}^{-2p}\|u\|^{p}_{L^{p}(B_{4R_{0}}(0))} \\
 \nonumber  && \leq C\left\{\left(\frac{C}{(2R_{0})^{\frac{N-p}{p}+\tau}}\right)^{p}+R_{0}^{-p}\|u\|^{p}_{L^{p^{\star}}(\mathbb{R}^{N})}\right\}\leq CR_{0}^{-p}\leq C, \qquad\quad \forall\,\, |x|>2\widetilde{R},
\end{eqnarray}
where the constant $C>0$ is independent of $x$ and $\widetilde{R}$. That is, the uniform boundedness in \eqref{a2} holds.

\medskip

Finally, assume that $u(x)\leq C|x|^{-\gamma}$ for some $C>0$, $\gamma>\frac{N-p}{p}$ and $|x|$ large (actually, the better preliminary estimate $u(x)\leq\frac{C}{1+|x|^{\frac{N-p}{p}+\tau}}$ with $\tau>0$ in Lemma \ref{lm.5} indicates that this assumption holds for $\gamma=\frac{N-p}{p}+\tau$), then the basic estimates on the convolution $V(x)=|x|^{-2p}\ast u^{p}$ in (ii) of Lemma \ref{lm.6} (with $\sigma=2p$ and $q=p$) implies that, for $|x|>R$ large,
\begin{equation}\label{V}
  V(x)=|x|^{-2p}\ast u^{p}\leq C_{\star}|x|^{-\beta} \qquad \text{with} \qquad \beta:=p\gamma+2p-N=p+\tau p>p.
\end{equation}
Thus Theorem \ref{th2.1} can be applied to equation \eqref{eq1.1} and hence positive weak solution $u$ to \eqref{eq1.1} satisfies the regularity result and the sharp asymptotic estimates \eqref{eq0806} and \eqref{eq0806+} in Theorem \ref{th2.1}. This finishes our proof of Theorem \ref{th2.1-}.

\section{Proof of Theorem \ref{th2}: radial symmetry and monotonicity of solutions}\label{sc4}

Thanks to the sharp asymptotic estimates \eqref{eq0806} and \eqref{eq0806+} in Theorems \ref{th2.1} and \ref{th2.1-}, we are ready to prove Theorem \ref{th2}. In this section, using the method of moving planes, we will show the radial symmetry and strictly radial monotonicity of positive $D^{1,p}$-weak solutions to the $D^{1,p}$-critical nonlocal quasi-linear Schr\"{o}dinger-Hartree equation \eqref{eq1.1} by discussing the singular elliptic case $1<p<2$ and the degenerate elliptic case $p>2$ separately. So we will divide this section into two sub-sections.

\medskip

Before carrying out the moving planes technique to prove Theorem \ref{th2}, we need some standard notations.

\smallskip

Let $\nu$ be any direction in $\R^N$, i.e. $\nu\in \R^N$ and $|\nu|=1$, for arbitrary $\lambda\in\R$, we define
$$ \Sigma_\lambda^\nu : =\{ x\in\R^N : x \cdot \nu < \lambda \}, $$
$$ T_\lambda^\nu : =\{ x\in\R^N : x \cdot \nu = \lambda \},$$
$$ u^\nu_\lambda (x) := u(x^\nu_\lambda), \quad\,\,\,\,\forall \,\, x\in\R^N,$$
and define the reflection $R^\nu_\lambda$ w.r.t. the hyperplane $T^\nu_\lambda$ by
$$ x_\lambda^\nu : =R^\nu_\lambda(x)= x + 2(\lambda-x\cdot \nu)\nu,\quad\,\,\,\,\forall \,\, x\in\R^N.$$
We denote by $Z_u$ the critical set of $u$ defined by $Z_u=\{ x\in\R^N \mid |\nabla u (x)|=0 \}$ and define
\begin{equation}\label{Z}
  Z_\lambda^\nu :=\{ x\in\R^N \mid |\nabla u (x)|=|\nabla u_\lambda^\nu (x)|=0 \} \subseteq  Z_u.
\end{equation}

\smallskip

Next, we show the following basic property of $V(x)=|x|^{-2p}\ast u^p(x)$.
\begin{lem}\label{lm.1}
For any direction $\nu$ in $\R^N$, one has
\begin{align}\label{eq0805}
V(x)-V(x^{\nu}_\lambda)&=\int_{\R^N} \frac{u^p(y)}{|x-y|^{2p}}\mathrm{d}y-\int_{\R^N} \frac{u^p(y)}{|x^{\nu}_\lambda-y|^{2p}}\mathrm{d}y \\
&=\int_{\Sigma^{\nu}_\lambda}\left(\frac{1}{|x-y|^{2p}}-\frac{1}{|x^{\nu}_\lambda-y|^{2p}}\right)\left[u^p(y)-(u^{\nu}_\lambda)^p(y)\right]\mathrm{d}y. \nonumber
\end{align}
\end{lem}
\begin{proof}
Since $|x-y^{\nu}_\lambda|=|x^{\nu}_\lambda-y|$ and $|x-y|=|x^{\nu}_\lambda-y^{\nu}_\lambda|$, we have
\begin{align*}
V(x)=\int_{\R^N} \frac{u^p(y)}{|x-y|^{2p}}\mathrm{d}y
&=\int_{\Sigma^{\nu}_\lambda} \frac{u^p(y)}{|x-y|^{2p}}\mathrm{d}y+\int_{\R^N \setminus \Sigma^{\nu}_\lambda} \frac{u^p(y)}{|x-y|^{2p}}\mathrm{d}y \nonumber\\
&=\int_{\Sigma^{\nu}_\lambda} \frac{u^p(y)}{|x-y|^{2p}}\mathrm{d}y+\int_{\Sigma^{\nu}_\lambda} \frac{\left(u^{\nu}_\lambda\right)^p(y)}{|x^{\nu}_\lambda-y|^{2p}}\mathrm{d}y,
\end{align*}
and
\begin{align*}
V(x^{\nu}_\lambda)=\int_{\R^N} \frac{u^p(y)}{|x^{\nu}_\lambda-y|^{2p}}\mathrm{d}y
&=\int_{\Sigma^{\nu}_\lambda} \frac{u^p(y)}{|x^{\nu}_\lambda-y|^{2p}}\mathrm{d}y+\int_{\R^N \setminus \Sigma^{\nu}_\lambda} \frac{u^p(y)}{|x^{\nu}_\lambda-y|^{2p}}\mathrm{d}y \nonumber\\
&=\int_{\Sigma^{\nu}_\lambda} \frac{u^p(y)}{|x^{\nu}_\lambda-y|^{2p}}\mathrm{d}y+\int_{\Sigma^{\nu}_\lambda} \frac{\left(u^{\nu}_\lambda\right)^p(y)}{|x-y|^{2p}}\mathrm{d}y.
\end{align*}
Combining the above two formulae yields \eqref{eq0805}. This finishes the proof of Lemma \ref{lm.1}.
\end{proof}

\subsection{The degenerate elliptic case $2<p<\frac{N}{2}$}\label{sc4.1}

Since equation \eqref{eq1.1} is invariant under any rotations and translations, we may fix $\nu=e_1=(1,0,\cdots,0)$ and use the notations $\Sigma_\lambda: =\Sigma_\lambda^{e_1}$, $T_\lambda:=T_\lambda^{e_1}$ and $x_\lambda :=x_\lambda^{e_1}$ for the sake of simplicity.

Next, we define
$$ \Lambda:=\{ \lambda \in \R \mid u\leq u_\mu \,\,\, \mbox{in}\,\,\Sigma_\mu,\,\,\,\forall \,\, \mu \leq \lambda \},$$
and, if $\Lambda\neq\emptyset$, we set
$$\lambda_0 := \sup \Lambda.$$

Now we will prove the radial symmetry and strictly radial monotonicity in Theorem \ref{th2} for $2<p<\frac{N}{2}$.
\begin{proof}
We start moving planes from the left-hand side of the $e_1$-direction until its limiting position. The moving planes procedure consists of three main steps.

\smallskip

{\bf Step 1.} We will prove that $u\leq u_\lambda$ for $\lambda<0$ with $|\lambda|$ large, and hence $\Lambda\neq\emptyset$.

\smallskip

We may choose $(u-u_\lambda)^+ \chi_{\Sigma_\lambda} $ as the test function in the definition of $D^{1,p}(\R^N)$-weak solution of equation \eqref{eq1.1} (see Definition \ref{weak}) and get
$$ \int_{\Sigma_\lambda} |\nabla u|^{p-2} \langle \nabla u,\nabla(u-u_\lambda)^+ \rangle \mathrm{d}x= \int_{\Sigma_\lambda} V(x)u^{p-1} (u-u_\lambda)^+  \mathrm{d}x,$$
where $V(x):=|x|^{-2p}\ast u^{p}(x)$. Since $u_\lambda$ satisfies $\displaystyle -\Delta_p u_\lambda =\int_{\R^N} \frac{u^p(y)}{|x_\lambda-y|^{2p}}\mathrm{d}y \cdot u_\lambda^{p-1}$, analogously we have
$$ \int_{\Sigma_\lambda} |\nabla u_\lambda|^{p-2} \langle \nabla u_\lambda,\nabla(u-u_\lambda)^+ \rangle \mathrm{d}x= \int_{\Sigma_\lambda} V(x_\lambda)u_\lambda^{p-1} (u-u_\lambda)^+  \mathrm{d}x.$$
By subtracting between the above two formulae, we obtain
\begin{align*}
\int_{\Sigma_\lambda}  \langle  |\nabla u|^{p-2}\nabla u- |\nabla u_\lambda|^{p-2}\nabla u_\lambda,\nabla(u-u_\lambda)^+ \rangle \mathrm{d}x= \int_{\Sigma_\lambda} ( V(x)u^{p-1} - V(x_\lambda)u_\lambda^{p-1} ) (u-u_\lambda)^+  \mathrm{d}x,
\end{align*}
and, using \eqref{eq0802}, we arrive at
\begin{align}\label{eq0803}
\int_{\Sigma_\lambda} (|\nabla u|+|\nabla u_\lambda|)^{p-2} |\nabla(u-u_\lambda)^+|^2 \mathrm{d}x \leq \frac{1}{C'} \int_{\Sigma_\lambda} ( V(x)u^{p-1} - V(x_\lambda)u_\lambda^{p-1} ) (u-u_\lambda)^+  \mathrm{d}x.
\end{align}

\smallskip

Using Lemma \ref{lm.1}, the fact that $u^{p-1}$ is convex in $u$ and the fact that $u\geq u_\lambda$ in the support of $(u-u_\lambda)^+$, we get
\begin{align}\label{eq0804}
&\quad \int_{\Sigma_\lambda} ( V(x)u^{p-1} - V(x_\lambda)u_\lambda^{p-1} ) (u-u_\lambda)^+ \mathrm{d}x  \\
&= \int_{\Sigma_\lambda}  V(x) (u^{p-1} - u_\lambda^{p-1} ) (u-u_\lambda)^+ \mathrm{d}x + \int_{\Sigma_\lambda} ( V(x)- V(x_\lambda) ) u_\lambda^{p-1}   (u-u_\lambda)^+ \mathrm{d}x \nonumber\\
&\leq (p-1) \int_{\Sigma_\lambda}  V(x) u^{p-2} [(u-u_\lambda)^+]^2 \mathrm{d}x + \int_{\Sigma_\lambda} P_\lambda(x) u^{p-1}   (u-u_\lambda)^+ \mathrm{d}x\nonumber\\
&=: I_1 + I_2, \nonumber
\end{align}
where
$$ P_\lambda(x):= p\int_{\Sigma_\lambda} \frac{u^{p-1}(y) (u-u_\lambda)^+(y)}{|x-y|^{2p}} \mathrm{d}y.$$

First, we estimate $I_1$. From Theorems \ref{th2.1} and \ref{th2.1-}, and \eqref{V}, we obtain
\begin{align}\label{eq0808}
I_1 &= (p-1) \int_{\Sigma_\lambda}  V(x) u^{p-2} [(u-u_\lambda)^+]^2 \mathrm{d}x \\
    &\leq (p-1)\int_{\Sigma_\lambda}  \frac{C_{\star}}{|x|^{p+\tau p}} \cdot \frac{C_0^{p-2}}{|x|^{\frac{N-p}{p-1}(p-2)}} [(u-u_\lambda)^+]^2 \mathrm{d}x, \nonumber
\end{align}
where $\tau$ is the same as in Lemma \ref{lm.5} and \eqref{V}. Now, in order to apply the weighted Hardy--Sobolev inequality (Lemma \ref{HDSI}), we set
$$ s:=-\frac{N-1}{p-1}(p-2)\quad\,\,\mbox{and}\quad\,\,\beta_1:=p+\tau p + \frac{N-p}{p-1}(p-2) +s-2.$$
Notice that
$$ s=-\frac{N-1}{p-1}(p-2)=-\left(1-\frac{1}{p-1}\right)(N-1)>2-N$$
and
$$\beta_1= \tau p >0.$$
Now, we are in the position to apply a version of the weighted Hardy-Sobolev inequality in Lemma \ref{HDSI} for decaying functions. From \eqref{eq0808}, Lemma \ref{HDSI}, the sharp lower bound on $|\nabla u|$ in \eqref{eq0806+} of Theorems \ref{th2.1} and \ref{th2.1-}, and the fact that $|x|>|\lambda|$ in $\Sigma_\lambda$, we get
\begin{align}\label{eq0809}
I_1 &\leq (p-1) C_{\star} \cdot C_0^{p-2} \frac{1}{|\lambda|^{\beta_1}} \int_{\Sigma_\lambda} |x|^{s-2} [(u-u_\lambda)^+]^2 \mathrm{d}x \\
&\leq (p-1) C_{\star} \cdot C_0^{p-2} \frac{1}{|\lambda|^{\tau p}} \left( \frac{2}{N+s-2} \right)^2 \int_{\Sigma_\lambda} |x|^s |\nabla(u-u_\lambda)^+|^2 \mathrm{d}x \nonumber\\
&\leq (p-1) \frac{C_{\star} \cdot C_0^{p-2}}{c_0^{p-2}} \frac{1}{|\lambda|^{\tau p}} \left( \frac{2}{N+s-2} \right)^2 \int_{\Sigma_\lambda} |\nabla u|^{p-2} |\nabla(u-u_\lambda)^+|^2 \mathrm{d}x \nonumber\\
&\leq (p-1) \frac{C_{\star} \cdot C_0^{p-2}}{c_0^{p-2}} \frac{1}{|\lambda|^{\tau p}} \left( \frac{2}{N+s-2} \right)^2 \int_{\Sigma_\lambda} \left(|\nabla u| +|\nabla u_\lambda|\right)^{p-2} |\nabla(u-u_\lambda)^+|^2 \mathrm{d}x, \nonumber
\end{align}
where $c_0$ is given by Theorem \ref{th2.1}.

\smallskip

Now we estimate $I_2$. By Hardy-Littlewood-Sobolev inequality (Lemma \ref{HLSI}) and H\"{o}lder's inequality, we have
\begin{align}\label{eq0810}
I_2 &=\int_{\Sigma_\lambda} P_\lambda(x) u^{p-1}(x)   (u-u_\lambda)^+(x) \mathrm{d}x  \\
&=p \int_{\Sigma_\lambda} \int_{\Sigma_\lambda} \frac{ u^{p-1}(x) (u-u_\lambda)^+(x) u^{p-1}(y) (u-u_\lambda)^+(y)}{|x-y|^{2p}} \mathrm{d}x\mathrm{d}y\nonumber\\
&\leq p\cdot C || u^{p-1}\cdot (u-u_\lambda)^+ ||^2_{L^\frac{N}{N-p}(\Sigma_\lambda)} \nonumber\\
&= p\cdot C || u^{\frac{p(N-2p)}{2(N-p)}}\cdot u^{\frac{p^{\star}-2}{2}} (u-u_\lambda)^+ ||^2_{L^\frac{N}{N-p}(\Sigma_\lambda)} \nonumber\\
&\leq p\cdot C \left( \int_{\Sigma_\lambda} u^{p^{\star}}\mathrm{d}x \right)^{\frac{N-2p}{N}}  \cdot  \int_{\Sigma_\lambda} u^{p^{\star}-2} [(u-u_\lambda)^+]^2 \mathrm{d}x  \nonumber\\
&\leq p\cdot C \left( \int_{\R^N} u^{p^{\star}}\mathrm{d}x \right)^{\frac{N-2p}{N}}  \cdot  \int_{\Sigma_\lambda} u^{p^{\star}-2} [(u-u_\lambda)^+]^2 \mathrm{d}x. \nonumber
\end{align}
Next, let $\beta_2:=(p^{\star}-2)\frac{N-p}{p-1}+s-2>0$, and similar to the estimates \eqref{eq0808} and \eqref{eq0809} of $I_1$, we have
\begin{align}\label{eq0901}
&\quad I_2 \leq p\cdot C \|u\|_{L^{p^{\star}}(\R^N)}^\frac{p(N-2p)}{N-p} \cdot  \int_{\Sigma_\lambda} u^{p^{\star}-2} [(u-u_\lambda)^+]^2 \mathrm{d}x \\
&\leq  p\cdot C \|u\|_{L^{p^{\star}}(\R^N)}^\frac{p(N-2p)}{N-p} C_0^{p^{\star}-2} \cdot \int_{\Sigma_\lambda} \frac{1}{|x|^{(p^{\star}-2)\frac{N-p}{p-1}}} [(u-u_\lambda)^+]^2 \mathrm{d}x \nonumber\\
&\leq  p\cdot C \|u\|_{L^{p^{\star}}(\R^N)}^\frac{p(N-2p)}{N-p} C_0^{p^{\star}-2} \frac{1}{|\lambda|^{\beta_2}} \cdot \int_{\Sigma_\lambda}  |x|^{s-2} [(u-u_\lambda)^+]^2 \mathrm{d}x \nonumber\\
&\leq  p\cdot C \|u\|_{L^{p^{\star}}(\R^N)}^\frac{p(N-2p)}{N-p} C_0^{p^{\star}-2} \frac{1}{|\lambda|^{\beta_2}}\left( \frac{2}{N+s-2} \right)^2 \cdot \int_{\Sigma_\lambda}  |x|^s |\nabla(u-u_\lambda)^+|^2 \mathrm{d}x \nonumber\\
&\leq \frac{ p\cdot C \cdot C_0^{p^{\star}-2}}{c_0^{p^{\star}-2}} \|u\|_{L^{p^{\star}}(\R^N)}^\frac{p(N-2p)}{N-p}  \frac{1}{|\lambda|^{\beta_2}}\left( \frac{2}{N+s-2} \right)^2 \cdot \int_{\Sigma_\lambda}  (|\nabla u|+|\nabla u_\lambda|)^{p-2} |\nabla(u-u_\lambda)^+|^2 \mathrm{d}x. \nonumber
\end{align}
From \eqref{eq0803}, \eqref{eq0804}, \eqref{eq0809} and \eqref{eq0901}, we have
\begin{align}\label{eq0902}
&\int_{\Sigma_\lambda}  (|\nabla u|+|\nabla u_\lambda|)^{p-2} |\nabla(u-u_\lambda)^+|^2 \mathrm{d}x \\
&\leq (p-1) \frac{C_{\star} \cdot C_0^{p-2}}{C' \cdot c_0^{p-2}} \frac{1}{|\lambda|^{\tau p}} \left( \frac{2}{N+s-2} \right)^2 \int_{\Sigma_\lambda}\left(|\nabla u| +|\nabla u_\lambda|\right)^{p-2} |\nabla(u-u_\lambda)^+|^2 \mathrm{d}x \nonumber\\
&\,\,+\frac{ p\cdot C \cdot C_0^{p^{\star}-2}}{C' \cdot c_0^{p^{\star}-2}} \|u\|_{L^{p^{\star}}(\R^N)}^\frac{p(N-2p)}{N-p}  \frac{1}{|\lambda|^{\beta_2}}\left( \frac{2}{N+s-2} \right)^2 \cdot \int_{\Sigma_\lambda}  (|\nabla u|+|\nabla u_\lambda|)^{p-2} |\nabla(u-u_\lambda)^+|^2 \mathrm{d}x,  \nonumber
\end{align}
where the positive constants $C_{0}$ and $c_{0}$ come from the sharp asymptotic estimates \eqref{eq0806} and \eqref{eq0806+} in Theorems \ref{th2.1} and \ref{th2.1-}, and $C'$ comes from the basic point-wise inequalities in Lemma \ref{BSICI}. For $\lambda$ sufficiently negative such that
$$ (p-1) \frac{C_{\star} \cdot C_0^{p-2}}{C' \cdot c_0^{p-2}} \frac{1}{|\lambda|^{\tau p}} \left( \frac{2}{N+s-2} \right)^2 + \frac{ p\cdot C \cdot C_0^{p^{\star}-2}}{C' \cdot c_0{p^{\star}-2}} \|u\|_{L^{p^{\star}}(\R^N)}^\frac{p(N-2p)}{N-p}  \frac{1}{|\lambda|^{\beta_2}}\left( \frac{2}{N+s-2} \right)^2 <1,$$
it follows that \eqref{eq0902} provides a contradiction unless $(u-u_\lambda)^{+}\chi_{\Sigma_{\lambda}}=0$. Therefore, for $\lambda$ sufficiently negative,
$$ u\leq u_\lambda \quad\,\,\,\mbox{in}\,\,\Sigma_\lambda.$$
Note that, by the strong comparison principle (Lemma \ref{th3.1}) and the fact that
\begin{align}\label{eq0903}
-\Delta_p u =V(x)u^{p-1} \leq V(x_\lambda) u^{p-1}_\lambda= -\Delta_p u_\lambda \,\, \quad\,\,\mbox{in}\,\,\Sigma_\lambda,
\end{align}
we deduce that, either $u<u_\lambda$ in $\Sigma_\lambda$ or $u=u_\lambda$ in $\Sigma_\lambda$.

Now we have proved $\Lambda\neq\emptyset$ and finished Step 1. By continuity, we have $u\leq u_{\lambda_0}$ in $\Sigma_{\lambda_0}$.

\smallskip

{\bf Step 2.} We are to show that $u=u_{\lambda_0}$ in $\Sigma_{\lambda_0}$.

Assume by contradiction that $u$ does not coincide with $u_{\lambda_0}$ in $\Sigma_{\lambda_0}$. Again, by the strong comparison principle (Lemma \ref{th3.1}) and \eqref{eq0903}, we deduce that $u<u_{\lambda_0}$ in $\Sigma_{\lambda_0}$. Let us take $R > 0$ large (to be determined later), $\epsilon>0$ and $\delta>0$ small, and set
$$ B^\varepsilon_R:=\Sigma_{\lambda_0 + \varepsilon} \setminus B_R, \quad\quad  S_\delta^\varepsilon:=(\Sigma_{\lambda_0 + \varepsilon} \setminus \Sigma_{\lambda_0 - \delta})\cap B_R  \quad \,\,\mbox{and}\,\,\quad  K_\delta:=\overline{B_R\cap\Sigma_{\lambda_0 - \delta}}.$$
It is clear that
$$ \Sigma_{\lambda_0 + \varepsilon}=B^\varepsilon_R \cup S_\delta^\varepsilon \cup K_\delta.$$
For $\delta>0$, we have $u_{\lambda_0}>u$ in $K_\delta$. Since $K_\delta$ is compact, there exists a small $\overline{\varepsilon}>0$ such that
$$u_{\lambda_0+\varepsilon}>u \,\,\,\,\quad \,\,\mbox{in}\,\,K_\delta$$
for any $0\leq\varepsilon\leq\overline{\varepsilon}$. For $0\leq\varepsilon\leq\overline{\varepsilon}$, we may argue as in Step 1 by choosing $(u-u_{\lambda_0+\varepsilon})^+ \chi_{\Sigma_{\lambda_0+\varepsilon}}$ as the test function. In the following, the integral in $B^\varepsilon_R$ can be estimated as in Step 1 and recalling that weighted Hardy-Sobolev inequality only requires the functions to vanish at infinity. Hence, we can deduce from the weighted Hardy-Sobolev inequality in Lemma \ref{HDSI}, Hardy-Littlewood-Sobolev inequality, H\"{o}lder's inequality and the sharp asymptotic estimates in Theorems \ref{th2.1} and \ref{th2.1-} that
\begin{align}\label{eq1001}
&\quad\,\,\int_{\Sigma_{\lambda_0+\varepsilon}} (|\nabla u|+|\nabla u_{\lambda_0+\varepsilon}|)^{p-2} |\nabla(u-u_{\lambda_0+\varepsilon})^+|^2 \mathrm{d}x \\
&\leq \frac{1}{C'} \int_{\Sigma_{\lambda_0+\varepsilon}}  V(x) (u^{p-1} - u_{\lambda_0+\varepsilon}^{p-1} ) (u-u_{\lambda_0+\varepsilon})^+ \mathrm{d}x\nonumber\\
&\,\,+ \frac{1}{C'} \int_{\Sigma_{\lambda_0+\varepsilon}} ( V(x)- V(x_{\lambda_0+\varepsilon}) ) u_{\lambda_0+\varepsilon}^{p-1}   (u-u_{\lambda_0+\varepsilon})^+ \mathrm{d}x \nonumber\\
&\leq \frac{p-1}{C'} \int_{\Sigma_{\lambda_0+\varepsilon}}  V(x) u^{p-2} [(u-u_{\lambda_0+\varepsilon})^+]^2 \mathrm{d}x \nonumber \\
     &\quad + \frac{p C}{C'}   \|u\|_{L^{p^{\star}}(\R^N)}^\frac{p(N-2p)}{N-p}    \int_{\Sigma_{\lambda_0+\varepsilon}} u^{p^{\star}-2} [(u-u_{\lambda_0+\varepsilon})^+]^2 \mathrm{d}x \nonumber\\
&\leq  \frac{(p-1) C_{\star}  C_0^{p-2}}{C' c_0^{p-2}} \frac{1}{R^{\tau p}} \left( \frac{2}{N+s-2} \right)^2 \int_{B^\varepsilon_R} \left(|\nabla u| +|\nabla u_{\lambda_0+\varepsilon}|\right)^{p-2} |\nabla(u-u_{\lambda_0+\varepsilon})^+|^2 \mathrm{d}x \nonumber\\
&\,\,+\frac{ p C C_0^{p^{\star}-2}}{C'  c_0^{p^{\star}-2}} \|u\|_{L^{p^{\star}}(\R^N)}^\frac{p(N-2p)}{N-p}  \frac{1}{R^{\beta_2}}\left( \frac{2}{N+s-2} \right)^2   \int_{B^\varepsilon_R}  (|\nabla u|+|\nabla u_{\lambda_0+\varepsilon}|)^{p-2} |\nabla(u-u_{\lambda_0+\varepsilon})^+|^2 \mathrm{d}x  \nonumber\\
&\,\,+ \frac{(p-1)C_{2} \|u\|_\infty^{p-2}}{C'} \int_{S_\delta^\varepsilon} [(u-u_{\lambda_0+\varepsilon})^+]^2 \mathrm{d}x
      + \frac{p C \|u\|_\infty^{p^{\star}-2}}{C'} \|u\|_{L^{p^{\star}}(\R^N)}^\frac{p(N-2p)}{N-p}  \int_{S_\delta^\varepsilon} [(u-u_{\lambda_0+\varepsilon})^+]^2 \mathrm{d}x, \nonumber 
\end{align}
where the positive constants $C_{0}$ and $c_{0}$ come from the sharp asymptotic estimates \eqref{eq0806} and \eqref{eq0806+} in Theorems \ref{th2.1} and \ref{th2.1-}, $C_{\ast}$ comes from the upper bound on decay rate of $V(x)$ in \eqref{V} and $C_{2}$ comes from the local upper bound of $V(x)$ in \eqref{eq10.26.30.03}.

\smallskip

To estimate the integral in $S_\delta^\varepsilon$, we will use the weighted Poincar\'{e} type inequality in Lemma \ref{lm.3}. For $x\in S_\delta^\varepsilon$, we let $v(x):=(u-u_{\lambda_0+\varepsilon})^+$. By the definition of $S_\delta^\varepsilon$ and $v(x)$, we have $v(x)=0$ on $\partial\Sigma_{\lambda_0+\varepsilon} \cup (\Sigma_{\lambda_0-\delta}\cap B_R)$, then we can define $v(x+te_{1}):=0$ for all $x\in \partial\Sigma_{\lambda_0+\varepsilon}\cap B_{R}$ and $t>0$. For all $x\in S_\delta^\varepsilon$, we can write
$$ v(x)=-\int_{0}^{+\infty}\frac{\partial v}{\partial x_1}(x+te_{1}) \mathrm{d}t=-\int_{0}^{\lambda_0+\varepsilon-x_{1}}\frac{\partial v}{\partial x_1}(x+te_{1}) \mathrm{d}t,$$
which in conjunction with integrating on $x_{2},\cdots,x_{N}$ and finite covering theorem imply that, for some $C>0$,
$$ |v(x)|\leq C\int_{S_\delta^\varepsilon} \frac{|\nabla v(y)|}{|x-y|^{N-1}}\mathrm{d}y, \quad\quad \,\,\,\,\forall \,\, x\in S_\delta^\varepsilon.$$
Thus $v$ satisfies \eqref{eq10.25.52}  with $\Omega=S_\delta^\varepsilon$. From Lemma \ref{lm.4}, Corollary \ref{re2333} or Remark \ref{rem0}, one has that $\rho=|\nabla u|^{p-2}$ satisfies \eqref{eq10.25.2}. So by the weighted Poincar\'{e} type inequality in Lemma \ref{lm.3}, we obtain
\begin{align}\label{eq1002}
\int_{S_\delta^\varepsilon} [(u-u_{\lambda_0+\varepsilon})^+]^2 \mathrm{d}x \leq C(S_\delta^\varepsilon) \int_{S_\delta^\varepsilon} (|\nabla u|+|\nabla u_{\lambda_0+\varepsilon}|)^{p-2} |\nabla(u-u_{\lambda_0+\varepsilon})^+|^2 \mathrm{d}x,
\end{align}
where $C(S_\delta^\varepsilon)\to 0$ if $|S_\delta^\varepsilon|\to 0.$

Collecting \eqref{eq1001} and \eqref{eq1002}, we deduce
\begin{align}\label{eq1003}
&\quad \int_{\Sigma_{\lambda_0+\varepsilon}} (|\nabla u|+|\nabla u_{\lambda_0+\varepsilon}|)^{p-2} |\nabla(u-u_{\lambda_0+\varepsilon})^+|^2 \mathrm{d}x \\
&\leq  \frac{(p-1) C_{\star}  C_0^{p-2}}{C' c_0^{p-2}} \frac{1}{R^{\tau p}} \left( \frac{2}{N+s-2} \right)^2 \int_{B^\varepsilon_R} |\nabla u +\nabla u_{\lambda_0+\varepsilon}|^{p-2} |\nabla(u-u_{\lambda_0+\varepsilon})^+|^2 \mathrm{d}x \nonumber\\
&\,\,+\frac{ p C C_0^{p^{\star}-2}}{C'  c_0^{p^{\star}-2}} \|u\|_{L^{p^{\star}}(\R^N)}^\frac{p(N-2p)}{N-p}  \frac{1}{R^{\beta_2}}\left( \frac{2}{N+s-2} \right)^2 \nonumber\\
&\quad \times\int_{B^\varepsilon_R}  (|\nabla u|+|\nabla u_{\lambda_0+\varepsilon}|)^{p-2} |\nabla(u-u_{\lambda_0+\varepsilon})^+|^2 \mathrm{d}x  \nonumber\\
&\,\,+ \frac{(p-1)C_{2} \|u\|_\infty^{p-2}}{C'} C(S_\delta^\varepsilon) \int_{S_\delta^\varepsilon} (|\nabla u|+|\nabla u_{\lambda_0+\varepsilon}|)^{p-2} |\nabla(u-u_{\lambda_0+\varepsilon})^+|^2 \mathrm{d}x  \nonumber\\
&\,\,+ \frac{p C \|u\|_\infty^{p^{\star}-2}}{C'} \|u\|_{L^{p^{\star}}(\R^N)}^\frac{p(N-2p)}{N-p} C(S_\delta^\varepsilon)\int_{S_\delta^\varepsilon} (|\nabla u|+|\nabla u_{\lambda_0+\varepsilon}|)^{p-2} |\nabla(u-u_{\lambda_0+\varepsilon})^+|^2 \mathrm{d}x. \nonumber
\end{align}
Now we take care of the variable parameters $R, \delta, \overline{\varepsilon}$. First we fix $R$ large so that
$$\frac{(p-1) C_{\star}  C_0^{p-2}}{C' c_0^{p-2}} \frac{1}{R^{\tau p}} \left( \frac{2}{N+s-2} \right)^2 + \frac{ p C C_0^{p^{\star}-2}}{C'  c_0^{p^{\star}-2}} \|u\|_{L^{p^{\star}}(\R^N)}^\frac{p(N-2p)}{N-p}  \frac{1}{R^{\beta_2}}\left( \frac{2}{N+s-2} \right)^2 < 1.$$
Then, recalling that $|S_\delta^\varepsilon|\to 0$ if $\delta\to 0$ and $\varepsilon\to 0$, we choose $\delta$ and $\overline{\varepsilon}$ small such that
$$  \frac{(p-1)C_{2} \|u\|_\infty^{p-2}}{C'} C(S_\delta^\varepsilon) + \frac{p C \|u\|_\infty^{p^{\star}-2}}{C'} \|u\|_{L^{p^{\star}}(\R^N)}^\frac{p(N-2p)}{N-p} C(S_\delta^\varepsilon) <1$$
for any $0\leq\varepsilon\leq\overline{\varepsilon}$. With such choice of parameters, by \eqref{eq1003}, we have
$$\int_{\Sigma_{\lambda_0+\varepsilon}} (|\nabla u|+|\nabla u_{\lambda_0+\varepsilon}|)^{p-2} |\nabla(u-u_{\lambda_0+\varepsilon})^+|^2 \mathrm{d}x\leq 0,$$
which implies $(u-u_{\lambda_0+\varepsilon})^+=0$ in $\Sigma_{\lambda_0+\varepsilon}$ for any $0\leq\varepsilon\leq\overline{\varepsilon}$. This is a contradiction with the definition of $\lambda_0$.
Thus, we prove that
$$ u=u_{\lambda_0} \quad\quad \,\,\mbox{in}\,\,\Sigma_{\lambda_0}.$$
Furthermore, $ u<u_{\lambda}$ in $\Sigma_\lambda$ for any $\lambda < \lambda_0$ and consequently $u$ is strictly monotone increasing in $\Sigma_{\lambda_0}$ so that $\frac{\partial u}{\partial x_{1}}\geq 0$ in $\Sigma_{\lambda_0}$.

\medskip

{\bf Step 3.} $u$ is radially symmetric and strictly decreasing about some point $x_0\in\R^N$.

By the strong maximum principle of the linearized operator \cite[Theorem 1.2]{LDBS}, it actually follows that
\begin{equation}\label{sd}
  \frac{\partial u}{\partial x_{1}}> 0 \qquad \text{in} \,\, \Sigma_{\lambda_0}.
\end{equation}


It is known that for any direction $\nu$, $u$ is symmetric about the plane $T^\nu_{\lambda_0(\nu)}$. Now consider all the $N$ orthogonal directions $\{e_k\}_{k=1}^N$ in $\R^N$, and one can obtain in entirely similar way that $u$ is strictly increasing in each direction $e_k$ and $\frac{\partial u}{\partial e_{k}}>0$ in $\Sigma^{e_k}_{\lambda_0(e_k)}$ and symmetric about the plane $T^{e_k}_{\lambda_0(e_k)}$. This shows that $u$ is symmetric with respect to a point $x_0\in\R^N$ $(x_0 = \bigcap^N_{k=1}T^{e_k}_{\lambda_0(e_k)})$, which is the unique critical point for $u$. Furthermore, by exploiting the moving plane procedure in any direction $\nu\in S^{N-1}$, we finally get that $u$ is radial and (strictly) radially decreasing w.r.t. $x_{0}$. That is, positive solution $u$ must assume the form
\[u(x)=\lambda^{\frac{N-p}{p}}U(\lambda(x-x_{0}))\]
for $\lambda:=u(0)^{\frac{p}{N-p}}>0$ and some $x_{0}\in\mathbb{R}^{N}$, where $U(x)=U(r)$ with $r=|x|$ is the positive radial solution to \eqref{eq1.1} with $U(0)=1$, $U'(0)=0$ and $U'(r)\leq0$ for any $r>0$. From \eqref{sd}, one has $U'(r)<0$ for any $r>0$. This completes our proof of Theorem \ref{th2} for $2<p<\frac{N}{2}$.
\end{proof}

\subsection{The singular elliptic case $1<p<2$}\label{sc4.2}

In this sub-section, we carry out our proof of Theorem \ref{th2} in the singular elliptic case $1<p<2$ via the moving planes method. In the following, we define the set
$$\overline{\Lambda}(\nu):=\{ \lambda\in\R \mid u\geq u_\mu^\nu\,\quad \mbox{in}\,\,\Sigma_\mu^\nu,\,\,\,\,\forall \,\, \mu\geq\lambda \},$$
and, if $\overline{\Lambda}(\nu)\neq\emptyset$, we define
$$ \lambda_0(\nu):=\inf \overline{\Lambda}(\nu).$$

\subsubsection{Weak comparison principles}
For the singular elliptic case $1<p<2$, in order to show the radial symmetry and strictly radial monotonicity of the weak solution $u$ to \eqref{eq1.1} via the method of moving planes, we need weak comparison principles in unbounded domains. To this end, we will use the following Poincar\'{e} type inequality from \cite{LDMR} (see Lemma \ref{PTI}).

Let us consider a ball $B(P,R)$ with center $P\in T_\lambda^\nu$ and radius $R$. Set
$$B_\lambda^\nu:=B(P,R)\cap\Sigma_\lambda^\nu,$$
and for any $A\subset\R^N$, let
$$ M_A:=\sup_A (|\nabla u|+|\nabla u^{\nu}_{\lambda}|).$$
For any $A\subset \Sigma_\lambda^\nu$, define
$$A^\prime:=R_\lambda^\nu(A).$$

We apply the following Poincar\'{e} type inequality to functions on $B_\lambda^\nu$, vanishing on $T^\nu_\lambda$.
\begin{lem}[Poincar\'{e} type inequality, \cite{LDMR}]\label{PTI}
Let $w\in W^{1,q}_{loc} (\R^N)$ with $1\leq q< \infty$ vanish on $T^\nu_\lambda$, and let $(supp\,w)\cap B_\lambda^\nu =A_1\cup A_2$ for some disjoint measurable sets $A_i$. Then there exists a constant $C=C(N)$ such that
$$ \| w \|_{L^q(B_\lambda^\nu)} \leq C |B_\lambda^\nu\cap supp \, w|^\frac{1}{Nq} \left( |A_1|^\frac{1}{Nq^\prime} \| \nabla w \|_{L^q(A_1)} + |A_2|^\frac{1}{Nq^\prime} \| \nabla w \|_{L^q(A_2)}  \right),$$
where $ \frac{1}{q}+ \frac{1}{q^\prime} =1 $.
\end{lem}

Now, with the help of Lemma \ref{PTI}, we are able to prove the following weak comparison principle.
\begin{lem}\label{lm.4.1}
Assume $N>2$, $1<p<2$ and $u$ is a positive weak solution of \eqref{eq1.1}. Let $R_0$ be the same as in Theorem \ref{th2.1}.

{\bf $(i)$} There exists $R_1>R_0$, $R_1$ depending on $N,p,R_0$, such that if
$$\Sigma_\lambda^\nu \cap R_\lambda^\nu(B_{R_1}) = \emptyset,$$
then $u^{\nu}_{\lambda} \leq u$ in $\Sigma_\lambda^\nu$;

{\bf $(ii)$} Let $\Sigma_\lambda^\nu \cap R_\lambda^\nu(B_{R_1}) \neq \emptyset$ and let $R_2=R_2(R_1,\lambda,\nu)>2R_{1}$ be such that $\Sigma_\lambda^\nu \cap R_\lambda^\nu(B_{R_1}) \subset B(P_\lambda^\nu,R_2)$, where $P_\lambda^\nu$ is the projection of the origin on $T_\lambda^\nu$.
Then there exists $R>R_2$, depending on $R_2$, such that, for any $\overline{R}\geq R$, there exist positive numbers $\delta$ and $M$, depending $\overline{R},\lambda,p,N$ and the $L^\infty$--norms of $u,v,|\nabla u|,|\nabla v|$, with the property that whenever
$$supp(u^{\nu}_{\lambda}-u)^+ \cap\Sigma^{\nu}_{\lambda}\cap B(P_\lambda^\nu,\overline{R})=A \cup B$$
with $|A\cap B|=0$, $M_B<M$ and $|A|+|B|<\delta$, we have
$$ u^{\nu}_{\lambda} \leq u\,\,\quad \mbox{in}\,\,\Sigma_\lambda^\nu.$$
\end{lem}

\begin{proof}


Let us take $(u^{\nu}_{\lambda}-u)^+\chi_{\Sigma^{\nu}_\lambda}$ as a test function in the Definition \ref{weak} of weak solutions. Proceeding as in the deduction of \eqref{eq0803}, we get
\begin{align}\label{eq1004}
&\quad \int_{\Sigma_\lambda^\nu} (|\nabla u|+|\nabla u^{\nu}_{\lambda}|)^{p-2} |\nabla(u^{\nu}_{\lambda}-u)^+|^2 \mathrm{d}x \\
&\leq \frac{1}{C'} \int_{\Sigma_\lambda^\nu} ( V(x^{\nu}_{\lambda})\left(u^{\nu}_{\lambda}\right)^{p-1} - V(x)u^{p-1} ) (u^{\nu}_{\lambda} - u)^+  \mathrm{d}x  \nonumber\\
&= \frac{1}{C'} \int_{\Sigma_\lambda^\nu}  V(x) ( \left(u^{\nu}_{\lambda}\right)^{p-1} - u^{p-1} ) (u^{\nu}_{\lambda}-u )^+ \mathrm{d}x \nonumber \\
&\quad + \frac{1}{C'} \int_{\Sigma_\lambda^\nu} ( V(x^{\nu}_{\lambda})- V(x) )\left(u^{\nu}_{\lambda}\right)^{p-1}   (u^{\nu}_{\lambda}-u)^+ \mathrm{d}x \nonumber\\
&=:II_1 + II_2. \nonumber
\end{align}

\medskip

\noindent {\bf Proof of $(i)$.} Let $R > R_0$, with $R_0$ be the same as in Theorem \ref{th2.1}, to be determined later. By the sharp asymptotic estimates for $u$ in Theorems \ref{th2.1} and \ref{th2.1-}, there exists $R_1= R_1(R) > R$ such
that
$$ \inf_{B_R} u \geq \sup_{\R^N\setminus B_{R_1}} u,$$
where $B_{r}:=B_{r}(0)$ for any $r>0$. If $\lambda>0$ is sufficiently large such that
$$\Sigma_\lambda^\nu \cap R_\lambda^\nu(B_{R_1}) = \emptyset,$$
then we have
$$u\geq u^{\nu}_{\lambda} \quad \,\mbox{in}\,\, B_R \,\,\quad \mbox{and}\quad \,\, supp(u^{\nu}_{\lambda}-u)^+\cap\Sigma^{\nu}_{\lambda}\subset \Sigma_\lambda^\nu\setminus B_R.$$

Since $1<p<2$, using Lagrange's mean value theorem, Theorems \ref{th2.1} and \ref{th2.1-}, and \eqref{V},
we obtain
\begin{align}\label{eq1005}
II_1 &= \frac{1}{C'} \int_{\Sigma_\lambda^\nu}  V(x) (\left(u^{\nu}_{\lambda}\right)^{p-1} - u^{p-1} ) (u^{\nu}_{\lambda}-u )^+ \mathrm{d}x \\
     &\leq \frac{p-1}{C'} \int_{\Sigma_\lambda^\nu\setminus B_R}  V(x) u^{p-2} [(u^{\nu}_{\lambda}-u )^+]^2 \mathrm{d}x \nonumber\\
     &\leq \frac{(p-1) C_{\star}  c_0^{p-2} }{C'} \int_{\Sigma_\lambda^\nu\setminus B_R}  \frac{1}{|x|^{p+\tau p}} \cdot \frac{1}{|x|^{\frac{N-p}{p-1}(p-2)}} [(u^{\nu}_{\lambda}-u )^+]^2 \mathrm{d}x, \nonumber
\end{align}
where $\tau$ is given in Lemma \ref{lm.5} and \eqref{V}. Similar to \eqref{eq0809}, using the weighted Hardy-Sobolev inequality, we get
\begin{align}\label{eq10061}
II_1 &\leq \frac{(p-1) C_{\star}  c_0^{p-2} }{C'} \sup_{\Sigma_\lambda^\nu\setminus B_R} \left(\frac{1}{|x|^{\beta_1}}\right)
      \int_{\Sigma_\lambda^\nu\setminus B_R} |x|^{s-2} [(u^{\nu}_{\lambda}-u )^+]^2 \mathrm{d}x \\
      &\leq \frac{(p-1) C_{\star}  c_0^{p-2} }{C'} \sup_{\Sigma_\lambda^\nu\setminus B_R} \left(\frac{1}{|x|^{\tau p}}\right)
      \left(\frac{2}{N+s-2}\right)^2 \int_{\Sigma_\lambda^\nu\setminus B_R} |x|^s |\nabla(u^{\nu}_{\lambda}-u )^+|^2 \mathrm{d}x, \nonumber
\end{align}
where $s =-\frac{N-1}{p-1}(p-2)>2-N$ and $\beta_1=p+\tau p + \frac{N-p}{p-1}(p-2) +s-2=\tau p>0$. By Theorems \ref{th2.1} and \ref{th2.1-}, and the fact that $|x|<|x^{\nu}_{\lambda}|$ in $\Sigma_\lambda$, we have
$$|\nabla u|+|\nabla u^{\nu}_{\lambda}| \leq C_0 \left( \frac{1}{|x|^\frac{N-1}{p-1}} + \frac{1}{|x^{\nu}_{\lambda}|^\frac{N-1}{p-1}} \right) \leq
  \frac{2C_0}{|x|^\frac{N-1}{p-1}} \,\,\quad \mbox{in}\,\,\Sigma_\lambda^\nu\setminus B_R,$$
which implies
$$(|\nabla u|+|\nabla u^{\nu}_{\lambda}|)^{p-2} \geq
  \frac{(2C_0)^{p-2}}{|x|^{\frac{N-1}{p-1}(p-2)}}=(2C_0)^{p-2} |x|^s  \quad \,\,\mbox{in}\,\,\Sigma_\lambda^\nu\setminus B_R.$$
This yields
\begin{align}\label{eq1006}
II_1 \leq  \frac{(p-1) C_{\star}  c_0^{p-2} }{C' (2C_0)^{p-2}} \frac{1}{R^{\tau p}}
           \left(\frac{2}{N+s-2}\right)^2\int_{\Sigma_\lambda^\nu} (|\nabla u|+|\nabla u^{\nu}_{\lambda}|)^{p-2} |\nabla(u^{\nu}_{\lambda}-u )^+|^2 \mathrm{d}x.
\end{align}

For $II_2$, using Lemma \ref{lm.1}, Lagrange's mean value theorem, Hardy-Littlewood-Sobolev inequality (Lemma \ref{HLSI}) and H\"{o}lder's inequality, we get
\begin{small}\begin{align}\label{eq1007}
& \quad II_2=     \frac{1}{C'} \int_{\Sigma_\lambda^\nu} ( V(x^{\nu}_{\lambda})- V(x) )\left(u^{\nu}_{\lambda}\right)^{p-1}   (u^{\nu}_{\lambda}-u)^+ \mathrm{d}x \\
     &\leq  \frac{p}{C'} \int_{\Sigma_\lambda^\nu\cap supp(u^{\nu}_{\lambda}-u)^+} \left(\int_{\Sigma_\lambda^\nu}\frac{\left(u^{\nu}_{\lambda}\right)^{p-1}(y) (u^{\nu}_{\lambda}-u)^+(y) }{|x-y|^{2p}} \mathrm{d}y \right)\left(u^{\nu}_{\lambda}\right)^{p-1}(x)   (u^{\nu}_{\lambda}-u)^+(x) \mathrm{d}x \nonumber\\
     &\leq  \frac{p}{C'} \int_{\Sigma_\lambda^\nu\cap supp(u^{\nu}_{\lambda}-u)^+} \int_{\Sigma_\lambda^\nu\cap supp(u^{\nu}_{\lambda}-u)^+} \frac{\left(u^{\nu}_{\lambda}\right)^{p-1}(y) (u^{\nu}_{\lambda}-u)^+(y) \left(u^{\nu}_{\lambda}\right)^{p-1}(x) (u^{\nu}_{\lambda}-u)^+(x)  }{|x-y|^{2p}} \mathrm{d}x\mathrm{d}y \nonumber\\
     &\leq  \frac{pC}{C'} \| \left(u^{\nu}_{\lambda}\right)^{p-1}  (u^{\nu}_{\lambda}-u)^+ \|^2_{L^\frac{N}{N-p}(\Sigma_\lambda^\nu\cap supp(u^{\nu}_{\lambda}-u)^+)} \nonumber\\
     &\leq  \frac{pC}{C'} \| u^{\nu}_{\lambda} \|^{2(p-1)}_{L^{p^{\star}}(\Sigma_\lambda^\nu\cap supp(u^{\nu}_{\lambda}-u)^+)}  \| (u^{\nu}_{\lambda}-u)^+ \|^2_{L^{p^{\star}}(\Sigma_\lambda^\nu\cap supp(u^{\nu}_{\lambda}-u)^+)}\nonumber\\
     &\leq  \frac{pC}{C' S^\frac{2}{p}} \| u^{\nu}_{\lambda} \|^{2(p-1)}_{L^{p^{\star}}(\Sigma_\lambda^\nu\cap supp(u^{\nu}_{\lambda}-u)^+)} \left( \int_{\Sigma_\lambda^\nu} |\nabla (u^{\nu}_{\lambda}-u)^+|^p \mathrm{d}x \right)^\frac{2}{p}, \nonumber
\end{align}\end{small}

\noindent where we have also used Sobolev inequality and $S$ is the best Sobolev constant (see \cite{Talenti}). By H\"{o}lder's inequality, we have
\begin{align}\label{eq1008}
&\quad \int_{\Sigma_\lambda^\nu} |\nabla (u^{\nu}_{\lambda}-u)^+|^p \mathrm{d}x \\
&=\int_{\Sigma_\lambda^\nu} [(|\nabla u|+|\nabla u^{\nu}_{\lambda}|)^\frac{p(2-p)}{2}][(|\nabla u|+|\nabla u^{\nu}_{\lambda}|)^\frac{p(p-2)}{2}|\nabla (u^{\nu}_{\lambda}-u)^+|^p] \mathrm{d}x \nonumber\\
&\leq \left( \int_{\Sigma_\lambda^\nu} (|\nabla u|+|\nabla u^{\nu}_{\lambda}|)^p \mathrm{d}x  \right)^\frac{2-p}{2} \left(  \int_{\Sigma_\lambda^\nu}(|\nabla u|+|\nabla u^{\nu}_{\lambda}|)^{p-2} |\nabla (u^{\nu}_{\lambda}-u)^+|^2 \mathrm{d}x \right)^\frac{p}{2} \nonumber\\
&\leq \overline{C} \left(  \int_{\Sigma_\lambda^\nu}(|\nabla u|+|\nabla u^{\nu}_{\lambda}|)^{p-2} |\nabla (u^{\nu}_{\lambda}-u)^+|^2 \mathrm{d}x \right)^\frac{p}{2}, \nonumber
\end{align}
where we have used the fact that $|\nabla u|$ and $|\nabla u^{\nu}_{\lambda}|$ are both in $L^p(\R^N)$ since $u\in D^{1,p}(\R^{N})$. Combining \eqref{eq1007},  \eqref{eq1008} and the fact that $supp(u^{\nu}_{\lambda}-u)^+\cap\Sigma^{\nu}_{\lambda}\subset \Sigma_\lambda^\nu\setminus B_R$, it follows that
\begin{align}\label{eq1009}
II_2 &\leq  \frac{p C \overline{C}^\frac{2}{p}}{C' S^\frac{2}{p}} \| u^{\nu}_{\lambda} \|^{2(p-1)}_{L^{p^{\star}}(\Sigma_\lambda^\nu \setminus B_R)}
      \int_{\Sigma_\lambda^\nu}(|\nabla u|+|\nabla u^{\nu}_{\lambda}|)^{p-2} |\nabla (u^{\nu}_{\lambda}-u)^+|^2 \mathrm{d}x.
\end{align}
From \eqref{eq1004},  \eqref{eq1006} and \eqref{eq1009}, we get
\begin{align}\label{eq1010}
&\quad \int_{\Sigma_\lambda^\nu}(|\nabla u|+|\nabla u^{\nu}_{\lambda}|)^{p-2} |\nabla (u^{\nu}_{\lambda}-u)^+|^2 \mathrm{d}x \\
&\leq   \frac{(p-1) C_{\star}  c_0^{p-2} }{C' (2C_0)^{p-2}} \frac{1}{R^{\tau p}}
           \left(\frac{2}{N+s-2}\right)^2  \int_{\Sigma_\lambda^\nu} (|\nabla u|+|\nabla u^{\nu}_{\lambda}|)^{p-2} |\nabla(u^{\nu}_{\lambda}-u )^+|^2 \mathrm{d}x  \nonumber\\
&      \quad+\frac{p C \overline{C}^\frac{2}{p}}{C' S^\frac{2}{p}} \| u^{\nu}_{\lambda} \|^{2(p-1)}_{L^{p^{\star}}(\Sigma_\lambda^\nu \setminus B_R)}
        \int_{\Sigma_\lambda^\nu}(|\nabla u|+|\nabla u^{\nu}_{\lambda}|)^{p-2} |\nabla (u^{\nu}_{\lambda}-u)^+|^2 \mathrm{d}x.  \nonumber
\end{align}
We can now choose $R>R_0$ such that
$$ G(R):= \frac{(p-1) C_{\star}  c_0^{p-2} }{C' (2C_0)^{p-2}} \frac{1}{R^{\tau p}}
           \left(\frac{2}{N+s-2}\right)^2 + \frac{p C \overline{C}^\frac{2}{p}}{C' S^\frac{2}{p}} \| u^{\nu}_{\lambda} \|^{2(p-1)}_{L^{p^{\star}}(\Sigma_\lambda^\nu \setminus B_R)} <1,$$
and then take $R_1 = R_1(R)$.  Then \eqref{eq1010} yields that $(u^{\nu}_{\lambda}-u)^+=0$ a.e. in $\Sigma_\lambda^\nu$, which implies that $u\geq u^{\nu}_{\lambda}$ in $\Sigma_\lambda^\nu$ for any $\lambda>R_1$.

\medskip

\noindent {\bf Proof of $(ii)$.}  Now suppose that $\Sigma_\lambda^\nu\cap R_\lambda^\nu(B_{R_1}) \neq \emptyset$ and let $R>R_2$ to be determined later. For $\overline{R}\geq R$, we define $K:=B(P_\lambda^\nu,\overline{R})$, where $P_\lambda^\nu \in T_\lambda^\nu$ is the projection of the origin on $T_\lambda^\nu$. Observe that, for $R$ sufficiently large, in $\Sigma_\lambda^\nu\cap K^c \cap supp(u^{\nu}_{\lambda}-u)^+$, we have $u \leq u^{\nu}_{\lambda}$ and the sharp asymptotic estimates \eqref{eq0806} and \eqref{eq0806+} hold. Furthermore, $|x| \approx |x^{\nu}_{\lambda}|$. Thus, similar to the proof of {(i)} and the estimate of $II_2$, we get
\begin{align}\label{eq1201}
&\quad \int_{\Sigma_\lambda^\nu} (|\nabla u|+|\nabla u^{\nu}_{\lambda}|)^{p-2} |\nabla(u^{\nu}_{\lambda}-u)^+|^2 \mathrm{d}x \\
&\leq  \frac{1}{C'} \int_{\Sigma_\lambda^\nu}  V(x) (\left(u^{\nu}_{\lambda}\right)^{p-1} - u^{p-1} ) (u^{\nu}_{\lambda}-u )^+ \mathrm{d}x \nonumber \\
&\quad + \frac{1}{C'} \int_{\Sigma_\lambda^\nu} (V(x^{\nu}_{\lambda})- V(x) )\left(u^{\nu}_{\lambda}\right)^{p-1}   (u^{\nu}_{\lambda}-u)^+ \mathrm{d}x \nonumber\\
&\leq \frac{p-1}{C'} \int_{\Sigma_\lambda^\nu\setminus K}  V(x) u^{p-2} [(u^{\nu}_{\lambda}-u )^+]^2 \mathrm{d}x  +  \frac{p-1}{C'} \int_{B_\lambda^\nu}  V(x) u^{p-2} [(u^{\nu}_{\lambda}-u )^+]^2 \mathrm{d}x \nonumber\\
&\quad+  \frac{p C \overline{C}^\frac{2}{p}}{C' S^\frac{2}{p}} \| u^{\nu}_{\lambda} \|^{2(p-1)}_{L^{p^{\star}}(\Sigma_\lambda^\nu \cap supp(u^{\nu}_{\lambda}-u)^+)}
      \int_{\Sigma_\lambda^\nu}(|\nabla u|+|\nabla u^{\nu}_{\lambda}|)^{p-2} |\nabla (u^{\nu}_{\lambda}-u)^+|^2 \mathrm{d}x, \nonumber
\end{align}
where $B_\lambda^\nu:=K\cap\Sigma_\lambda^\nu$. Similar to the estimates \eqref{eq1005}-\eqref{eq1006} on $II_1$ in the proof of {(i)}, we can obtain the following estimates in $\Sigma_\lambda^\nu\setminus K$:
\begin{align}\label{eq1202}
&\quad \frac{p-1}{C'} \int_{\Sigma_\lambda^\nu\setminus K}  V(x) u^{p-2} [(u^{\nu}_{\lambda}-u )^+]^2 \mathrm{d}x  \\
&\leq \frac{(p-1) C_{\star}  c_0^{p-2} }{C'} \int_{\Sigma_\lambda^\nu\setminus K}  \frac{1}{|x|^{p+\tau p}} \cdot \frac{1}{|x|^{\frac{N-p}{p-1}(p-2)}} [(u^{\nu}_{\lambda}-u )^+]^2 \mathrm{d}x  \nonumber\\
&\leq \frac{(p-1) C_{\star}  c_0^{p-2} }{C'} \sup_{\Sigma_\lambda^\nu\setminus K} \left(\frac{1}{|x|^{\tau p}}\right)
      \int_{\Sigma_\lambda^\nu\setminus K} |x|^{s-2} [(u^{\nu}_{\lambda}-u )^+]^2 \mathrm{d}x \nonumber\\
      &\leq \frac{(p-1) C_{\star}  c_0^{p-2} }{C'} \sup_{\Sigma_\lambda^\nu\setminus K} \left(\frac{1}{|x|^{\tau p}}\right)
      \left(\frac{2}{N+s-2}\right)^2 \int_{\Sigma_\lambda^\nu\setminus K} |x|^s |\nabla(u^{\nu}_{\lambda}-u )^+|^2 \mathrm{d}x \nonumber\\
&\leq  \frac{(p-1) C_{\star}  c_0^{p-2} }{C' (2C_0)^{p-2}} \frac{1}{R^{\tau p}}
           \left(\frac{2}{N+s-2}\right)^2  \int_{\Sigma_\lambda^\nu} (|\nabla u|+|\nabla u^{\nu}_{\lambda}|)^{p-2} |\nabla(u^{\nu}_{\lambda}-u )^+|^2 \mathrm{d}x, \nonumber
\end{align}
where $\tau$ is the same as in Lemma \ref{lm.5} and \eqref{V}.

Since $(u^{\nu}_{\lambda}-u )^+$ vanishes on $T_\lambda^\nu$ and belongs to $W^{1,p}_{loc}(\R^N)$ and $supp(u^{\nu}_{\lambda}-u)^+ \cap\Sigma^{\nu}_{\lambda}\cap B(P_\lambda^\nu,\overline{R})=A \cup B$, by the Poincar\'{e} type inequality in Lemma \ref{PTI}, we have
\begin{align}\label{eq1203}
&\quad \frac{p-1}{C'} \int_{B_\lambda^\nu}  V(x) u^{p-2} [(u^{\nu}_{\lambda}-u )^+]^2 \mathrm{d}x \\
&\leq \frac{(p-1)C_{2} }{C'\left[\inf\limits_K u\right]^{2-p}} \int_{B_\lambda^\nu} [(u^{\nu}_{\lambda}-u )^+]^2 \mathrm{d}x \nonumber\\
&\leq \frac{(p-1)C_{2} }{C'\left[\inf\limits_K u\right]^{2-p}} C |B_\lambda^\nu|^\frac{1}{N} \left( |A|^\frac{1}{N} \int_{A} |\nabla(u^{\nu}_{\lambda}-u )^+|^2 \mathrm{d}x + |B|^\frac{1}{N} \int_{B} |\nabla(u^{\nu}_{\lambda}-u )^+|^2 \mathrm{d}x \right)\nonumber\\
&\leq \frac{(p-1)C_{2} }{C'\left[\inf\limits_K u\right]^{2-p}} C |B_\lambda^\nu|^\frac{1}{N} \left( |A|^\frac{1}{N} M_K^{2-p} + |B_\lambda^\nu|^\frac{1}{N} M_B^{2-p}\right)  \nonumber\\
&\quad \times \int_{\Sigma_\lambda} (|\nabla u|+|\nabla u^{\nu}_{\lambda}|)^{p-2} |\nabla(u^{\nu}_{\lambda}-u )^+|^2 \mathrm{d}x, \nonumber
\end{align}
where $M_A =\sup\limits_A (|\nabla u|+|\nabla u^{\nu}_{\lambda}|)$. Combining this inequality with \eqref{eq1201}--\eqref{eq1203} and the fact that $\left(\Sigma^{\nu}_{\lambda}\cap supp(u^{\nu}_{\lambda}-u)^+\right)\subset (\R^N \setminus B(P_\lambda^\nu,\overline{R}))\cup (A \cup B)$, we get
\begin{align}\label{eq1204}
&\quad\int_{\Sigma^{\nu}_\lambda} (|\nabla u|+|\nabla u^{\nu}_{\lambda}|)^{p-2} |\nabla(u^{\nu}_{\lambda}-u)^+|^2 \mathrm{d}x \\
&\leq  \frac{(p-1) C_{\star}  c_0^{p-2} }{C' (2C_0)^{p-2}} \frac{1}{R^{\tau p}} \left(\frac{2}{N+s-2}\right)^2  \int_{\Sigma^{\nu}_\lambda} (|\nabla u|+|\nabla u^{\nu}_{\lambda}|)^{p-2} |\nabla(u^{\nu}_{\lambda}-u )^+|^2 \mathrm{d}x \nonumber\\
&\quad+ \frac{(p-1)C_{2} }{C'\left[\inf\limits_K u\right]^{2-p}} C |B_\lambda^\nu|^\frac{1}{N} \left( |A|^\frac{1}{N} M_K^{2-p} + |B_\lambda^\nu|^\frac{1}{N} M_B^{2-p}\right) \nonumber\\
&\quad\quad \times\int_{\Sigma^{\nu}_\lambda} (|\nabla u|+|\nabla u^{\nu}_{\lambda}|)^{p-2} |\nabla(u^{\nu}_{\lambda}-u )^+|^2 \mathrm{d}x  \nonumber\\
&\quad+  \frac{p C \overline{C}^\frac{2}{p}}{C' S^\frac{2}{p}} \| u^{\nu}_{\lambda} \|^{2(p-1)}_{L^{p^{\star}}(\R^N \setminus B(P_\lambda^\nu,R)) }
      \int_{\Sigma^{\nu}_\lambda}(|\nabla u|+|\nabla u^{\nu}_{\lambda}|)^{p-2} |\nabla (u^{\nu}_{\lambda}-u)^+|^2 \mathrm{d}x  \nonumber\\
&\quad+  \frac{p C \overline{C}^\frac{2}{p}}{C' S^\frac{2}{p}}
       \| u \|_\infty^{2(p-1)} |A\cup B|^{\frac{2(p-1)}{p^{\star}}}
      \int_{\Sigma^{\nu}_\lambda}(|\nabla u|+|\nabla u^{\nu}_{\lambda}|)^{p-2} |\nabla (u^{\nu}_{\lambda}-u)^+|^2 \mathrm{d}x. \nonumber
\end{align}
Let
$$G(R):=\frac{(p-1) C_{\star}  c_0^{p-2} }{C' (2C_0)^{p-2}} \frac{1}{R^{\tau p}} \left(\frac{2}{N+s-2}\right)^2 + \frac{p C \overline{C}^\frac{2}{p}}{C' S^\frac{2}{p}} \| u^{\nu}_{\lambda} \|^{2(p-1)}_{L^{p^{\star}}(\R^N \setminus B(P_\lambda^\nu,R)) }.$$
It is easy to see that $G(R)\to 0$ as $R\to\infty.$
Now, we can choose $R$ so that $G(\overline{R})\leq G(R) <\frac{1}{4}$ for each $\overline{R}\geq R$. Then it is enough to choose
$$ \delta:= \min\left\{ \frac{1}{|B_\lambda^\nu|} \left(\frac{C'\left[\inf\limits_K u\right]^{2-p}}{4(p-1)C_{2}C M_K^{2-p}}\right)^N , \left[\frac{C' S^\frac{2}{p}}{4 p C \overline{C}^\frac{2}{p} \| u \|_\infty^{2(p-1)}}\right]^{\frac{p^{\star}}{2(p-1)}} \right\}$$
and
$$ M^{2-p}:=\frac{C'\left[\inf\limits_K u\right]^{2-p}}{4(p-1)C_{2} C |B_\lambda^\nu|^\frac{2}{N} }.$$
With such choices of $\delta$ and $M$, it is easy to see that, whenever
$$ \Sigma^{\nu}_{\lambda}\cap supp(u^{\nu}_{\lambda}-u)^+\cap K=A\cup B,\quad\,\, |A\cap B|=0,\quad\,\, |A\cup B|<\delta,\quad\,\, M_B<M,$$
\eqref{eq1204} implies that $(u^{\nu}_{\lambda}-u)^+=0$ in $\Sigma_\lambda^\nu$, i.e., $u^{\nu}_{\lambda}\leq u$ in $\Sigma_\lambda^\nu$. This concludes our proof of Lemma \ref{lm.4.1}.
\end{proof}

\begin{rem}\label{re.4.1}
Observe that, for given $\lambda_0$ and $\nu_0$, $\overline{R}$ can be chosen so large such that, for all $\lambda$ in a small neighbourhood of $\lambda_0$, say, $[ \lambda_0-\theta_0 ,\lambda_0 + \theta_0 ]$ with $\theta_{0}>0$ small, and for all $\nu$ in $\overline{I_{\theta_0}(\nu_0)}:=\{ \nu \mid |\nu|=1,|\nu-\nu_0|\leq \theta_0 \}$, we have
$$ \overline{R}>R=R(\lambda,\nu)\,\quad \mbox{and}\,\quad B(P_\lambda^\nu,R)\subset B(P_{\lambda_0}^{\nu_0},\overline{R}), $$
where $R=R(\lambda,\nu)$ is given by Lemma \ref{lm.4.1} $(ii)$. Furthermore, $\delta$ and $M$ can also be chosen uniformly for all these $\lambda$ and $\nu$ in a $\theta_0$--neighbourhood of $(\lambda_0,\nu_0)$, by replacing $|B_\lambda^\nu|$ with $|B(P_{\lambda_0}^{\nu_0},\overline{R})|$ on the R.H.S of \eqref{eq1204}.
\end{rem}

\begin{lem}\label{lm.4.2}
Assume $u$ is a positive weak solution of \eqref{eq1.1} for $1<p<2$. Let $\nu_0$ be any unit vector in $\R^N$ and $\lambda_0\in\overline{\Lambda}(\nu_0)$. Let $\overline{R}$, $\delta$ and $M$ be given by Remark \ref{re.4.1} corresponding to $\theta_0-$neighbourhood of $(\lambda_0,\nu_0)$.

If we have
\begin{equation}\label{assump}
  u>u_{\lambda_0}^{\nu_0} \,\quad\,\,\mbox{in} \,\,\,(\Sigma_{\lambda_0}^{\nu_0}\setminus Z_{\lambda_0}^{\nu_0})  \cap  B(P_{\lambda_0}^{\nu_0},\overline{R}),
\end{equation}
then there exists $0<\theta_1 < \theta_0$ such that
$$ u\geq u^{\nu}_{\lambda} \,\quad \mbox{in}\,\, \Sigma_\lambda^\nu,\,\,\,\,\quad \forall\,\,(\lambda,\nu)\in (\lambda_0-\theta_1,\lambda_0+\theta_1)\times I_{\theta_1}(\nu_0),$$
where $I_{\theta_1}(\nu_0):=\{ \nu \,|\, |\nu|=1, |\nu-\nu_0|<\theta_1 \}$.
\end{lem}
\begin{proof}
Let $K:=B(P_{\lambda_0}^{\nu_0},\overline{R})$, where $P_{\lambda_0}^{\nu_0}$ is the projection of the origin on $T_{\lambda_0}^{\nu_0}$, and let $\delta$ and $M$ be chosen as in Remark \ref{re.4.1} uniformly for all $(\lambda,\nu)$ in a $\theta_0$-neighbourhood of $(\lambda_0,\nu_0)$.

Since $\overline{R}\geq R>R_{2}>2R_{1}>2R_{0}$, $\Sigma_\lambda^\nu \cap R_\lambda^\nu(B_{R_1}) \subset B(P_\lambda^\nu,R_2) \subset B(P_\lambda^\nu,R)\subset B(P_{\lambda_0}^{\nu_0},\overline{R})$ and $\Sigma_\lambda^\nu \cap R_\lambda^\nu(B_{R_1}) \neq \emptyset$ imply $B_{R_0}\subset B(P_{\lambda_0}^{\nu_{0}},\overline{R})$, by the sharp asymptotic estimates in Theorems \ref{th2.1} and \ref{th2.1-}, we get
$$Z_u\subset K.$$
On the other hand, by Remark \ref{re2334}, the Lebesgue measure of $Z_{u}$ is zero, i.e., $|Z_u|=0$.  Thus we can choose a small neighborhood $A$ of $\partial\left(K\cap \Sigma_{\lambda_0}^{\nu_0}\right)$ with $|A|<\frac{\delta}{4}$ and another neighbourhood $O$ of $Z_{\lambda_0}^{\nu_0} \cap K$ with $|O|<\frac{\delta}{4}$ and $M_O^{\lambda_0,\nu_0}<\frac{M}{2}$, where
$$ M_O^{\lambda,\nu}=\sup_O \{ |\nabla u|+|\nabla u^{\nu}_{\lambda}| \}.$$

Note that $K_0=(\overline{K\cap\Sigma_{\lambda_0}^{\nu_0}}) \setminus (A\cup O)$ is a compact set, so by our assumption \eqref{assump}, one has, there exists some positive constant $m>0$ such that
$$ u-u_{\lambda_0}^{\nu_0}>m>0 \quad\quad \,\mbox{in}\,\,\, K_0. $$
Then the continuity with respect to $\lambda$ and $\nu$ implies that, there exists $\theta_1=\theta_1(\lambda_0,\nu_0,K)<\theta_0$, such that
\begin{align*}
& u-u^{\nu}_{\lambda}> \frac{m}{2}>0 \qquad \,\mbox{in}\,K_0, \\
& M_{O\cap \Sigma_\lambda^\nu}^{\lambda,\nu}< M, \\
& |A\cap B(P_\lambda^\nu,R) \cap \Sigma_\lambda^\nu |< \frac{\delta}{2},\\
& |O\cap B(P_\lambda^\nu,R) \cap \Sigma_\lambda^\nu |< \frac{\delta}{2},\\
& B(P_\lambda^\nu,R)\cap \Sigma_\lambda^\nu\subset (B(P_{\lambda_0}^{\nu_0},\overline{R})\cap \Sigma_{\lambda_0}^{\nu_0}) \cup A
\end{align*}
for all $(\lambda,\nu)\in (\lambda_0-\theta_1, \lambda_0+\theta_1)\times I_{\theta_1}(\nu_0)$, where $R=R(\lambda_0,\nu_0)$ is given by Lemma \ref{lm.4.1} $(ii)$. Then we have
$$ supp(u^{\nu}_{\lambda}-u)^+ \cap B(P_\lambda^\nu,R)\cap \Sigma_\lambda^\nu \subset (A\cup O) \cap B(P_\lambda^\nu,R)  = (A\cap B(P_\lambda^\nu,R)) \cup (O\cap B(P_\lambda^\nu,R))$$
for all $(\lambda,\nu)\in (\lambda_0-\theta_1, \lambda_0+\theta_1)\times I_{\theta_1}(\nu_0)$, which implies that
$$ supp(u^{\nu}_{\lambda}-u)^+ \cap B(P_\lambda^\nu,R)\cap \Sigma_\lambda^\nu\subset(A\cap B(P_\lambda^\nu,R)\cap \Sigma_\lambda^\nu) \cup (O\cap B(P_\lambda^\nu,R)\cap \Sigma_\lambda^\nu) $$
satisfies the conditions of Lemma \ref{lm.4.1} $(ii)$. Therefore, we derive
$$ u\geq u^{\nu}_{\lambda} \quad \,\,\,\,\,\mbox{in}\,\, \Sigma_\lambda^\nu$$
for all $(\lambda,\nu)\in (\lambda_0-\theta_1,\lambda_0+\theta_1)\times I_{\theta_1}(\nu_0)$. This completes our proof of Lemma \ref{lm.4.2}.
\end{proof}

\begin{rem}\label{re.4.2}
Under the same assumptions as Lemma \ref{lm.4.2}, if, for some $\lambda_0 \in \overline{\Lambda}(\nu_0)$, there exist open neighbourhoods $A$ and $O$ of the sets $\partial\left(\Sigma^{\nu_0}_{\lambda_0} \cap B(P^{\nu_0}_{\lambda_0},\overline{R})\right)$ and $Z^{\nu_0}_{\lambda_0}\cap B(P^{\nu_0}_{\lambda_0},\overline{R})$ respectively such that
$$ u>u^{\nu_0}_{\lambda_0}\,\,\quad\, \mbox{in}\,\,\overline{\left(B(P^{\nu_0}_{\lambda_0},\overline{R})\cap \Sigma^{\nu_0}_{\lambda_0}\right)}\setminus (A\cup O) $$
with $|A|+|O|<\frac{\delta}{2}$ and $M_O^{\nu_0,\lambda_0}<\frac{M}{2}$, where $\delta$ and $M$ are chosen as in Remark \ref{re.4.1} uniformly for all  $(\lambda,\nu)\in (\lambda_0-\theta_0, \lambda_0+\theta_0)\times I_{\theta_0}(\nu_0)$. Then there exists a sufficiently small $\theta_1>0$ such that Lemma \ref{lm.4.2} holds.
\end{rem}

\subsubsection{Proof of the radial symmetry and strictly radial monotonicity}

In this subsection, we are to complete the proof of Theorem \ref{th2} in the singular elliptic case $1<p<2$.

\begin{proof}
Along any direction $\nu$ in $\R^N$, we start moving the planes $T^{\nu}_{\lambda}$ from $\lambda$ near $+\infty$ until its limiting position $T^{\nu}_{\lambda_{0}(\nu)}$.

\medskip

{\bf Step 1.} We will show that $\overline{\Lambda}(\nu)\neq\emptyset$ and $\lambda_0(\nu)$ is well-defined and finite.

We claim that $\overline{\Lambda}(\nu) \supset (R_1,+\infty)$, where $R_1=R_1(N,p,R_0)>0$ is given by Lemma \ref{lm.4.1} $(i)$. If $\lambda > R_1$,
$$ \Sigma_\lambda^\nu \cap R_\lambda^\nu(B_{R_1})=\emptyset.$$
Hence using Lemma \ref{lm.4.1} $(i)$, we have
$$ u\geq u_\lambda^\nu\,\quad\,\,\mbox{in}\,\, \Sigma_\lambda^\nu $$
for any $\lambda\in\overline{\Lambda}(\nu)$ and $\lambda > R_1$.

\medskip

{\bf Step 2.} We are to prove $u \equiv u^\nu_{\lambda_0(\nu)}$ in $\Sigma^\nu_{\lambda_0(\nu)}$.

At the limiting position $T^{\nu}_{\lambda_0(\nu)}$, by continuity, we have $u\geq u^\nu_{\lambda_0(\nu)}$ in $\Sigma^\nu_{\lambda_0(\nu)}$. Next, by the strong comparison principle in Lemma \ref{th3.2} and the fact that
\begin{align}\label{eq2305}
-\Delta_p u =V(x)u^{p-1} \geq V(x^{\nu}_{\lambda_0(\nu)})\left(u^{\nu}_{\lambda_0(\nu)}\right)^{p-1}= -\Delta_p u^{\nu}_{\lambda_0(\nu)} \quad\,\,\mbox{in}\,\,\Sigma^{\nu}_{\lambda_0(\nu)},
\end{align}
we deduce that, either $u>u^\nu_{\lambda_0(\nu)}$  or $u=u^\nu_{\lambda_0(\nu)}$ in all $\mathcal{C}^\nu$, where $\mathcal{C}^\nu$ is a connected component of $\Sigma^\nu_{\lambda_0(\nu)} \setminus Z^\nu_{\lambda_0(\nu)}$.

Noting that the sharp lower bound of $|\nabla u|$ in Theorems \ref{th2.1} and \ref{th2.1-} indicates that the critical set $Z_u=\{x \in \R^N \mid |\nabla u(x)| = 0 \}$ of the solution $u$ is a closed set belonging to $B_{R_0}(0)$, i.e.,
$$Z^\nu_{\lambda_0(\nu)}\subset Z_u \subset B_{R_0}.$$
Meanwhile, it follows from Corollary \ref{re2333} that $|Z_{u}|=0$ (see Remark \ref{re2334}), and from Lemma \ref{lm.7} that $\Omega\setminus Z_u$ is connected for any smooth bounded domain $\Omega\subset\mathbb{R}^{N}$ with connected boundary such that $B_{R_0}(0)\subseteq\Omega$. As a consequence, one has $\mathbb{R}^{N}\setminus Z_u$ is connected. Due to the closeness of $Z_{u}$ and the fact that $|Z_{u}|=0$, we can deduce that $\mathbb{R}^{N}\setminus Z^\nu_{\lambda_{0}(\nu)}$ is connected.

\smallskip

If $\Sigma^\nu_{\lambda_0(\nu)} \setminus Z^\nu_{\lambda_0(\nu)}$ has at least two different connected components $\mathcal{C}^{\nu}_{1}$ and $\mathcal{C}^{\nu}_{2}$ such that $u=u^\nu_{\lambda_0(\nu)}$ in $\mathcal{C}^{\nu}_{1}$ and $u>u^\nu_{\lambda_0(\nu)}$ in $\mathcal{C}^{\nu}_{2}$. Then, by symmetry, $\mathcal{C}^{\nu}_{1}\bigcup R^{\nu}_{\lambda_{0}(\nu)}(\mathcal{C}^{\nu}_{1})\subsetneq \mathbb{R}^{N}\setminus Z^\nu_{\lambda_{0}(\nu)}$ contains at least one connected component of $\mathbb{R}^{N}\setminus Z^\nu_{\lambda_{0}(\nu)}$, which is absurd. If we assume that
\begin{align}\label{eq2305009}
u > u^\nu_{\lambda_0(\nu)} \quad\,\,\,\mbox{in}\,\,\Sigma^\nu_{\lambda_0(\nu)} \setminus Z^\nu_{\lambda_0(\nu)},
\end{align}
then Lemma \ref{lm.4.2} implies that $u\geq u_\lambda^\nu$ in $\Sigma_\lambda^\nu$ for any $\lambda\in(\lambda_0(\nu)-\theta_1, \lambda_0(\nu))$ with $\theta_{1}>0$ small, which contradicts the definition of $\lambda_0(\nu)$. As a consequence, we have proved that
$$ u\equiv u^\nu_{\lambda_0(\nu)} \,\,\quad \mbox{in}\,\, \Sigma^\nu_{\lambda_0(\nu)} \setminus Z^\nu_{\lambda_0(\nu)}.$$
Furthermore, by the definition of $Z^\nu_{\lambda_0(\nu)}$ and $u\in C_{loc}^{1,\alpha}(\R^N)$, it follows that
\begin{align}\label{eq23050091}
u\equiv u^\nu_{\lambda_0(\nu)} \quad\,\,\mbox{in} \,\,\Sigma^\nu_{\lambda_0(\nu)}.
\end{align}

On the other hand, according to the strong comparison principle Lemma \ref{th3.2}, one knows that
$ u > u^\nu_{\lambda}$ in $\Sigma^\nu_\lambda\setminus Z^\nu_{\lambda(\nu)}$ for any $\lambda > \lambda_0(\nu)$.    
Consequently, $u$ is strictly monotone increasing in $\Sigma^\nu_{\lambda_0(\nu)}\setminus Z^\nu_{\lambda_0(\nu)}$ for any direction $\nu$, and hence $\frac{\partial u}{\partial \nu} \geq 0$ in $\Sigma^\nu_{\lambda_0(\nu)} \setminus Z^\nu_{\lambda_0(\nu)}$.
%

\medskip

{\bf Step 3.} We will show $u$ is radial and (strictly) radially decreasing about some point $x_0\in\R^N$.

It is known that for any direction $\nu$, $u$ is symmetric about the plane $T^\nu_{\lambda_0(\nu)}$. Now consider all the $N$ orthogonal directions $\{e_k\}_{k=1}^N$ in $\R^N$, and one can obtain in entirely similar way that $u$ is strictly increasing in each direction $e_k$ in $\Sigma^{e_k}_{\lambda_0(e_k)}$ and symmetric about the plane $T^{e_k}_{\lambda_0(e_k)}$. This shows that $u$ is symmetric with respect to a point $x_0\in\R^N$ $(x_0 = \bigcap^N_{k=1}T^{e_k}_{\lambda_0(e_k)})$, which is the unique critical point for $u$. Furthermore, by exploiting the moving plane procedure in any direction $\nu\in S^{N-1}$, we finally get that $u$ is radial and (strictly) radially decreasing w.r.t. $x_{0}$. That is, positive solution $u$ must assume the form
\[u(x)=\lambda^{\frac{N-p}{p}}U(\lambda(x-x_{0}))\]
for $\lambda:=u(0)^{\frac{p}{N-p}}>0$ and some $x_{0}\in\mathbb{R}^{N}$, where $U(x)=U(r)$ with $r=|x|$ is the positive radial solution to \eqref{eq1.1} with $U(0)=1$, $U'(0)=0$ and $U'(r)\leq0$ for any $r>0$.

\smallskip

Next, for $1<p<2$, we will prove that
\begin{equation}\label{final}
  U^\prime (r) < 0, \quad \,\,\,\forall \,\, r>0.
\end{equation}
We will prove \eqref{final} by contradiction arguments. We assume on the contrary that there exists $r_0>0$ such that $U^\prime (r_0) = 0$. Since Lemma \ref{re2333} and Remark \ref{re2334} implies that $|Z_U|=0$, one has $U'(r)<0$ for $r\in(0,r_{0})\cup(r_{0},+\infty)$, so $U(r)>U(r_0)$ in $[0,r_0)$. Using the Strong Maximum Principle and H\"{o}pf's Lemma in Lemma \ref{hopf} (c.f. \cite{JLV}, see also Theorem 2.1 in \cite{LDBS04}, Theorem 2.4 in \cite{LDSMLMSB} or \cite{LD,PS}) to the positive radial solution $W(x):=U(x)-U(r_0)=U(|x|)-U(r_{0})$ of
\begin{align*}
\left\{ \begin{array}{ll} \displaystyle
-\Delta_p W  = \int_{\R^N} \frac{U^p(|y|)}{|x-y|^{2p}}\mathrm{d}y \cdot U^{p-1}(x) > 0  \,\,\,\,&\mbox{in}\, \,B_{r_0}, \\
W>0   &\mbox{in}\,\,  B_{r_0},\\
W=0   &\mbox{on}\,\, \partial B_{r_0},
\end{array}
\right.\hspace{1cm}
\end{align*}
we obtain $W^\prime (r_0)=U'(r_{0})<0$, which is absurd. Hence \eqref{final} holds, i.e., $U'(r)<0$ for any $r>0$. This completes our proof of Theorem \ref{th2} in the singular elliptic case $1<p<2$.

This concludes our proof of Theorem \ref{th2}.
\end{proof}

\section{Proof of Theorem \ref{gth2}}\label{sc5}

\begin{proof}[Proof of Theorem \ref{gth2}]

We will first show that $V(x):=\left(|x|^{-\sigma}\ast u^{p_{\sigma}^{\star}}\right)(x)u^{p_{\sigma}^{\star}-p}(x)$ satisfies all the assumptions in Theorem \ref{th2.1}, and hence derive the regularity results and the sharp asymptotic estimates on $D^{1,p}(\mathbb{R}^{N})$-weak solution $u$ to the general nonlocal quasi-linear equation \eqref{gnqe} and $|\nabla u|$ from Theorem \ref{th2.1}.

\medskip

In fact, note that $0\leq V(x):=\left(|x|^{-\sigma}\ast u^{p_{\sigma}^{\star}}\right)(x)u^{p_{\sigma}^{\star}-p}(x)\in L^{\frac{N}{p}}(\mathbb{R}^{N})$, by Theorem \ref{th2.1}, we know that $u\in L^{r}_{loc}(\mathbb{R}^{N})$ for any $0<r<+\infty$ and the decay property in \eqref{a1} holds for sufficiently large $\bar{p}\geq \hat{p}$ and $R>R_{0}>1$. First, suppose $u\in L^{r}_{loc}(\mathbb{R}^{N})$ for any $0<r<+\infty$, then we can deduce from the Hardy-Littlewood-Sobolev inequality and the decay property in \eqref{a1} that, for any $q>\frac{N}{p}$ large enough, there exists $\bar{p}\geq\hat{p}\geq p^{\star}$ such that, for any $R>R_{0}$,
\begin{eqnarray}\label{1g}
  && \|V\|_{L^{q}(B_{R}(0))}\leq C\|u\|^{2p^{\star}_{\sigma}-p}_{L^{\bar{p}}(\mathbb{R^{N}})}\leq C\left\{\|u\|^{2p^{\star}_{\sigma}-p}_{L^{\bar{p}}(B_{2R}(0))}+\|u\|^{2p^{\star}_{\sigma}-p}_{L^{\bar{p}}(\mathbb{R}^{N}\setminus B_{2R}(0))}\right\} \\
 \nonumber  && \qquad\qquad\quad\,\,\, \leq C\left\{\|u\|^{2p^{\star}_{\sigma}-p}_{L^{\bar{p}}(B_{2R}(0))}+\left(\frac{C}{R^{\frac{N-p}{p}+\tau}}\right)^{2p^{\star}_{\sigma}-p}\right\}<+\infty,
\end{eqnarray}
that is, $V\in L^{q}_{loc}(\mathbb{R}^{N})$ for any $\frac{N}{p}<q<+\infty$.

\medskip

Second, assume that the decay property \eqref{a1} holds for sufficiently large $\bar{p}\geq \hat{p}$ and $R>R_{0}>1$, then by the Hardy-Littlewood-Sobolev inequality and the H\"{o}lder's inequality, we have, for any given $q>\frac{N}{p}$ large enough and $\widetilde{R}=4R_{0}$,
\begin{eqnarray}\label{2g}
  &&\quad \|V\|_{L^{q}(B_{1}(x))}\\
  \nonumber && \leq \Big\|\left[|x|^{-\sigma}\ast (u^{p^{\star}_{\sigma}}\chi_{\mathbb{R}^{N}\setminus B_{4R_{0}}(0)})\right]u^{p^{\star}_{\sigma}-p}\Big\|_{L^{q}(B_{1}(x))}+\Big\|\left[|x|^{-\sigma}\ast (u^{p^{\star}_{\sigma}}\chi_{B_{4R_{0}}(0)})\right]u^{p^{\star}_{\sigma}-p}\Big\|_{L^{q}(B_{1}(x))} \\
 \nonumber && \leq C\|u\|^{2p^{\star}_{\sigma}-p}_{L^{\bar{p}}(\mathbb{R}^{N}\setminus B_{4R_{0}}(0))}+\left(|B_{1}(0)|\right)^{\frac{p^{\star}_{\sigma}}{\bar{p}}-\frac{N-\sigma}{N}}R_{0}^{-\sigma}
 \|u\|^{p^{\star}_{\sigma}}_{L^{p^{\star}_{\sigma}}(B_{4R_{0}}(0))}\|u\|^{p^{\star}_{\sigma}-p}_{L^{\bar{p}}(\mathbb{R}^{N}\setminus B_{4R_{0}}(0))} \\
 \nonumber  && \leq C\left\{\left(\frac{C}{(2R_{0})^{\frac{N-p}{p}+\tau}}\right)^{2p^{\star}_{\sigma}-p}
 +R_{0}^{-\frac{\sigma}{2}}\left(\frac{C}{(2R_{0})^{\frac{N-p}{p}+\tau}}\right)^{p^{\star}_{\sigma}-p}\|u\|^{p^{\star}_{\sigma}}_{L^{p^{\star}}(\mathbb{R}^{N})}\right\} \\
 \nonumber  && \leq CR_{0}^{-p}\leq C, \qquad\quad \forall\,\, |x|>2\widetilde{R},
\end{eqnarray}
where the constant $C>0$ is independent of $x$ and $\widetilde{R}$. That is, the uniform boundedness in \eqref{a2} holds.

\medskip

Finally, assume that $u(x)\leq C|x|^{-\gamma}$ for some $C>0$, $\gamma>\frac{N-p}{p}$ and $|x|$ large (actually, the better preliminary estimate $u(x)\leq\frac{C}{1+|x|^{\frac{N-p}{p}+\tau}}$ with $\tau>0$ in Lemma \ref{lm.5} indicates that this assumption holds for $\gamma=\frac{N-p}{p}+\tau$), then the basic estimates on the convolution $V(x)=\left(|x|^{-\sigma}\ast u^{p_{\sigma}^{\star}}\right)(x)u^{p_{\sigma}^{\star}-p}(x)$ in (ii) of Lemma \ref{lm.6} (with $0<\sigma<N$, $\sigma\leq 2p$ and $q=p_{\sigma}^{\star}$) implies that, for $|x|>R$ large,
\begin{equation}\label{Vg}
  V(x)=\left(|x|^{-\sigma}\ast u^{p_{\sigma}^{\star}}\right)(x)u^{p_{\sigma}^{\star}-p}(x)\leq C|x|^{-(p _{\sigma}^{\star}\gamma+\sigma-N)}|x|^{-(p_{\sigma}^{\star}-p)\gamma}=:C|x|^{-\beta},
\end{equation}
where $\beta:=p_{\sigma}^{\star}\gamma+\sigma-N+(p_{\sigma}^{\star}-p)\gamma=p+(2p_{\sigma}^{\star}-p)\tau>p$. Thus Theorem \ref{th2.1} can be applied to the general nonlocal quasi-linear equation \eqref{gnqe} and hence positive weak solution $u$ to \eqref{gnqe} satisfies the regularity results in Theorem \ref{th2.1}, and the sharp asymptotic estimates
$$u\sim\left(1+|x|^\frac{N-p}{p-1}\right)^{-1} \quad \text{in} \,\, \mathbb{R}^{N} \quad\,\,\, \text{and} \quad\,\,\, |\nabla u|\sim |x|^{-\frac{N-1}{p-1}} \quad \text{in} \,\, \mathbb{R}^{N}\setminus B_{R_{0}}(0).$$

\medskip

Thanks to the sharp asymptotic estimates on $u$ and $|\nabla u|$, the radial symmetry and strictly radial monotonicity in Theorem \ref{gth2} can be proved via the method of moving planes in entirely the same way as in the proof of Theorem \ref{th2} in Section \ref{sc4} by replacing $V(x)=|x|^{-2p}\ast u^{p}$ with $V(x)=\left(|x|^{-\sigma}\ast u^{p_{\sigma}^{\star}}\right)(x)u^{p_{\sigma}^{\star}-p}(x)$ and making slightly suitable modifications accordingly. For instance, we only need to use the decay estimate \eqref{Vg} instead of \eqref{V}, use Lemma \ref{lm.6} and Remark \ref{rm.2.7} with general $\sigma\in(0,N)\cap(0,2p]$ and $q=p^{\star}_{\sigma}\in(\frac{p^{\star}}{2},p^{\star})\cap[p,p^{\star})$ instead of the endpoint case $\sigma=2p$ and $q=p$, and use the following identity
\begin{align*}
V(x)-V(x^{\nu}_\lambda)&=\int_{\R^N} \frac{u^{p_{\sigma}^{\star}}(y)}{|x-y|^{\sigma}}\mathrm{d}y\left[u^{p_{\sigma}^{\star}-p}(x)-(u^{\nu}_\lambda)^{p_{\sigma}^{\star}-p}(x)\right] \\
&\quad+\int_{\Sigma^{\nu}_\lambda}\left(\frac{1}{|x-y|^{\sigma}}-\frac{1}{|x^{\nu}_\lambda-y|^{\sigma}}\right)\left[u^{p_{\sigma}^{\star}}(y)-(u^{\nu}_\lambda)^{p_{\sigma}^{\star}}(y)\right]\mathrm{d}y \cdot(u^{\nu}_\lambda)^{p_{\sigma}^{\star}-p}(x) \nonumber
\end{align*}
instead of \eqref{eq0805} in Lemma \ref{lm.1}, where $\nu$ is any direction in $\R^N$, and so on. So we omit the details in the rest of the proof for the sake of simplicity. This concludes our proof of Theorem \ref{gth2}.
\end{proof}

\end{document}